\documentclass[aos]{imsart}

\RequirePackage{amsthm,amsmath,amsfonts,amssymb,calrsfs,xr,enumerate,enumitem}

\usepackage[longnamesfirst]{natbib}

\usepackage{caption}
\usepackage{subcaption}
\usepackage{bbm}
\usepackage{tcolorbox}
\RequirePackage[breaklinks,colorlinks,citecolor=blue,urlcolor=blue]{hyperref}
\usepackage{mathtools}
\usepackage{multirow}
\usepackage{color}
\usepackage{cancel}
\usepackage[normalem]{ulem}

\startlocaldefs
\newcommand{\vertiii}[1]{{\left\vert\kern-0.25ex\left\vert\kern-0.25ex\left\vert #1 \right\vert\kern-0.25ex\right\vert\kern-0.25ex\right\vert}}
\numberwithin{equation}{section}
\theoremstyle{plain}
\newtheorem{definition}{Definition}[section]
\newtheorem{example}{Example}[section]
\newtheorem{thm}{Theorem}[section]
\newtheorem{cor}[thm]{Corollary}
\newtheorem{lemma}[thm]{Lemma}
\newtheorem{prop}[thm]{Proposition}

\newtheorem{remark}{Remark}[section]

\def\pa{\mathbbm{p}}
\def\ii{\mathbbm{i}}
\def\C{\mathbb C}
\def\R{\mathbb R}
\def\H{\mathbb H}

\def\what{\widehat}
\def\E{\mathbb E}
\def\Z{\mathbb Z}
\def\tcr{\textcolor{red}}
\def\CT{\mathcal T}
\def\bbH{\mathbb H}
\def\bbX{\mathbb X}
\def\bbT{\mathbb T}
\def\bbR{\mathbb R}
\def\tr{\mathrm{tr}}

\def\cov{\mathrm{Cov}}
\def\wt{\widetilde}
\def\wtilde{\widetilde}
\def\P{\mathbb P}
\DeclareMathOperator*{\argmin}{argmin}

\def\ol{\overline}

\endlocaldefs

\newcommand{\pr}{\mathbb{P}}
\newcommand{\HS}{{\rm HS}}
\newcommand{\cid}{\stackrel{d}{\to}}
\newcommand{\bbC}{\mathbb{C}}

\renewcommand\comment[1]{ \fbox{\tcr{Commented out material.}}}

\renewcommand{\theremark}{\arabic{section}.\arabic{remark}}

\newcounter{assumption}
\newenvironment{assumption}[1][]{\refstepcounter{assumption}\par\medskip
   \textbf{Assumption~\theassumption #1.} \rmfamily}{\medskip}

\usepackage{xr}
\makeatletter
\newcommand*{\addFileDependency}[1]{
  \typeout{(#1)}
  \@addtofilelist{#1}
  \IfFileExists{#1}{}{\typeout{No file #1.}}
}
\makeatother

\newcommand*{\myexternaldocument}[1]{
    \externaldocument{#1}
    \addFileDependency{#1.tex}
    \addFileDependency{#1.aux}
}

\myexternaldocument{supplement_material}

\allowdisplaybreaks

\begin{document}

\begin{frontmatter}
\title{Spectral Density Estimation of Function-Valued Spatial Processes}
\runtitle{Spectral Density Estimation for Functional Data}

\begin{aug}
\author[A]{\fnms{Rafail} \snm{Kartsioukas}\ead[label=e1]{rkarts@umich.edu}},
\author[B]{\fnms{Stilian} \snm{Stoev}\ead[label=e2,mark]{sstoev@umich.edu}}
\and
\author[B]{\fnms{Tailen} \snm{Hsing}\ead[label=e3,mark]{thsing@umich.edu}}
\address[A]{\printead{e1}}

\address[B]{\printead{e2,e3}}
\end{aug}

\begin{abstract}
The spectral density function describes the second-order properties of a stationary stochastic process on $\R^d$. This paper considers the nonparametric estimation of the spectral density of a continuous-time stochastic process taking values in a separable Hilbert space. Our estimator is based on kernel smoothing and can be applied to a wide variety of spatial sampling schemes including those in which data are observed at irregular spatial locations. Thus, it finds immediate applications in Spatial Statistics, where irregularly sampled data naturally arise. The rates for the bias and variance of the estimator are obtained under general conditions in a mixed-domain asymptotic setting. When the data are observed on a regular grid, the optimal rate of the estimator matches the minimax rate for the class of covariance functions that decay according to a power law. The asymptotic normality of the spectral density estimator is also established under general conditions for Gaussian Hilbert-space valued processes. Finally, with a view towards practical applications the asymptotic results are specialized to the case of discretely-sampled functional data in a reproducing kernel Hilbert space.
\end{abstract}



\end{frontmatter}

\section{Introduction}\label{sec:intro}

Historically, the study of signals, such as  electromagnetic or acoustic waves, in physics naturally led to the investigation of the spectral density. 
The current literature on the inference problem of the spectral density contains an abundance of well-established estimators and algorithms \citep[see, e.g.,][and the references therein]{hannan:1970,brillinger:1981,brockwell:davis:2006,percival:walden:2020}. 
The most classical approach is based on the periodogram \citep{schuster1898investigation}, which is at the core of the majority of the procedures that are known today.
However, alternative approaches that involve, for instance, the inversion of the empirical covariance
\cite[see, e.g., the review paper of][]{robinson1983review} and wavelets \citep[][]{percival:walden:2006,bardet:bertrand:2012} have also been extensively considered.

The traditional statistical research on spectral density estimation considers scalar-valued processes.  Modern scientific applications involve, however, high-dimensional or even function-type data, which are typically indexed by space and/or time.  Recently, there has been a growing interest in functional time series in general, where data are observed at times $1,2,\ldots,T$; see \cite{hormann2012functional}, \cite{panaretos2013fourier}, \cite{horvath2014testing}, \cite{li2020long}, \cite{zhu2020higher}, to mention a few. In particular, \cite{panaretos2013fourier} and \cite{zhu2020higher} both address the inference of the spectral density of functional time series.
\cite{panaretos2013fourier} considers the smoothed periodogram estimator where the notion of periodogram kernel is introduced for functional data taking values in $L^2[0,1]$. \cite{zhu2020higher} considers the same estimator, but focuses on a particular type of kernel, called flat-top kernel, in performing nonparametric smoothing. 

This paper studies the nonparametric estimation of the spectral density for a continuous-time stationary process $X=\{X(t),t\in\mathbb{R}^d\}$ taking values in some Hilbert space $\bbH$. More information will be given in Section \ref{sec:background} regarding $\bbH$ and the definition of second-order stationarity. One of the novelties of our paper is the consideration of functional data $X(t)$ sampled at irregular spatial locations $t_1,\ldots,t_n\in\bbR^d$ as opposed to at regular grid points, e.g., $t=1,2,\ldots$ as in functional time series. In general, spatial data are not gridded data. An excellent example is provided by the Argo dataset which has recently become an important resource for oceanography and climate research \citep[cf.][]{roemmich_135_2012} and has inspired new approaches in spatial statistics \citep[see, e.g.,][]{kuusela2017locally, yarger:stoev:hsing:2022}. 

For spatial data observed at irregular locations, periodogram-based approaches do not easily generalize. 
We consider in this paper a so-called lag-window estimator \citep[cf.][]{brockwell:davis:2006,zhu2020higher} based on estimating the covariance, which can accommodate rather general observational schemes. The performance of the estimator will be evaluated by asymptotic theory. In doing so, we will assume the framework of the so-called mixed-domain asymptotics, which means that the sampling locations become increasingly dense and the sampling region becomes increasingly large as the number of observations increases; see, e.g., \cite{hall1994properties, fazekas2000infill, matsuda2009fourier, chang2017mixed,  maitra2020increasing}.
The rate bound of the mean squared error of our estimator will be developed for a rather general mixed-domain setting. However, when data are observed on a regular grid assuming a specific covariance model, the rate bound calculations can be made precise, paving the way for assessing the optimality of the estimator. In particular, we establish the rate minimaxity of our estimator based on gridded data if the decay of the covariance function is bounded by a power law.

We now provide a summary of each of the sections below.  In Section \ref{sec:background}, we describe the general notion of second-order stationarity for a process taking values in a complex Hilbert space $\bbH$. Despite the prevalence of multidimensional spatial data, this notion is understood much less well than the corresponding notion in the one-dimensional case. In particular, we will explain the subtlety of why the scalar field of $\bbH$ must be taken as complex in order to conduct the spectral analysis of the process. We will also review Bochner's Theorem which facilitates the definition of spectral density. Section \ref{sec:estimator} introduces the key assumptions and defines the lag-window estimator that is the main focus of the paper. In Section \ref{s:asymp}, we establish upper bounds on the rate of decay of the bias and variance, and hence the mean squared error of the spectral density estimator under general conditions. These rates are made more precise  in Section \ref{sec:discrete-time} for the setting of gridded data, where the grid size either stays fixed or shrinks to zero with sample size and we focus on a class of covariance functions that are dominated by a power law. By comparing with these carefully computed rates, we show in Section \ref{sec:minimax} that our spectral density estimator is minimax rate optimal for these models.  In Section \ref{s:CLT}, we establish the asymptotic distribution of the estimator under Gaussianity. The proof is based on a novel Isserlis type formula which is used to compute all moments of the estimator. In Section \ref{s:rkhs}, we consider the issue of incomplete functional data in the reproducing kernel Hilbert space (RKHS) setting. Finally, Section \ref{sec:comparisons} briefly summarizes the results in \cite{panaretos2013fourier} and \cite{zhu2020higher}, and provides some comparisons with the ones in this paper.

Whenever feasible, we will provide an outlined proof immediately after stating a result. However, all the detailed proofs are included in the supplement.

\section{Covariance and spectral density of a stationary process in a Hilbert space}\label{sec:background}
Throughout this paper, let $\bbH$ be a separable Hilbert space over the field of complex numbers $\mathbb{C}$. Common examples of $\bbH$ in functional-data applications include $L^2$ spaces of functions and RKHS's. However, except in Section \ref{s:rkhs}, no additional assumptions will be made on $\bbH$.

The inner product and norm of $\bbH$ are denoted by $\langle\cdot,\cdot\rangle$ and $\|\cdot\|$, respectively. In a small number of instances, we will denote these by $\langle\cdot,\cdot\rangle_\bbH$ and $\|\cdot\|_\bbH$ for clarity.
The main purpose of this section is to recall some fundamental results for the spectral analysis of stochastic processes $X=\{X(t),t\in\mathbb{R}^d\}$ taking values in $\bbH$.

\subsection{Second-order stationary}

We first address the notion of second-order stationarity or covariance stationarity for a process taking values in a complex Hilbert space.
We begin by considering a zero-mean Gaussian process $X$ with $\bbH=\bbC$. Let $\Re(X(t))$ and $\Im(X(t))$ denote the real and imaginary parts of $X(t)$, respectively.
Recall that $X$ is strictly stationary if and only if the two-dimensional real Gaussian process 
$$
Y(t) = (X_R(t), X_I(t))^T := (\Re(X(t)), \Im(X(t)))^\top \in\bbR^2, 
$$
is second-order stationary, i.e., the covariance function $C_Y(t,s) := \E\left[Y(t)Y(s)^\top\right]$ is a function of $t-s$. Let
$$
C(t,s) = \E[X(t) \overline X(s)] \quad\mbox{and}\quad \check C(t,s) = \E[X(t)X(s)].
$$
It follows that
\begin{align*}
C(t,s) & = \E[X_R(t)X_R(s)] + \E[X_I(t)X_I(s)] -\ii \big(\E[X_R(t)X_I(s)]-\E[X_I(t)X_R(s)]\big), \\
\check C(t,s) & = \E[X_R(t)X_R(s)] - \E[X_I(t)X_I(s)] + \ii \big(\E[X_R(t)X_I(s)] + \E[X_I(t)X_R(s)]\big).
\end{align*}
Observe that $\{C(t,s), \check C(t,s), t, s \in \bbR^d\}$ contains the same information as that in $\{C_Y(t,s)$, $t,s\in\bbR^d\}$. 
In particular, $Y$ is second-order stationary if and only if both $C(t,s)$ and $\check C(t,s)$ are functions of $t-s$. 
The functions, $C(t,s)$ and $\check C(t,s)$, are commonly referred to as the covariance function and pseudo-covariance function, respectively,
which are equal if and only if $X$ is real valued. 
Going beyond the Gaussian setting, we shall take this as the definition of second-order stationarity for a general complex-valued process $X$ with finite second moments, where the inference on the covariance of $X$ can be conducted on $C_Y$ or $C$ and $\check C$ combined. While the stationary covariance $C$ is positive definite, which provides a basis for inference in the spectral domain, it is not the case for $\check C$. Thus, the spectral inference on $X$ must be carried out on the real process $Y$ unless $X$ itself is real, in which case we can simply focus on $C$. The discussion above extends in a straightforward manner to the finite-dimensional case 
$\bbH = \bbC^p$ for any finite $p$, for which the outer-product is $x \overline y^\top, x, y \in \bbC^p$ \citep[cf.][]
{hannan:1970, brillinger:1981, tsay2013multivariate}.


If $\bbH$ is an infinite-dimensional Hilbert space over $\bbC$, then the cross-product (or outer-product) of $x,y\in \mathbb{H}$  is the linear operator defined as 
\begin{align} \label{e:outer_product}
  [x\otimes y](z) = x\cdot {\langle z,y\rangle}, \ \ z\in\mathbb{H},
\end{align}
and, provided that $\E[\|X(t)\|^2] < \infty$ for all $t$, we can define the covariance operator of $X$ as
\begin{align} \label{e:CXY}
C(t,s):= \E[ X(t)\otimes X(s)].
\end{align}
Note that $C(t,s)$ takes values in the space of trace-class operators $\bbT$ and is well-defined in the sense of Bochner in the Banach space $(\bbT,\|\cdot\|_{\rm tr})$.
More information on $\bbT$ will be given below in Section \ref{ss:Bochner}.
However, the discussion on stationarity for the finite-dimensional case and especially the notion of pseudo-covariance
requires modification since an immediate notion of ``complex conjugate'' does not exist. Following \cite{shen2020tangent}, we fix a 
complete orthonormal system (CONS) $\{e_j\}$ of $\H$ and refer to it as the \emph{real} CONS.  Then, for each $x\in \H$ such that $x = \sum_{j} \langle x,e_j\rangle e_j,$ define the complex conjugate 
${\rm conj}(x)\equiv \overline{x}$ as
$$
\overline{x} := \sum_{j} \overline{\langle x,e_j\rangle} e_j.
$$
Thus, ${\rm conj} :\H\mapsto \H$ is an anti-linear operator, i.e., 
$$
{\rm conj}(\alpha x +\beta y) = \overline{\alpha} {\rm conj}(x) + \overline{\beta}{\rm conj}(y), \ x,y\in\H,\ \alpha,\beta\in\C.
$$   
Also, for $x\in \H$, define its real and imaginary parts:
$$
\Re(x) := \frac{\ol{x}+x}{2},\ \Im(x) := \frac{x-\ol{x}}{2\ii}.
$$
This construction allows us to view the complex Hilbert space $\H$ as 
\begin{align} \label{e:H_decomp}
\H=\H_{\bbR} + \ii \H_{\bbR},
\end{align} 
where $\H_\R:=\{ x\in \H\, :\, \Im(x) = 0\}$  is the real Hilbert space of real elements of $\H$ \citep[see, e.g.,][]{cerovecki2017clt}.
Consequently,  $x$ will be called real if $x\in \H_\R$.

Define the pseudo-covariance operator for a second-order process $\{X(t)\}$ as
\begin{equation}\label{e:CXY_2}
\check C(t,s):= \E[X(t)\otimes \overline X(s)].
\end{equation}

The definition of second-order stationarity for a process in $\bbH$ can now be stated as follows.

\begin{definition} \label{def:covariance operator} 
A zero-mean stochastic process $X= \{X(t),\ t\in \R^d\}$ taking values in 
 $\H$ is said to be an $L^2$- or second-order process if $\E[ \|X(t)\|^2] <\infty$. 
The process $X$ will be referred to as second-order stationary or covariance stationary if both $C(t,s)$ and $\check C(t,s)$ depend only on the lag $t-s$.  In this case, 
we write $C(h) := C(t+h,t)$ and $\check C(h) := \check C(t+h,t)$, which are referred to as the stationary covariance operator and stationary pseudo-covariance operator, respectively.
\end{definition}

It is important to note that while the definition of $\check C(t,s)$ depends on the designated real basis, whether $\check C(t,s)$ is a function of the lag is basis independent; this can be seen using a change-of-basis formula.

We end this section with the following two remarks.

\begin{remark} \label{r:2.1}
As in the one-dimensional case, one can equivalently define stationarity in terms of the real process
$Y(t):= (\Re(X(t))$, $\Im(X(t)))$ taking values in the product Hilbert space $\H_{\bbR}\times \H_{\bbR}$ over $\bbR$. It follows that 
$X$ is second-order stationary if and only if $Y$ is. For much of the rest of this paper, we shall assume for simplicity that the process $X$ is real (based on some CONS), i.e., it takes values in $\H_{\bbR}\subset \H$ (cf. \eqref{e:H_decomp}), in which case, $C(h) = \check C(h)$. 
This simplification does not lead to less generality since all the results apply to $Y$. 
Two exceptions are Section \ref{ss:Bochner} and Section \ref{s:CLT} where we present more general results by considering a complex $X$.
 \end{remark}
 
 \begin{remark} \label{r:2.2}
In view of the last remark, a careful reader might wonder why we choose to work with the framework of complex Hilbert space in the first place. An important reason for that is because the spectral density of a process $X$, real or complex, in $\H$ will in general take values in $\bbT_+$, the space of positive trace-class operators over the {\em complex} Hilbert space $\H$.  To demonstrate the point, consider the following simple example.
Let $\{Z(t), t\in\bbR\}$ be a real, scalar-valued zero-mean Gaussian process with auto-covariance $\gamma(t) = \E[Z(t)Z(0)]$.  Let $a>0$, and define
 $
 X_a(t) := (Z(t),\ Z(t+a))^\top.
 $
 Then, $X_a=\{X_a(t), t\in \R\}$ is a stationary process in $\R^2$, with auto-covariance 
 $$
 C_a(t):= \E[ X_a(t) X_a(0)^\top] = \left(\begin{array}{cc}
  \gamma(t) & \gamma(t-a)\\
  \gamma(t+a) & \gamma(t)
 \end{array} \right).
$$ 
This shows that so long as $\gamma(t+a) \not =\gamma(t-a)$, for all $t$, i.e., the auto-covariance does not vanish on $(-a/2,a/2)$, we have that $C_a(t) \not = C_a(-t) \equiv C_a(t)^\top$, namely, the process $X_a$ is not time-reversible.  Remark \ref{r:time_reverse} below then shows that the spectral density cannot be real-valued.
The simple example illustrates that a complex spectral density is a norm rather than an exception if $d\not=1$.
 \end{remark}

\subsection{Bochner's Theorem} \label{ss:Bochner}

This subsection discusses the notion of spectral density for a second-order stationary process $X$ in $\bbH$.
First, we briefly review some basic facts on trace-class operators.  
The reader is referred to the standard texts on linear operators \citep[e.g.,][]{simon2015comprehensive} for details. 
Denote by $\mathbb{T}$ the collection of trace-class operators on $\mathbb{H}$, namely,
linear operators $\CT: \mathbb{H}\to \mathbb{H}$, with finite {\em trace norm}:
\begin{align} \label{e:tr_norm}
\|\CT\|_{\tr} := \sum_{j=1}^\infty \langle (\CT^*\CT)^{1/2} e_j, e_j\rangle < \infty,
\end{align}
where $\{e_j\}$ is an arbitrary CONS on ${\mathbb H}$, and $\CT^*$ denotes
the adjoint operator of $\CT$, i.e., defined by $\langle \CT^*f,g\rangle = \langle f, \CT g\rangle, f,g\in \mathbb{H}$. 
The trace norm does not depend on the choice of the CONS, 
and the space $\bbT$ equipped with the trace norm is a Banach space.
By the definitions of the outer product \eqref{e:outer_product} and trace norm \eqref{e:tr_norm}, we have $\|X(t)\otimes X(s)\|_\tr = \|X(t)\|\|X(s)\|$.
The fact that $X$ is second order then implies that
$$
 \E[ \| X(t)\otimes X(s)\|_{\rm tr} ] \le \sqrt{\E [\|X(t)\|^2]  \E [\|X(s)\|^2]} <\infty.
$$
Consequently, the covariance operator $C(t,s)$ in \eqref{e:CXY} is well defined in $\bbT$ in the sense of Bochner; see, e.g., Lemma S.2.2 of \cite{shen2020tangent}.  

Recall that $\CT$ is self-adjoint if $\mathcal{T}=\mathcal{T}^*$.  Also $\CT$ is positive definite 
(or just positive), denoted $\CT\ge 0$, if $\CT$ is self-adjoint and $\langle f,\CT f\rangle\ge 0$, 
for all $f\in\bbH$.  The class of positive, trace-class operators will be denoted by $\mathbb{T}_+$.

The classical Bochner's Theorem \citep[cf.][]{bochner1948vorlesungen, 
khintchine1934korrelationstheorie}, which characterizes positive-definite functions, has provided a fundamental tool for  constructing useful 
models for stationary processes.  Below we state an extension of that classical result for our infinite-dimensional setting. To do so, we need the notion of integration with respect to a 
$\mathbb{T}_+$-valued measure which we now briefly describe. 
Let ${\cal B}(\R^d)$ denote the $\sigma$-field of Borel sets in $\R^d$.
We say that $\mu: \mathcal{B}(\R^d) \mapsto \mathbb{T}_+$ is a $\mathbb{T}_+$-valued measure on $\mathcal{B}(\R^d)$ if $\mu$
is $\sigma$-additive.  Note that, {\em a fortiori}, $\mu(\emptyset) = 0$ and $\mu$ is finite in the sense that 
$0\le \mu(B)\le \mu(\R^d) \in \bbT_+,\ B\in{\cal B}(\R^d)$, where for ${\cal T}_1, {\cal T}_2 \in \bbT_+$, ${\cal T}_1\le {\cal T}_2$ means that ${\cal T}_2-{\cal T}_1\in \bbT_+$. 
Integration of a $\mathbb{C}$-valued measurable function on $\bbR^d$ with respect to such $\mu$ can be defined along the line of Lebesgue 
integral \citep[see, e.g.,][for details]{shen2020tangent}.

\begin{thm}\label{thm:Bochner-Neeb}
Let $X$ be a second-order stationary process taking values in $\bbH$, and
let $C(h),\ h\in\R^d$, be its $\bbT$-valued stationary covariance function defined in Definition \ref{def:covariance operator}. Assume that $C$ is continuous at $0$ in trace norm. Then, there exists a unique $\mathbb{T}_+$-valued measure $\nu$ such that 
$$
 C(h)= \int_{\mathbb{R}^d}e^{-\mathbbm{i}h^\top\theta}\nu(d\theta),\ \ h\in\R^d.
$$
In particular, we have that $\|\nu(\mathbb{R}^d)\|_{\rm tr}={\rm trace}(\nu(\mathbb{R}^d))<\infty.$

If, moreover,  $\int_{h\in\mathbb{R}^d}\|C(h)\|_{\rm tr}dh<\infty,$ then the measure $\nu$ has a density with respect to the Lebesgue measure given by 
\begin{align}\label{def:H-f(theta)}
f(\theta):=\frac{1}{(2\pi)^d}\int_{\mathbb{R}^d}e^{\mathbbm{i}h^{\top}\theta}C(h)dh,\ \ \theta\in \R^d,
\end{align}
where the last integral is understood in the sense of Bochner.
\end{thm}

The density function $f$ in \eqref{def:H-f(theta)} is referred to as the spectral density of the stationary process. The detailed proof of Theorem \ref{thm:Bochner-Neeb} can be found in \cite{shen2020tangent}, where the role of separability and complex scalar field are made clear. 

The following is a follow-up remark to Remark \ref{r:2.2}.

\begin{remark} \label{r:time_reverse}
Theorem \ref{thm:Bochner-Neeb} holds for a general second-order stationary process $X$ in $\bbH$. Let us consider an interesting property of the spectral density if the process is real (defined according to some fixed CONS). To do that, define the conjugate $\ol{{\cal A}}$ of an operator ${\cal A} :\H \to \H$ by $\ol{\cal A}: x \mapsto \ol{{\cal A}(\ol x)},\ x\in \H$; accordingly, define
$$
\Re({\cal A}) := \frac{\ol{{\cal A}}+{\cal A}}{2},\ \Im({\cal A}) := \frac{{\cal A}-\ol{{\cal A}}}{2\ii}.
$$
Thus, ${\cal A}$ will be called real if $\Im({\cal A}) = 0$.
Suppose now $X$ is real (cf.\ Remark \eqref{r:2.1}). By the simple fact that $x\otimes y$ is real if both $x$ and $y$ are real, we have $C(h) = \overline{C(h)}$.
It then follows from \eqref{def:H-f(theta)} that the {\em time-reversed} process $Y = \{ X(-t),\ t\in \bbR^d\}$ has the spectral density
$$
f_Y(\theta) = \overline {f_X(\theta)},\ \ \ \theta\in \R^d.
$$
The uniqueness of the spectral density entails that $X$ and $Y$ have the same auto-covariance 
if and only if $f_Y(\theta) = f_X(\theta) = \overline {f_X(\theta)}$, that is, $f_X(\theta)$ is a real operator, 
for all $\theta\in\R^d$. This is a special property that is automatically true only when $\bbH$ is one-dimensional.
For further discussions, see Section 4.3 of \cite{shen2020tangent}.
\end{remark}

\section{Spectral density estimation based on irregularly sampled data}\label{sec:estimator}

Our inference problem focuses on a second-order real process $X=\{X(t),\ t\in \R^d\}$ taking values in $\bbH$. 
Following Definition \ref{def:covariance operator}, we define the stationary covariance operator $C$ and assume that the following holds.

\renewcommand\theassumption{C}
\begin{assumption}\label{a:cov} Let $C=\{C(h),\ h\in \R^d\}$ be the $\bbT$-valued stationary covariance operator of 
the second-order stationary real process $X = \{X(t),\ t\in \R^d\}$ taking values in $\H$. Assume that
\begin{itemize}
\item[(a)]
$\int_{h\in\mathbb{R}^d} \|C(h)\|_{\rm tr} dh<\infty$, and
\vskip.3cm
\item[(b)] 
$C(h)$ is $L^1$-$\gamma$-H\"older in the following sense:
\begin{align}\label{e:C-L1-Holder}
\int_{x\in \R^d} \Big( \sup_{y\, :\, \|x-y\|\le \delta }\left\Vert C(y)-C(x)\right\Vert_{\rm tr}   \Big) dx\le 
 \vertiii{C}_{\gamma} \cdot \delta^\gamma,
\end{align}
for some $0<\gamma\leq1$ and some (and hence all) $\delta>0$, where $\vertiii{C}_\gamma<\infty$ is a fixed constant. 
\end{itemize}
\end{assumption}

Property (a) in Assumption \ref{a:cov} guarantees the existence of the spectral density $f$ given by \eqref{def:H-f(theta)}. 
Property (b) will be needed to compute the bias of our estimator which is based on discretely observed data. 
It can be seen that Condition (b) holds with $\gamma=1$ if $C$ has an integrable and smoothly varying derivative.

We next introduce our sampling framework. As mentioned in Section \ref{sec:intro}, we adopt the mixed-domain asymptotics framework, which means that both the domain and the density of the data increase with sample size.  Assume that the process $\{X(t),\ t\in \R^d\}$ is observed at distinct locations $t_{n,i}, i=1,\ldots,n$. Let
$\mathbb{T}_n:=\{t_{n,i}\}_{i=1}^n$, and $T_n$ denote the closed convex hull of $\mathbb{T}_n$. We refer to $T_n$ as the {\em sampling region},
which contains points where $X(t)$ could potentially be observed.  However, as seen in our proofs, 
other contiguous regions may also be used for $T_n$.
For our purpose, it is convenient to view $T_n$ as a tessellation comprising disjoint cells, $V(t_{n,i})$, that are ``centered'' at the $t_{n,i}$:
\begin{align*}
T_n=\bigcup_{i=1}^nV(t_{n,i}), \quad \mbox{where} \quad t_{n,i}\in V(t_{n,i}) \quad \mbox{and}  \quad |V(t_{n,i})\cap V(t_{n,j})|= 0, \ i\neq j.
\end{align*}
Here and elsewhere, $|A|$ denotes the Lebesgue measure of a measurable set $A\subset\R^d$.  
Denote $\mathbb{V}=\{V(t_{n,i}), \ i=1,\hdots,n\}$. 
The Voronoi tessellation \cite{voronoi1908} is a natural example of 
such tessellation and can also be efficiently constructed \cite{yan2013efficient}.
While our results hold for a wide class of tessellations, to fix ideas we will adopt the Voronoi tessellation in the sequel.

Define the diameter of the Voronoi tessellation as 
\begin{align}\label{def:delta_n}
\delta_n:={\rm diam}_{T_n}(\{t_{n,i}\})=\max_{i=1,\hdots,n}\sup_{t\in V(t_{n,i})}\|t-t_{n,i}\|_2,    
\end{align}
where $\|\cdot\|_2$ denotes the Euclidean norm in $\mathbb{R}^d$. 
The parameter $\delta_n$ can be thought of as a measure of the maximal size of the tessellation cells, and 
can be equivalently written as 
$$
\delta_n = \sup_{t\in T_n}\min_{i=1,\hdots,n}\|t-t_{n,i}\|_2.
$$
Throughout, we will assume the following rather general sampling framework.

\renewcommand\theassumption{S}
 
\begin{assumption}\label{a:sampling}
\begin{itemize}
    \item[(a)] The sequence $\delta_n$ defined in \eqref{def:delta_n} tends to zero as $n\to\infty.$ Moreover, $|T_n|\to\infty$ as $n\to\infty$.\\

    \item[(b)]  The sample design is such that  
\begin{align} \label{e:boundary}
\frac{T_n}{|T_n|^{1/d}}\to T,\ \textrm{as}\ n\to\infty,
\end{align}
holds in probability, in the Hausdorff metric, for some fixed bounded convex set $T$ with non-empty interior. 
\end{itemize}
\end{assumption}
The condition (a) above describes the mixed-domain asymptotics framework alluded to earlier. 
Relation \eqref{e:boundary} in (b) essentially imposes a regularity condition on the boundary points of $\mathbb{T}_n$; for instance, 
if $\mathbb{T}_n = \{1,2,\ldots, n\}^d$ then $T=[0,1]^d$.  \\

The definition of our proposed estimator involves a kernel function, which satisfies the following standard conditions.

 \renewcommand\theassumption{K}

\begin{assumption} \label{a:kernel}
The kernel $K$ is a continuous function from $\mathbb{R}^d$ to $\R_+$ satisfying
\begin{itemize}
    \item [(a)] The support $S_K:=\{t\in\R^d\, :\, K(t)>0\}$ of $K$ is a bounded set containing $0$;
    \item [(b)] $\|K\|_{\infty}:=\sup_{u\in S_K}K(u) = K(0)=1$;
    \item[(c)] $K$ is differentiable in an $\epsilon$-neighborhood of 0 for some fixed $\epsilon>0$, with $$\|\nabla K\|_{\infty}^{(\epsilon)}:=\sup_{\|u\|_2<\epsilon}\|\nabla K(u)\|_2<\infty,$$
    where $\nabla$ stands for the gradient operator.
\end{itemize}
\end{assumption}

{\bf The estimator.} In this paper, we focus on the following non-parametric estimator of the spectral density 
$f(\theta)$:
\begin{align}\label{def:f^K(theta)}
\begin{split}
    \hat f_n(\theta) & = \frac{1}{(2\pi)^d}\sum_{t\in\mathbb{T}_n}\sum_{s\in\mathbb{T}_n}e^{\mathbbm{i}(t-s)^{\top}\theta}\frac{X(t)\otimes X(s)}{|T_n\cap(T_n-(t-s))|} \\
&  \hspace{2cm}  
\cdot K\left(\frac{t-s}{\Delta_n}\right)\cdot|V(t)|\cdot|V(s)|,
\end{split}
\end{align}
where $\Delta_n>0$ is a bandwidth parameter, the purpose of which is providing weighted averaging over observations that are at most $\Delta_n\cdot|S_K|$ apart. The choice of $\Delta_n$ that will lead to satisfactory estimation results depends on both $\delta_n$ and $|T_n|.$ 

The estimator in \eqref{def:f^K(theta)} can be applied to the general setting of functional data sampled irregularly over space and time, which is frequently 
encountered in applications \citep[see, e.g.,][]{yarger:stoev:hsing:2022}. In the special case where $\mathbb{T}_n$ is a regular grid, which includes the time-series setting, 
the terms $V(t)$ are constant for any $t\in\mathbb{T}_n$ and hence $|V(s)|$ and $|V(t)|$ can be factored out of the summation in $\hat f_n$ (see 
Section \ref{sec:discrete-time}).  In this case, the estimator in \eqref{def:f^K(theta)} is related to the so-called lag window estimator in time-series analysis; 
see \cite{robinson1983review}, \cite{zhu2020higher} and the discussions in Section \ref{subsec:flat-top} below.

To gain some insight into the definition \eqref{def:f^K(theta)}, consider the idealized setting where the full sample path of $\{X(t),t\in T_n\}$ is
available.  In view of \eqref{def:H-f(theta)}, one would naturally use the estimator 
\begin{align}\label{def:hat f_n}
    g_n(\theta)=\frac{1}{(2\pi)^d}\int_{t\in T_n}\int_{s\in T_n}e^{\mathbbm{i}(t-s)^\top\theta}\frac{X(t)\otimes X(s)}{|T_n\cap(T_n-(t-s))|} K\left(\frac{t-s}{\Delta_n}\right)dtds.
\end{align}
Since the full sample path is not available in practice one must consider approximations such as $\hat f_n(\theta)$,
which can be viewed as a Riemann sum for the integral defining $g_n(\theta)$.  The function $g_n(\theta)$ motivates the definition of
$\hat f_n(\theta)$ and in fact arises in the proofs of the asymptotic theory. 

We end the section with the following remarks.

\begin{remark} In our data scheme, we assume a fixed design where the observation points $t_{n,i}$ are nonrandom.
Our results can be modified in a straightforward manner to include
the case of a random design that is independently generated from the process $\{X(t)\}$. 
In this case, the definition of the estimator in \eqref{def:f^K(theta)} needs to be modified slightly to incorporate the probability densities of the sample design
in place of the volume elements \citep[cf., for example,][]{matsuda2009fourier}.
\end{remark}

\begin{remark}\label{re:Guyon parametrization}
The normalization $|T_n\cap(T_n-(t-s))|$ in \eqref{def:f^K(theta)} and \eqref{def:hat f_n} might seem unusual at first glance, 
whereas the simpler normalization by $|T_n|$ would seem more natural.  It turns out that the use of the latter normalization leads to a bias with a higher order
in the spatial context $d\ge 2$. Similar phenomenon arises for periodogram-based estimators in time series when data 
are observed over a regular lattice \citep[][]{guyon1982parameter}.
\end{remark}
\begin{remark}\label{re: Practicality}
The estimator $\hat f_n$ is defined assuming that we have fully observed functional data $X(t),t\in \mathbb{T}_n$. If $\bbH$ is
infinite dimensional, then the functional data $X(t)$ can never be observed in its entirety. 
In that case, we need to approximate $X(t)\otimes X(s)$ in some manner based on what is actually observed for the
functional data, which may affect the performance of the estimator.
We will discuss this point in more detail in Section \ref{s:rkhs}.
\end{remark}

\section{Asymptotic properties} \label{s:asymp}

We start our investigation of $\hat f_n(\theta)$ defined in Section 3 by first developing the asymptotic bounds for its bias and variance. This will yield results on 
the consistency and rate of convergence of the estimator. Although $f(\theta)$ and $\hat f_n(\theta)$ are trace-class operators 
on $\bbH$, in order to facilitate the variance calculation, it is more natural to work with the Hilbert-Schmidt (HS) norm. Let $\mathbb{X}$ denote 
the class of Hilbert-Schmidt operators on $\bbH$.  The Hilbert-Schmidt inner product  of the linear operators $\mathcal{A},\mathcal{B}\in\mathbb{X}$
is defined as 
\begin{align*}
\langle \mathcal{A},\mathcal{B}\rangle_{\rm HS} ={\rm trace}\left(\mathcal{A}^*\mathcal{B}\right)
\end{align*}
and $\|\mathcal{A}\|_{\rm HS}:=\sqrt{\|\mathcal{A}^*\mathcal{A}\|_{\rm tr}}$ \citep[see, e.g.,][]{simon2015comprehensive}.

It is straightforward to establish the following bias-variance decomposition
\begin{align} \label{eq: bias-variance decomposition}
\begin{split}
    \mathbb{E}\left\Vert\hat f_n(\theta)-f(\theta)\right\Vert_{\rm HS}^2  & = \mathbb{E}\left\Vert\hat f_n(\theta)-\mathbb{E}\hat f_n(\theta)\right\Vert_{\rm HS}^2 + \left\Vert \mathbb{E} \hat f_n(\theta) -f(\theta)\right\Vert_{\rm HS}^2\\
        & =: {\rm Var}\left(\hat f_n(\theta)\right)+ {\rm Bias}\left(\hat f_n(\theta)\right)^2.
\end{split}
\end{align}

\subsection{Asymptotic bias}\label{subsec:unbiasedness}
In this subsection, we evaluate the rate of the bias of $\hat f_n(\theta)$ for large $n$.  We start with a general bound, which is made more informative in the sequel.
Consistent with \eqref{eq: bias-variance decomposition}, the bounds in the following Theorem \ref{thm:expectation f_n^K-f_n} are stated in the Hilbert-Schmidt norm. However, we note that the result remains valid if the stronger trace norm is used throughout.

\begin{thm}\label{thm:expectation f_n^K-f_n}
Let Assumptions \ref{a:cov}, \ref{a:kernel}, and \ref{a:sampling} hold. 
Choose $\Delta_n\to\infty$ such that 
\begin{align*} 
\Delta_n\cdot S_K\subset T_n-T_n \quad\mbox{ for all $n$}
\end{align*}
where $A-B:=\{a-b:\ a\in A,b\in B\}$ for sets $A,B\subset \mathbb{R}^d$.
Then, for any bounded set $\Theta$ 
\begin{align}\label{e:thm:expectation f_n^K-f_n rate}
    \sup_{\theta\in\Theta}\left\Vert\mathbb{E}\hat f_n(\theta) - f(\theta)\right\Vert_{\rm HS} = \mathcal{O}\Big(\delta_n^{\gamma}+B_1(\Delta_n)+B_2(\Delta_n)\Big),
\end{align}
where \begin{align}\label{e:B terms}
\begin{split}
    B_1(\Delta_n) & := \left\Vert\int_{h\in\Delta_n\cdot S_K}e^{\mathbbm{i}h^{\top}\theta}C(h)\left(1-K\left(\frac{h}{\Delta_n}\right)\right)dh\right\Vert_{\rm HS}, \\
    B_2(\Delta_n)&:= \left\Vert\int_{h\not\in \Delta_n\cdot S_K}e^{\mathbbm{i}h^{\top}\theta}C(h)dh\right\Vert_{\rm HS}.
    \end{split}
\end{align}
\end{thm}

\begin{proof}[Proof (Outline)] The complete proof of Theorem \ref{thm:expectation f_n^K-f_n} is given in Section \ref{s:proof_of_bias}. Here, we
provide a brief outline.  Let $g_n(\theta)$ be defined by \eqref{def:hat f_n}.
By the triangle inequality,
\begin{align*}
    \Big\Vert\mathbb{E}\hat f_n(\theta)-f(\theta)\Big\Vert_{\rm HS}\leq  \Big\Vert\mathbb{E}\hat f_n(\theta)-\mathbb{E}g_n(\theta) \Big\Vert_{\rm HS}+ \Big\Vert\mathbb{E}g_n(\theta)-f(\theta) \Big\Vert_{\rm HS}.
\end{align*}
It is immediate from the representation \eqref{def:H-f(theta)} for $f$ and the inclusion
$\Delta_n\cdot S_K\subset T_n-T_n$, that
\begin{equation*} 
\left\|\mathbb{E}g_n(\theta)-f(\theta)\right\|_{\rm HS} \le B_1(\Delta_n) + B_2(\Delta_n).
\end{equation*}
To complete the proof one needs to show that
\begin{equation}\label{e:Efhat-Eg}
\left\|\mathbb{E}\hat f_n(\theta)-\mathbb{E}g_n(\theta)\right\|_{\rm HS} = {\cal O}(\delta_n^\gamma).
\end{equation}
To evaluate $ \left\|\mathbb{E}\hat f_n(\theta)-\mathbb{E}g_n(\theta)\right\|_{\rm HS}$, let
\begin{align*} 
h_n(t,s;\theta):=e^{\mathbbm{i}(t-s)^{\top}\theta}\frac{X(t)\otimes X(s)}{|T_n\cap(T_n-(t-s))|}K\left(\frac{t-s}{\Delta_n}\right),
\end{align*}
and write
\begin{align} \label{e:riemann}
\begin{split}
& g_n(\theta)-\hat f_n(\theta) \\
& =\frac{1}{(2\pi)^d}\sum_{w\in\mathbb{T}_n}\sum_{v\in\mathbb{T}_n}\int_{t\in V(w)}\int_{s\in V(v)} \left(h_n(t,s;\theta) - h_n(w,v;\theta)\right)  
\mathbbm{1}_{\left(t\in V(w),s\in V(v)\right)} dtds.
\end{split}
\end{align}
This implies that  
\begin{align*}
\begin{split}
& \left\Vert\mathbb{E} g_n(\theta)-\mathbb{E}\hat f_n(\theta)\right\Vert_{\rm HS} \\
& \le \frac{1}{(2\pi)^d}\sum_{w\in\mathbb{T}_n}\sum_{v\in\mathbb{T}_n} \int_{t\in V(w)}\int_{s\in V(v)}\|\mathbb{E}h_n(t,s;\theta)-\mathbb{E}h_n(w,v;\theta)\|_{\rm HS}dtds.
\end{split}
\end{align*}
Then, using the regularity conditions on $K$ and $C$, routine but technical analysis shows that the last sum is of order ${\cal O}(\delta_n^\gamma)$.
This yields \eqref{e:Efhat-Eg} and completes the proof of \eqref{e:thm:expectation f_n^K-f_n rate}.
\end{proof}

Several remarks are in order.
      
\begin{remark} \label{r:bias} Theorem \ref{thm:expectation f_n^K-f_n} provides a general bound on the bias.  
Under the assumptions of the theorem, the bias vanishes as $n\to\infty$. 
We briefly discuss the terms $\delta_n^\gamma$ and $B_1(\Delta_n) + B_2(\Delta_n)$ which arise for different reasons.
\begin{itemize}
\item[1.] 
As can be seen from the above sketch of the proof, the terms $B_1(\Delta_n)$ and $B_2(\Delta_n)$ in 
\eqref{e:thm:expectation f_n^K-f_n rate} control the bias of the idealized estimator $g_n(\theta)$ based on the idealized data. A more specific but crude bound of $B_1(\Delta_n)$ and $B_2(\Delta_n)$ is the following:
\begin{align} \label{e:bias_estimate}
\begin{split}
B_1(\Delta_n) 
& \le \int_{|h|\le \epsilon\Delta_n}\|C(h)\|_\tr \left|1-K\left(\frac{h}{\Delta_n}\right)\right|dh
+ \int_{|h|> \epsilon\Delta_n}\|C(h)\|_{\tr} dh \\
& \le \|\nabla K\|_{\infty}^{(\epsilon)} \epsilon \int \|C(h)\|_{\tr} dh + \int_{|h|> \epsilon\Delta_n}\|C(h)\|_{\tr} dh.
\end{split}
\end{align}
The first term on the rhs depends only on the kernel, whereas the second term, which dominates $B_2(\Delta)$ for any small $\epsilon < 1$, depends on the decay rate of $\|C(h)\|_\tr$.
Thus, the rate of $B_1(\Delta_n)+B_2(\Delta_n)$ is bounded by
\begin{align*} 
\inf_\epsilon \big(\epsilon \vee \psi(\epsilon\Delta_n)\big) \quad\mbox{where}\quad \psi(u) := \int_{|h|>u}\|C(h)\|_{\tr} dh.
\end{align*}
More explicit bounds can be obtained by imposing specific assumptions on the behavior of $\psi(u)$ for large $u$,
as will be demonstrated in Section \ref{sec:discrete-time}. 

\vskip.3cm

\item[2.] In view of \eqref{e:Efhat-Eg}, the term ${\cal O}(\delta_n^\gamma)$ controls the bias due to 
discretization, which arises from sampling the process at the discrete set $\mathbb{T}_n \subset T_n$. In settings such as time series where the data are sampled on a regular grid, this term will be eliminated from the bias (cf. Theorem \ref{prop: consistency f^ts}). 
\end{itemize}
\end{remark}

\subsection{Asymptotic variance}\label{subsec:variance}

In view of the form of $\hat f_n(\theta)$, a ``fourth-moment'' condition of $X$ is needed to evaluate the variance of $\hat f_n$. 

Recall the definition of cumulant for random variables: For real-valued random variables
$Y_j,j=1,\ldots,k$, 
\begin{align} \label{e:cum_complex}
{\rm cum}\left(Y_1,\hdots,Y_k\right):=\sum_{\nu=(\nu_1,\hdots,\nu_q)}(-1)^{q-1}(q-1)!\prod_{l=1}^q\mathbb{E}\left(\prod_{j\in \nu_l}Y_j\right),
\end{align}
provided all the expectations on the rhs are well defined, where the sum is taken over all unordered partitions $\nu$ of $\{1,\hdots,k\}.$

We now define a notion of fourth-order cumulant for complex Hilbert space valued random variables $Y_1,Y_2,Y_3,Y_4$ with mean zero.

\begin{definition} \label{def:cum}
Let $Y_1, Y_2, Y_3, Y_4$ take values in $\bbH$. Then the fourth-order cumulant is defined as
\begin{align*} 
{\rm cum}\left(Y_1,Y_2, Y_3, Y_4\right) 
&:= \mathbb{E}\left\langle Y_1\otimes Y_2,Y_3\otimes Y_4 \right\rangle_{\rm HS} 
- \langle \E (Y_1\otimes Y_2), \E (Y_3\otimes Y_4) \rangle_{\rm HS} \\
& \hskip.5cm- \mathbb{E}\left\langle Y_1,Y_3 \right\rangle\cdot\mathbb{E}\left\langle Y_4,Y_2\right\rangle
-  \left\langle \E (Y_1\otimes \overline{Y_4}), \E (Y_3\otimes \overline{Y_2}) \right\rangle_{\rm HS},
\end{align*}
whenever the expression is well defined and finite.
\end{definition}

Note that ${\rm cum}\left(Y_1,Y_2, Y_3, Y_4\right)$ is well defined and finite if $\E\|Y_i\|^4 < \infty$ for each $i$ (cf. Proposition \ref{prop:Cov for cross product}). 
It is easy to check that this definition reduces to \eqref{e:cum_complex} with $k=4$ if $\bbH=\bbR$.
 
Some properties immediately follow from Proposition \ref{prop:Cov for cross product}. First,
\begin{align} \label{e:cum_1}
\left\langle Y_1\otimes Y_2,Y_3\otimes Y_4 \right\rangle_{\rm HS} 
= \left\langle Y_1,Y_3 \right\rangle\left\langle Y_4,Y_2\right\rangle,
\end{align}
and hence we can express the fourth-order cumulant as
\begin{align*} 
\begin{split}
{\rm cum}\left(Y_1,Y_2, Y_3, Y_4\right) & = \cov\left(\left\langle Y_1,Y_3 \right\rangle, \left\langle Y_2,Y_4\right\rangle\right)
- \langle \E (Y_1\otimes Y_2), \E (Y_3\otimes Y_4) \rangle_{\rm HS} \\
& \hspace{1cm} -  \left\langle \E (Y_1\otimes \overline{Y_4}), \E (Y_3\otimes \overline{Y_2}) \right\rangle_{\rm HS}.
\end{split}
\end{align*}
Next, for any CONS $\{e_j\}$ of $\bbH$, and with 
$Y_{i,j} := \langle Y_i, e_j\rangle$,
\begin{align} \label{e:cum_basis}
    {\rm cum}\left(Y_1, Y_2, Y_3, Y_4\right) 
    = \sum_i\sum_j {\rm cum}(Y_{1,i},\overline{Y_{2,j}},\overline{Y_{3,i}},Y_{4,j}).
    \end{align}
Observe that, unless $\bbH$ is one dimensional, ${\rm cum}\left(Y_1,Y_2, Y_3, Y_4\right)$ generally depends on the order in which the $Y_i$'s appear in the 
arguments. 
 
For the real process $X$ that we consider in our inference problem, assuming $\E \|X(t)\|^4 < \infty$ for all $t$, we have
\begin{align} \label{e:cum_2}
\begin{split}
& {\rm cum}\left(X(t),X(s),X(w),X(v)\right) \\
&:= \mathbb{E}\left\langle X(t)\otimes X(s),X(w)\otimes X(v) \right\rangle_{\rm HS} 
- \langle C(t,s), C(w,v) \rangle_{\rm HS} \\
& \hskip.5cm- \mathbb{E}\left\langle X(t),X(w) \right\rangle_{\mathbb{H}}\cdot\mathbb{E}\left\langle X(v),X(s)\right\rangle_{\mathbb{H}} 
-  \left\langle C(t,v), C(w,s)\right\rangle_{\rm HS}.
\end{split}
\end{align}

The following assumption will be needed to evaluate the variance of $\hat f_n(\theta)$. 

\renewcommand\theassumption{V} 
\begin{assumption}\label{a:var}  Suppose that the process $X$ is real and such that:
\begin{itemize}
\item[(a)] $\E \|X(t)\|^4 < \infty$ for all $t$;
\item[(b)] ${\rm cum}\left(X(t+\tau),X(s+\tau),X(w+\tau),X(v+\tau)\right) = {\rm cum}\left(X(t),X(s),X(w),X(v)\right)$ 
for all $t,s,w,v,\tau$;
\item[(c)] 
for some small enough $\delta>0$,  
$$\sup_{w\in\mathbb{R}^d}\int_{u\in\mathbb{R}^d}\int_{v\in\mathbb{R}^d}\sup_{\substack{\lambda_i \in B(0,\delta)\\ i=1,2,3}}\left|
   {\rm cum}\left(X(\lambda_1+u),X(\lambda_2+v),X(\lambda_3+w),X(0)\right)\right|dvdu<\infty.$$
 \end{itemize}
\end{assumption}

The following are a few remarks regarding Assumption \ref{a:var}.

\begin{remark} 
\begin{itemize}

\item[1.] Part (b) of this assumption can be thought of as ``fourth-order cumulant stationarity'', which is implied by but more general than strict stationarity.
For a second-order stationary process $X$, by \eqref{e:cum_1} and \eqref{e:cum_2}, part (b) amounts to
\begin{align*} 
\begin{split}
& \E (\langle X(t), X(s)\rangle \langle X(w), X(v)\rangle) \\
& = \E (\langle X(t+\tau), X(s+\tau)\rangle \langle X(w+\tau), X(v+\tau)\rangle)\quad\mbox{for all $t,s,w,v,\tau$}. 
\end{split}
\end{align*}

\vskip.2cm
\item[2.]
Part (c) of Assumption \ref{a:var}  is a variant of the cumulant condition ``$C(0,4)$'' of \cite{panaretos2013fourier} for functional time series
(see Remark \ref{s:rem:Panaretos-cumulants} for more details). \\ 

\item[3.] 
For Gaussian processes, by \eqref{e:cum_basis}, the fourth-order cumulants vanish and hence Assumption \ref{a:var} is trivially satisfied under stationarity.
\end{itemize}
\end{remark}

The variance bound of $\hat f_n(\theta)$ is provided by the following result.

\begin{thm}\label{thm:var for f^K}
Let $X=\left\{X(t), t\in\mathbb{R}^d\right\}$ be a real process taking values in $\bbH$, which has mean zero and is second-order stationary. 
Suppose that Assumptions \ref{a:cov}, \ref{a:kernel}, \ref{a:sampling}, and \ref{a:var} hold. Also, assume that $\Delta_n$ satisfies
\begin{align*} 
\Delta_n\cdot S_K\subset T_n-T_n \ \ \mbox{for all $n$, and}\ \ \Delta_n^d/|T_n|  \to 0 \mbox{ as $n\to\infty$}. 
\end{align*}
Then
\begin{equation}\label{e:thm:var for f^K}
\sup_{\theta\in\Theta}\mathbb{E}\Big( \left\|\hat f_n(\theta)-\mathbb{E}\hat f_n(\theta)\right\|_{\rm HS}^2 \Big) = \mathcal{O}\left(\frac{\Delta_n^d}{|T_n|}\right), \ \text{as}\ n\to\infty.
\end{equation}
\end{thm}

\begin{proof}[Proof (Outline)] The complete proof of Theorem \ref{thm:var for f^K} is presented in Section \ref{s:var_proof}.  
Here, we sketch the main steps. First, 
\begin{align} \label{e:var_1}
\begin{split}
& \mathbb{E}\left\|\hat f_n(\theta)-\mathbb{E}\hat f_n(\theta)\right\|_{\rm HS}^2 \\
& = \frac{1}{(2\pi)^{2d}}\sum_{t\in\mathbb{T}_n}\sum_{s\in\mathbb{T}_n}\mathop{\sum\sum}_{\substack{h\in[\Delta\cdot S_K]\cap(\mathbb{T}_n-t) \\h'\in[\Delta\cdot S_K]\cap(\mathbb{T}_n-s)}}e^{\mathbbm{i}(h-h')^{\top}\theta}K\left(\frac{h}{\Delta}\right)K\left(\frac{h'}{\Delta}\right)\\ 
& \hspace{3cm}\cdot|V(t+h)|\cdot|V(t)| \cdot|V(s+h')|\cdot|V(s)| \\
& \hspace{3cm}\cdot \frac{{\rm Cov}\left( X(t+h)\otimes X(t),X(s+h')\otimes X(s)\right)}{|T\cap(T-h)||T\cap(T-h')|}
\end{split}
\end{align}
where
\begin{align*}
\begin{split}
& {\rm Cov}\left( X(t+h)\otimes X(t),X(s+h')\otimes X(s)\right) \\
& := \mathbb{E}\left\langle X(t+h)\otimes X(t)-C(h),X(s+h')\otimes X(s)-C(h')\right\rangle_{\rm HS}.
\end{split}
\end{align*}
By \eqref{e:cum_2}, 
\begin{align} \label{e:var_3}
\begin{split}
& {\rm Cov}\left( X(t+h)\otimes X(t),X(s+h')\otimes X(s)\right) \\
    &=   \mathbb{E}\langle X(t+h),X(s+h')\rangle_{\mathbb{H}}\cdot\mathbb{E}\left\langle X(s),X(t)\right\rangle_{\mathbb{H}} \\
    & \hspace{.5cm} +  \langle  C(t-s+h),  C(s+h'-t)\rangle_{\rm HS}\\
    & \hspace{.5cm} + {\rm cum}\left(X(t+h),X(t),X(s+h'),X(s)\right).
\end{split}
\end{align}
In our detailed proof (presented in the Supplement), the components of the variance involving the cumulants will be evaluated using Assumption 
\ref{a:var}, while the other two terms are handled using the integrability condition of the covariance of Assumption \ref{a:cov}.
\end{proof}

\subsection{Rates of convergence} \label{subsec:consistency}
The results in Sections  \ref{subsec:unbiasedness} and \ref{subsec:variance} 
allow us to  obtain bounds on the rate of consistency of the estimator $\hat f_n(\theta)$. The following result is immediate from the bias-variance decomposition \eqref{eq: bias-variance decomposition}. 

\begin{thm}\label{thm:consistency f^K} 
Let the assumptions of Theorem \ref{thm:var for f^K} hold. Then, for any bounded $\Theta\subset \R^d$, we have
\begin{align}\label{e:thm:consistency f^K} 
\begin{split}
\sup_{\theta\in \Theta}\left( \mathbb{E} \left\| \hat f_n(\theta)-f(\theta)\right\|_{\rm HS}^2\right)^{1/2} 
 = \mathcal{O}\left(\delta_n^{\gamma}+B_1(\Delta_n)+B_2(\Delta_n)
     +\sqrt{\frac{\Delta_n^d}{|T_n|}}\right),
     \end{split}
\end{align}
as $n\to\infty$, where $B_1(\Delta_n)$ and $B_2(\Delta_n)$ are as defined in Theorem \ref{thm:expectation f_n^K-f_n}.
\end{thm}

Theorem \ref{thm:consistency f^K} provides general bounds on the rate of consistency of
the estimator $\hat f_n(\theta)$.  More explicit rates and their minimax optimality can be obtained under further conditions 
on the dependence structure of the process.  We conclude with several comments.

\begin{remark} 

\begin{itemize}
\item[1.]  The bound on the rate of consistency for the estimator $\hat f_n(\theta)$ in \eqref{e:thm:consistency f^K} depends on the quantities $\delta_n, \Delta_n$ and $|T_n|$. Among them, $\delta_n$ and $T_n$ consist of artifacts of the sample design, while $\Delta_n$ is a tuning parameter which can be controlled.   
Under the assumptions of the theorem, any choice of the bandwidth with $\Delta_n\to\infty$ and
$\Delta_n^d/|T_n|\to 0$ yields a consistent estimator $\hat f_n(\theta)$. 

\vskip.3cm

\item[2.] As discussed in Remark \ref{r:bias},  $B_1(\Delta_n)$ and $B_2(\Delta_n)$ in the rate mostly reflect the tail-decay of the covariance. They are not present 
in the bound on the variance \eqref{e:thm:var for f^K}, where smaller values of $\Delta_n$ lead to smaller variances of the estimator.
The bound in \eqref{e:thm:consistency f^K} reflects a natural bias-variance trade-off, where the optimal bound is obtained by picking 
$\Delta_n$ that balances the contribution of the bias and the variance.

\vskip.3cm

\item[3.]  Establishing rate-optimal choices of $\Delta_n$ depends on both the sampling design and the stochastic process under consideration. 
Indeed, the choice of $\Delta_n$  optimizing the bounds in \eqref{e:thm:consistency f^K} depends both on $\delta_n$ and $T_n$, as well as on the covariance structure of 
the process. In Section \ref{sec:discrete-time}, we will compute $B_1(\Delta_n)$ and $B_2(\Delta_n)$ and consider the choice 
of $\Delta_n$ for certain classes of covariance structures.
\end{itemize}
\end{remark}

\section{Data observed on a regular grid}\label{sec:discrete-time}
In this section and the following two sections we focus on data observed on a regular grid, namely, the sampling set is
\begin{align} \label{e:lattice}
\mathbb{T}_n=\bigtimes_{\ell=1}^d\{\delta_n, \hdots, n_\ell\delta_n\}, 
\end{align} 
where $\delta_n$ is the grid size. In our asymptotic theory in the next two subsections, we let $n_\ell\to\infty,\ell=1,\hdots,d$, and consider both cases of  fixed $\delta_n$ and $\delta_n\to 0$. 

In this setting, for convenience, we slightly modify our general estimator $\hat f_n(\theta)$ defined in \eqref{def:f^K(theta)}
and consider
\begin{align}\label{def: f^ts}
    \hat f_n (\theta) := \frac{\delta_n^{2d} }{(2\pi)^d}\sum_{t\in \mathbb{T}_n}\sum_{s\in \mathbb{T}_n}e^{\mathbbm{i}(t-s)^{\top}\theta}\frac{X(t)\otimes X(s)}{|T_n\cap(T_n-(t-s))|} K\left(\frac{t-s}{\Delta_n}\right).
\end{align}

Under the condition
\begin{align*}
\sum_{k\in\mathbb{Z}^d}\|C(k\delta_n)\|_{\rm tr}<\infty,   
\end{align*}
we have
\begin{align} \label{e:discrete_Bochner}
C(k\delta_n) = \int_{\theta\in[-\pi/\delta_n,\pi/\delta_n]^d}e^{-\ii k^{\top} \theta \delta_n} f(\theta;\delta_n)d\theta, \ k\in\mathbb{Z}^d,
\end{align}
where
\begin{align} \label{e:spec_fold}
f(\theta;\delta_n) := \frac{\delta^d}{(2\pi)^d}\sum_{k\in\mathbb{Z}^d} e^{\ii k^\top \theta\delta_n} C(k\delta_n), \ \theta\in[-\pi/\delta_n,\pi/\delta_n]^d,
\end{align}
which is a positive trace-class operator since $\{C(k\delta_n), k\in\mathbb{Z}^d\}$ is positive definite. The proof of \eqref{e:discrete_Bochner} follows easily using the fact that
the complex exponentials 
$$
\phi_k(\theta) := e^{\ii k^\top \theta\delta_n} (\delta_n/2\pi)^{d/2}\mathbbm{1}_{[-\pi/\delta_n, \pi/\delta_n]^d}(\theta),\ \ k\in\mathbb{Z}^d
$$ 
constitute a CONS of $L^2([-\pi/\delta_n$, $\pi/\delta_n]^d)$.
By Theorem \ref{thm:Bochner-Neeb}, \eqref{e:discrete_Bochner} also holds  if $f(\theta;\delta_n)$ is replaced by the folded spectral density 
\begin{align*} 
f_{\rm fold}(\theta) := \mathbbm{1}_{[-\pi/\delta_n, \pi/\delta_n]^d}(\theta) \sum_{k\in\mathbb{Z}^d} f(\theta+2\pi k/\delta_n).
\end{align*}
Utilizing again the fact that $\{\phi_k(\theta), k\in\mathbb{Z}^d\}$ is a CONS of $L^2([-\pi/\delta_n$, $\pi/\delta_n]^d)$, $f(\theta;\delta_n)$ is equal to the folded spectral density.
Thus, the knowledge of $C(k\delta_n), k\in\mathbb{Z}^d$, only allows us to identify the folded spectral density. In fact, this is reflected by our estimator $\hat f_n$ since
\begin{align*}
\hat f_n (\theta+2\pi k/\delta_n) = \hat f_n(\theta), \ \theta\in [-\pi/\delta_n, \pi/\delta_n]^d,
\end{align*}
for any vector $k\in\mathbb{Z}^d$.

For the purpose of estimating the folded spectral density, we define the following analogs of Assumptions \ref{a:cov} and \ref{a:var}.

\renewcommand\theassumption{C$^{\boldsymbol{\prime}}$}
\begin{assumption}\label{a:cov-prime} The trace-norm of the operator auto-covariance is summable:
\[
\sup_n \left\{ \delta_n^{d} \sum_{k\in\mathbb{Z}^d}  \|C(\delta_n k)\|_{\rm tr} \right\} <\infty.
\]
\end{assumption}

\vspace{-.25in}
\renewcommand\theassumption{V$^{\boldsymbol{\prime}}$}
\begin{assumption}\label{a:var_d}
The process $\{X(\delta_n t), t\in\mathbb{Z}^d\}$ satisfies
\begin{itemize}
\item [(a)] $\sup_n \E \|X(\delta_n t)\|^4 < \infty$ for all $t$;
\item[(b)] for all $t,s,w,v,\tau$,
\begin{align*}
& {\rm cum}\left(X(\delta_n(t+\tau)),X(\delta_n(s+\tau)),X(\delta_n(w+\tau)),X(\delta_n(v+\tau))\right) \\
& = {\rm cum}\left(X(\delta_nt),X(\delta_ns),X(\delta_nw),X(\delta_nv)\right);
\end{align*}
\item[(c)] $ \sup_n \Big\{ \delta_n^{2d} 
\sup_{w\in\mathbb{Z}^d}\sum_{u\in\mathbb{Z}^d}\sum_{v\in\mathbb{Z}^d}\left|{\rm cum}\left(X(\delta_n u),X(\delta_n v),X(\delta_n w),X(0)\right)\right|\Big\}<\infty$.
\end{itemize}
\end{assumption}

\noindent
Comparing with Assumptions \ref{a:cov} and \ref{a:var}, the modifications in Assumptions \ref{a:cov-prime} and \ref{a:var_d} are
motivated by the fact that discrete approximations of integrals is no longer an issue if our target of inference is the folded spectral density. 
We will apply these conditions in the time series context in Section \ref{s:ts}.

As before, Assumption \ref{a:var_d} holds trivially for Gaussian processes since the 4th order cumulants vanish. 
More generally, it holds for a wide class of short-memory $\H$-valued  processes (see Example \ref{ex:ar} in the Supplement). 

Note that our assumptions on the cumulants in Assumption \ref{a:var_d} are different from but related to the assumption based cumulant kernels employed on page 571 in 
\cite{panaretos2013fourier}.  For more details, see Remark \ref{s:rem:Panaretos-cumulants}.

\subsection{The case of fixed grid} \label{s:ts}

Consider the case where $\delta_n$ in \eqref{e:lattice} is fixed. Without loss of generality, let $\delta_n\equiv 1$. 
The discussion in the previous section shows that we can only identify the folded spectral density on $[-\pi,\pi]^d$. As such, without loss of generality, focus on a stochastic processes $\{X(t)\}$ indexed by $t\in\mathbb{Z}^d$. This framework includes time series (for $d=1$), and, more generally, many random fields observed at discrete locations/times. The spectral density $f$ in this case is defined by \eqref{e:spec_fold}.
With the normalization $|T_n\cap(T_n-(t-s))|$ replaced by $|T_n|$, 
$\hat f_n(\theta)$ recovers the classical lag-window estimator \citep[cf.][]{robinson1983review}.

The following result on the rate of $\hat f_n(\theta)$ is the analog to Theorem \ref{thm:consistency f^K} for the gridded setting.

\begin{thm}\label{prop: consistency f^ts}
Let $\{X(t),\ t\in\mathbb{Z}^d\}$ be a real process taking values in $\bbH$, which has mean zero and is second-order stationary. Suppose that Assumption \ref{a:kernel} holds, 
Assumptions \ref{a:cov-prime} and \ref{a:var_d} hold with $\delta_n\equiv 1$, and 
\begin{align*} 
\Delta_n\cdot S_K\cap \Z^d \subset \bbT_n-\bbT_n \hspace{.2cm}\mbox{for all $n$, and}\hspace{.2cm}
\Delta_n^d/|\bbT_n|  \to 0 \mbox{ as $n\to\infty$}.
\end{align*}
Then,
\begin{align} \label{e:bias-bound-time-series}
 \sup_{\theta\in[-\pi,\pi]^d} \left\| \E \hat f_n(\theta) - f(\theta) \right\|_{\rm HS} &\leq B_1(\Delta_n) + B_2(\Delta_n)\\
 \sup_{\theta\in[-\pi,\pi]^d} \E \left\| \hat f_n(\theta) - \E[\hat f_n(\theta)] \right\|_{\rm HS}^2 & = {\cal O}\left( \frac{\Delta_n^{d}}{{|\bbT_n|}}\right),
 \label{e:var-bound-time-series}
\end{align}
where 
\begin{align*}
    B_1(\Delta_n) & := \left\Vert\sum_{k\in(\Delta_n\cdot S_K)\cap \Z^d}e^{\mathbbm{i}k^{\top}\theta}C(k)\left(1-K\left(\frac{k}{\Delta_n}\right)\right)\right\Vert_{\rm HS}, \\
    B_2(\Delta_n)&:= \left\Vert\sum_{k\in \Z^d\setminus(\Delta_n\cdot S_K)}e^{\mathbbm{i}k^{\top}\theta}C(k)\right\Vert_{\rm HS}.
\end{align*}
Consequently,
\begin{align*}
\sup_{\theta\in[-\pi,\pi]^d} \Big( \mathbb{E}\left\Vert\hat f_n(\theta)-f(\theta)\right\Vert_{\rm HS}^2 \Big)^{1/2}= \mathcal{O}\left(B_1(\Delta_n)+B_2(\Delta_n)+\frac{\Delta_n^{d/2}}{\sqrt{|\bbT_n|}}\right),\ \text{as}\ n\to\infty.
\end{align*}
\end{thm}

In this result, the derivation of the bias bound \eqref{e:bias-bound-time-series} is more straightforward than that for the general 
case since it does not involve a Riemann approximation as in \eqref{e:riemann}. 
Here, the first term on the rhs of \eqref{e:thm:expectation f_n^K-f_n rate} is no longer present 
and the other two terms, $B_1(\Delta_n)$ and $B_2(\Delta_n)$, are similar to \eqref{e:B terms}, with sums replacing integrals.
The derivation of the variance bound \eqref{e:var-bound-time-series} is also 
simpler than that of \eqref{e:thm:var for f^K}, where the term involving $\delta_n$ is no longer needed. 
For completeness, the variance bound is established in Proposition \ref{thm: variance of hat f} of Supplement.

The bias bounds $B_1(\Delta_n), B_2(\Delta_n)$ in Theorem \ref{prop: consistency f^ts} hold for a very general class of models. However, more precise expressions of the bias can be obtained for specific models. We illustrate this next by considering a class of covariances that decay like the power law. 
The power-law decay class, ${\cal P}_D(\beta, L)$, for the discrete-time processes is defined as
\begin{align}\label{def:PL minimax class discrete time}
    {\cal P}_D(\beta,L):=\left\{f(\theta)=(2\pi)^{-d}\sum_{k\in\mathbb{Z}^d}C(k)e^{\ii \theta^{\top} k}:\ \sum_{k\in\mathbb{Z}^d}  \|C(k)\|_{\rm tr}
    (1+\|k\|_2^{\beta})\le L\right\}, 
\end{align}
for $\beta, L > 0$. By the theory of the Fourier transform, larger values of $\beta$ in this condition correspond to a higher order of smoothness of the spectral density at $\theta=0$; see, e.g., \cite{bingham1989regular}.

Below we establish an explicit upper bound on the rate of $\hat f_n(\theta)$ for this class by focusing on the bias terms $B_1(\Delta_n)$ and $B_2(\Delta_n)$ of Theorem \ref{prop: consistency f^ts}. First, we introduce an additional smoothness condition on the kernel $K$ that is compatible with the covariance model in ${\cal P}_D(\beta,L)$. Let $\boldsymbol{\alpha} = (\alpha_1,\ldots, \alpha_d) \in \mathbb{Z}_+^d$ and define the partial derivative
$$
\partial^{\boldsymbol{\alpha}}K(h) = \frac{\partial^{\boldsymbol{\alpha}}K(h)}{\partial h_1^{\alpha_1}\hdots\partial h_d^{\alpha_d}}.
$$
Then, for an integer $\lambda\ge 1$, define the condition
\begin{align}\label{e:Kernel differentiability}
\begin{split}
    & \partial^{\boldsymbol{\alpha}}K(0) =0
    \mbox{ for all $\boldsymbol{\alpha}$ with } 1\le |\boldsymbol{\alpha}|:=\sum_{i=1}^d \alpha_i \le \lambda, \mbox{ and} \\
    & \sup_h |\partial^{\boldsymbol{\alpha}}K(h)| < \infty \mbox{ for all $\boldsymbol{\alpha}$ with } |\boldsymbol{\alpha}|=\lambda+1.
\end{split}
\end{align}

\begin{thm}\label{prop:PL bias time series}
Let all the conditions of Theorem \ref{prop: consistency f^ts} hold. Moreover, assume that the spectral density $f$ belongs to ${\cal P}_D(\beta,L)$ for some $\beta>0$ and $L>0$, and that \eqref{e:Kernel differentiability} holds for some integer $\lambda> 0 \vee (\beta-1)$. Then, the following is a uniform bound on the rate of the bias
of $\hat f_n(\theta)$:
\begin{align} \label{e:bias_ts_PLD}
 \sup_{f\in {\cal P}_D(\beta,L)} \sup_{\theta\in[-\pi,\pi]^d}\left\Vert\mathbb{E}\hat f_n(\theta)-f(\theta)\right\Vert_{\rm HS} = \mathcal{O}\left(\Delta_n^{-\beta}\right),\quad\text{as}\quad n\to\infty.
\end{align}
Combining this with the variance bound $\Delta_n^d/|\bbT_n|$ in \eqref{e:var-bound-time-series} and choosing bandwidth $\Delta_n= |\bbT_n|^{\frac{1}{2\beta+d}}$, the following uniform bound on the mean squared error of  $\hat f_n(\theta)$ holds:
\begin{align}  \label{e:mse_ts_PLD}
 \sup_{f\in {\cal P}_D(\beta,L)} \sup_{\theta\in[-\pi,\pi]^d}\Big(\E \left\Vert\hat f_n(\theta)-f(\theta)\right\Vert_{\rm HS}^2\Big)^{1/2} = \mathcal{O}\left(|\bbT_n|^{-\frac{\beta}{2\beta+d}}\right).
\end{align}
\end{thm}

The proof of this result is given in Section \ref{appendix:bounds on the rates}.
An important motivation for singling out the class ${\cal P}_D(\beta,L)$ is that is covers a broad range of realistic covariance models whose tail-decay can be controlled by the 
parameter $\beta$.  Moreover, in Section \ref{sec:minimax} we establish a minimax lower bound for this class which matches the upper bound on the rate in \eqref{e:mse_ts_PLD}.
In this sense, our estimator with the oracle choice of the bandwidth is minimax rate-optimal.

\subsection{Dense gridded data} \label{s:denst_grid}

We now turn to the setting \eqref{e:lattice} where we assume $\delta_n\to 0$. In doing so, we continue to focus on the estimator $\hat f_n(\theta)$ in 
\eqref{def: f^ts} for gridded data. However, unlike the $\delta_n=1$ case, here we are in a position to estimate the full spectral density as opposed to the folded spectral density. As in the previous subsection, we also study a similar power law decay class. However, some slight 
modifications are necessary. The continuous time power law decay 
class ${\cal P}_C(\beta,L)$ where  $\beta,\ L>0$,  contains spectral densities 
for the continuous-time process, defined by
\begin{align}\label{def:CTPL minimax class}
    {\cal P}_C(\beta,L) := \Big\{ f(\theta)= (2\pi)^{-d} \int_{\R^d} e^{\ii x^\top \theta} C(x) dx\, :\, \int_{\R^d} (1+\|x\|_2^\beta) \|C(x)\|_{\rm tr} dx \le L \Big\}.
\end{align} 
Mimicking the approach in Section \ref{s:ts}, the following result can be stated for this class. 

\begin{thm}\label{prop:PL bias cont time}
Let all the assumptions of Theorem \ref{thm:consistency f^K} hold and assume that the spectral density $f(\theta)$ belongs in ${\cal P}_C(\beta,L)$ for some $\beta,L>0$. Suppose that
 \eqref{e:Kernel differentiability} holds for some integer  $\lambda> 0 \vee (\beta-1)$.
Then, for  every $f\in {\cal P}_C(\beta,L)$ and bounded $\Theta\subset\R^d$, the rate of the bias is  
$$  \sup_{\theta\in \Theta}\left\Vert\mathbb{E} \hat f_n(\theta)-f(\theta)\right\Vert_{\rm HS} 
= \mathcal{O}\left(\delta_n^{\gamma}+\Delta_n^{-\beta}\right),\quad\text{as}\quad n\to\infty.$$ 
In conjunction with Theorem \ref{thm:consistency f^K}, with the rate-optimal choice of 
$\Delta_n:= |T_n|^{1/(2\beta+d)}$, we obtain the overall rate bound:
\begin{align}\label{e:rate-bounds}
  \sup_{\theta\in\Theta}\left(\E \left\|\hat f_n(\theta)-f(\theta)\right\|_{\rm HS}^2 \right)^{1/2}
 =  {\cal O} \Big( \delta_n^{\gamma} \vee |T_n|^{-\beta/(2\beta + d)}\Big). 
\end{align}
\end{thm}
The proof of Theorem \ref{prop:PL bias cont time} is given in Section \ref{appendix:bounds on the rates}.  

\begin{remark} \label{r:CPL}
Observe that, in contrast to Theorem \ref{prop:PL bias time series}, the rate bounds in 
Theorem \ref{prop:PL bias cont time} are not uniform over the class ${\cal P}_C(\beta,L)$.   This is mainly because 
the constant $\vertiii{C}_\gamma$ in \eqref{e:C-L1-Holder} of Assumption \ref{a:cov} (b) cannot be bounded uniformly 
in ${\cal P}_C(\beta,L)$, since the tail behavior of $\|C(x)\|_\tr$ does not regulate the smoothness of $C(x)$.
At this point, we do not know whether there is an adaptive estimator for which the rate could be shown to be uniform.
\end{remark}

\begin{remark}\label{rem:sampling-regimes} To interpret the bound on the rate in \eqref{e:rate-bounds}, suppose, for example, that  
$\delta_n:=n^{-\alpha}$ for some $\alpha\in (0,1)$, which controls the sampling frequency relative to the sample size.
The greater the value of $\alpha$, the finer the grid. Also, assume that the grid is square with $n_\ell=n$, for all $\ell$, so that $|T_n| \sim (n\delta)^d$.
Let
\begin{align} \label{e:alpha_beta_gamma}
\alpha_{\beta,\gamma} = \left(1+\left(\frac{2}{d} + \frac{1}{\beta}\right) \gamma\right)^{-1}
\end{align}
and consider the following two regimes: 
 \begin{itemize}
   \item {\em (fine sampling)} When $\alpha\ge \alpha_{\beta,\gamma}$, then $\delta_n^{\gamma} = {\cal O}( |T_n|^{-\beta/(2\beta + d)})$, and   
   the rate bound in \eqref{e:rate-bounds} is
   $$
   {\cal O}\Big( (n\delta_n)^{-\beta d/(2\beta + d)} \Big) = {\cal O}\Big( n^{-\beta d(1-\alpha)/(2\beta+d)}\Big).
   $$
   \item {\em (coarse sampling)} When $0<\alpha<\alpha_{\beta,\gamma}$, then $|T_n|^{-\beta/(2\beta + d)}= {\cal O}(\delta_n^{\gamma})$ and the rate bound becomes
   $$
   {\cal O}( \delta_n^{\gamma} ) = {\cal O}( n^{-\alpha \gamma/2}).
   $$
   \end{itemize}  
   In the fine-sampling regime, the rate is the same as the minimax lower bound established in Theorem \ref{thm:PL minimax cont time}
   below. By \eqref{e:alpha_beta_gamma}, a larger $\gamma$ (i.e., a smoother $C$) leads to a wider range of sampling rates under which the minimax rate can be achieved by $\hat f_n(\theta)$. Similarly, a larger $d$ or larger $\beta$ (i.e., faster tail decay of $C$)  
   leads to a narrower range of sampling rates in order to achieve the minimax rate. 
  
   \end{remark}

\section{Minimax rates} \label{sec:minimax}

The minimax rates for the spectral density estimation problem have received some attention.  A few examples of such studies for times series include \cite{samarov1977lower}, \cite{bentkus1985rate}, and \cite{efromovich1998data}, among others.  The continuous-time setting, however, appears to have been less studied \citep[see, e.g.,][and the references therein]{ginovyan:2011}. To the best of our knowledge, results on 
minimax rates for the pointwise inference of the spectral density of functional time series or function-valued, continuous-time processes observed at discrete time points have not yet been established. Also, we are not aware of such results for random fields indexed by 
$\Z^d$ or $\R^d$, $d>1$. 

 Assuming $\{X(t)\}$ is Gaussian, below we extend the work of \cite{samarov1977lower} by focusing on the classes ${\cal P}_D(\beta, L)$ and ${\cal P}_C(\beta, L)$ considered in Section \ref{sec:discrete-time}. As in Section \ref{sec:discrete-time}, we assume the data are observed on a grid. 

Our first result is concerned with the case $\delta_n=1$, where, in accordance with Section \ref{s:ts}, we consider a discrete parameter process $\{X(t), t \in\mathbb{Z}^d\}$.

\begin{thm}\label{thm:pointwise rate time series}
Assume that $\{X(t), t \in\mathbb{Z}^d\}$ is a stationary Gaussian process with spectral density function $f$.  Let $\mathcal{M}_n$ be the class of all possible estimators $f_n$ of $f$ based on the observations $X(t), t\in \{1,\hdots,n\}^d$. Then, for any interior point $\theta_0\in(-\pi,\pi)^d$ and $\beta, L >0,$
\begin{align}\label{e:thm:pointwise rate time series}
\liminf_{n\to\infty} \inf_{f_n \in \mathcal{M}_n}\sup_{f\in {\cal P}_D(\beta,L)}\mathbb{P}\left(\left\|f_n(\theta_0)-f(\theta_0)\right\|_{\rm HS}\ge n^{-\frac{d\beta}{2\beta+d}}\right)>0,
\end{align}
where ${\cal P}_D(\beta,L)$ is defined in \eqref{def:PL minimax class discrete time}.
\end{thm}

\begin{remark} \label{rem:Zd-minimax}
Note that $|\bbT_n| = n^{d}$ for $\bbT_n=\{1,\hdots,n\}^d$. Hence, by Theorem \ref{prop:PL bias time series}, the estimator $\hat f_n(\theta_0)$ achieves 
the minimax rate $|\bbT_n|^{-\beta/(2\beta+d)} = n^{-(d\beta)/(2\beta+d)}$ uniformly over the class ${\cal P}_D(\beta,L)$.
Thus, in the setting of processes indexed by $\Z^d$, our estimators are rate-optimal in a uniform sense, for the power-law class and for all dimensions $d\ge 1$.
\end{remark}

\begin{proof}[ Proof of Theorem \ref{thm:pointwise rate time series} (Outline)] 
The detailed proof of Theorem \ref{thm:pointwise rate time series} is given in Section \ref{s:minimax_proof} of Supplement. 
We describe the key elements of the proof here. 
First, for any member $e_i$ of the real CONS, consider the (scalar) real-valued process $X_{e_i}(t) := \langle X(t), e_i\rangle_\bbH$ and let $C_{e_i}(x)$ and $f_{e_i}(\theta)$ be its stationary covariance and spectral density, respective. If $f\in {\cal P}_D(\beta,L)$ then 
$$
\int_{\R^d} (1+\|x\|_2^\beta) |C_{e_i}(x)| dx \le L\quad\mbox{and} \quad |\hat f_{e_i} (\theta_0) - f_{e_i}(\theta_0)| \le \|\hat f(\theta_0) - f(\theta_0)\|_{\rm HS}. 
$$
These follow from the simple fact that $|\langle {\cal A} \phi, \phi\rangle_\bbH|  \le \|{\cal A}\|_{\rm op}$ for any bounded linear operator ${\cal A}$ and unitary $\phi\in\bbH$. Thus, it suffices to prove Theorem \ref{thm:pointwise rate time series} by focusing on scalar, real-valued processes. 
The crucial step of the proof is constructing two functions $f_{0,n}, f_{1,n}$ in ${\cal P}_D(\beta,L)$ such that the distance between them accurately measures the complexity of the estimation problem.
Let  
\begin{align*} 
f_{0,n}(\theta)=L/2\cdot \mathbbm{1}(\theta\in[-\pi,\pi]^d).
\end{align*}
For $\theta = (\theta_i)_{i=1}^d\in \R^d$, define the function 
\begin{align*} 
g(\theta) = \epsilon \cdot \prod_{i=1}^d \varphi(\theta_i),\ \ 
\mbox{ where } \varphi(x) =  \exp\left(-\frac{1}{1-(x/\pi)^2}\right)\mathbbm{1}(|x| < \pi),\ \ x\in \R,
\end{align*}
for some $\epsilon>0$. Note that the so-called ``bump'' function $g$ is compactly supported and infinitely differentiable.
Consider 
\begin{align*} 
g_n(\theta) = h_n^{\beta}g\left(\frac{\theta-\theta_0}{h_n}\right),
\end{align*}
where $h_n = M \cdot n^{-d/(2\beta+d)}$ for some appropriate constant $M$.
Now, let  †
\begin{align*}
    f_{1,n}(\theta) & = f_{0,n}(\theta)+ [g_n(\theta)+g_n(-\theta)].
\end{align*}
Thus, the distance between $f_{0,n}(\theta)$ and $f_{1,n}(\theta)$ is $g_n(\theta)+g_n(-\theta)= {\cal O}(n^{-d\beta/(2\beta+d)})$.
We then apply Theorem 2.5(iii) in \cite{tsybakov2008introduction} to obtain the desired result by verifying the following:
\begin{enumerate} 
    \item[(1)] $f_{0n},f_{1n}\in {\cal P}_D(\beta,L)$;
   
    \item[(2)] $f_{1n}(\theta_0)-f_{0n}(\theta_0) =c_{\theta_0} n^{-d\beta/(2\beta+d)}$ for $n$ large enough, where 
    $c_{\theta_0}   = M^\beta (1+\mathbbm{1}(\theta_0=0))>0$;
    
     \item[(3)] $\sup_{n} {\rm KL}(\mathbb{P}_{1n},\mathbb{P}_{0n})<\infty$, where ${\rm KL}$ stands for the Kullback-Leibler divergence
      and $\mathbb{P}_{0,n}$ and $\mathbb{P}_{1,n}$ are probability distributions under $f_{0,n}$ and $f_{1,n}$ respectively.

\end{enumerate}
The most technically challenging part of the proof is the computation of ${\rm KL}(\mathbb{P}_{1,n},\mathbb{P}_{0,n})$ in part (3), which is accomplished by following and extending an approach introduced in 
\cite{samarov1977lower}.  For details, see Section \ref{s:minimax_proof} of Supplement. \end{proof}

The next result gives the minimax rate for the continuous-parameter Gaussian process whose covariance function belongs to 
${\cal P}_C(\beta,L)$, defined in \eqref{def:CTPL minimax class}.

\begin{thm}\label{thm:PL minimax cont time}
Let $\{X(t),\ t\in\mathbb{R}^d\}$ be a stationary Gaussian process with spectral density $f$. 
Let $\mathcal{M}_n$ be the class of all possible estimators $f_n$ of $f$ based on the observations $X(k\delta_n),k\in\{1,\hdots,n\}^d$. Then, for each $\theta_0\in\bbR^d$ and $\beta, L >0,$
\begin{align}\label{e:thm:PL minimax cont time}
\liminf_{n\to\infty}\inf_{f_n\in \mathcal{M}_n}\sup_{f\in {\cal P}_C(\beta,L)}\mathbb{P}\left(\|f_n(\theta_0)-f(\theta_0)\|_{\rm HS} \ge (n\delta_n)^{-d\beta/(2\beta+d)}\right)>0,
\end{align}
where ${\cal P}_C(\beta,L)$ is defined in \eqref{def:CTPL minimax class}.
\end{thm}

The proof of Theorem \ref{thm:PL minimax cont time} is similar to that of Theorem \ref{thm:pointwise rate time series} and 
is also included in Section \ref{s:minimax_proof} of Supplement. 
We conclude this section with several remarks.

 \begin{remark}  Comparing the minimax lower bounds in \eqref{e:thm:pointwise rate time series} and \eqref{e:thm:PL minimax cont time},
  one can interpret $(n \delta_n)^d$ as the ``effective'' sample size in the case of mixed-domain asymptotics: 
  $$ \delta_n\to 0\ \ \mbox{ and }\ \ n\delta_n\to \infty.$$
  \begin{itemize}
\item[1.]  Recall Remark \ref{rem:sampling-regimes} and observe that, in the fine sampling regime, the rate of $\hat f_n(\theta)$ obtained in \eqref{e:rate-bounds} matches the minimax lower bound in \eqref{e:thm:PL minimax cont time}.  To  the best of our knowledge, this is the first result on the minimax rate for spectral density estimation in a mixed domain setting.

\item[2.] An open problem is the construction of a narrower
  class ${\cal P}_C$, which reflects both the tail-decay of the auto-covariance (through $\beta$) and its smoothness (through $\gamma$) so that
  the upper- and lower-bounds on the rate of the estimators match in both the fine- and coarse-sampling regimes (cf. Remark \ref{rem:sampling-regimes}).
  \end{itemize}
  \end{remark}

\section{Asymptotic distribution} \label{s:CLT}

As in the previous two sections, we continue to consider the case of gridded data described by \eqref{e:lattice} and \eqref{def: f^ts}. The goal here is to present a central limit theorem 
for our spectral density estimator $\hat f_n(\theta)$ assuming that $\{X(t)\}$ is a stationary Gaussian process, where in this section we do not restrict $X$ to be real in $\bbH$. However, due to the technical nature of this topic, we will focus on the case $d=1$. As discussed 
in Remark \ref{re:Guyon parametrization}, for $d=1$ the normalization ${|\mathbb{T}_n\cap(\mathbb{T}_n-(t-s))|}$ in $\hat f_n(\theta)$ does not affect the rate. 
Thus, for convenience, we will eliminate that and consider instead
\begin{align}\label{def: f^ts_1}
\hat f_n(\theta) = \frac{\delta_n}{2\pi n}\sum_{i,j=1}^ne^{\mathbbm{i}(i-j)\delta_n\theta}X(\delta_n i)\otimes X(\delta_n j) K\left(\frac{i-j}{\Delta_n}\cdot \delta_n\right).
\end{align}
We will prove a central limit theorem for $\hat f_n(\theta)$ assuming that $\delta_n\to$ some $\delta_\infty\in [0,\infty)$ as $n\to\infty$. 
The time-series and mixed-domain cases are covered by $\delta_\infty=1$ and $0$, respectively.

Interestingly, the asymptotic distribution of 
$\hat f_n(\theta)$ involves the notion of pseudo-covariance. Recall that from \eqref{e:CXY_2} the pseudo-covariance function is defined as $\check C(h) = \E[X(t+h) \otimes \overline X(t)]$. 
In accordance with \eqref{def:H-f(theta)} and \eqref{e:spec_fold},
define the {\em pseudo-spectral density}:
$$
\check f(\theta) = \frac{1}{2\pi} \int_{\R} e^{-\ii \theta x} \check C(x) dx,\ \ \theta\in\R,
$$
and, for $\delta>0$, the {\em folded pseudo-spectral density}:
$$
\check f(\theta;\delta) := \frac{\delta}{2\pi}\sum_{k=-\infty}^\infty e^{-\ii k^\top \theta\delta} \check C(k\delta), \ \theta\in[-\pi/\delta,\pi/\delta].
$$
Note that $\check f(\theta)$ and $\check f(\theta;\delta)$ are well defined assuming that $\int_\bbR \|\check C(x)\|_\tr dx < \infty$ and
$\sum_{k=-\infty}^\infty \|\check C(k\delta)\|_\tr < \infty$, respectively. For convenience, also write $f(\theta;0) = f(\theta)$ and $\check f(\theta;0) = \check f(\theta)$.

Let now $\{e_j\}$ be an arbitrary fixed CONS of $\bbH$, and define
\begin{align*}
\begin{split}
C_{k,\ell}(t) & = \langle C(t) e_k, e_\ell\rangle = \E[\langle X(t), e_\ell\rangle \overline{\langle X(0), e_k\rangle}], \\
 \check C_{k,\ell}(t) & = \langle \check C(t) \overline {e_k}, e_\ell\rangle = \E[\langle X(t), e_\ell\rangle \langle X(0), e_k\rangle].
 \end{split}
\end{align*}
The following assumption will be needed for establishing the central limit theorem. 

\renewcommand\theassumption{CLT}
\begin{assumption}  \label{a:clt} 
Let the grid size $\delta_n$ and bandwidth $\Delta_n$ satisfy $\delta_n\to$ some $\delta_\infty\in [0,\infty)$ and $(n\delta_n)/\Delta_n\to\infty$.
Also, assume that there exist positive constants $L_n$ such that
$$
L_n \delta_n \to\infty,\ L_n/\Delta_n \to 0,
$$
and for which the following hold:
\begin{itemize}
    \item[(a)]  $\sup_n \delta_n\sum_{x=-\infty}^\infty \|C(\delta_n x)\|_\tr <\infty$ and $\delta_n\sum_{|x|>L_n} \|C(\delta_n x)\|_\tr \to 0$;
    
    \item[(b)] $\|f(\theta;\delta_n)-f(\theta;\delta_\infty)\|_\tr \to 0$;

    \item[(c)]  $\sup_n \delta_n\sum_{x=-\infty}^\infty \|\check C(\delta_n x)\|_\tr <\infty$ and $\delta_n\sum_{|x|>L_n} \|\check C(\delta_n x)\|_\tr \to 0$;
    
    \item[(d)]  $\|\check f(\theta;\delta_n)-\check f(\theta;\delta_\infty)\|_\tr \to 0$;
        
    \item[(e)] $\delta_n^2 \sum_{x_1,x_2=-\infty}^{\infty}|C_{k,k}(\delta_n x_1)|\cdot |C_{\ell,\ell}(\delta_n x_2)|\le a_{k,\ell},$ such that 
    $\sum_{k,\ell}a_{k,\ell}<\infty$;
    
    \item[(f)] $\delta_n^2 \sum_{x_1,x_2=-\infty}^{\infty}|\check C_{k,\ell}(\delta_n x_1)|\cdot |\check C_{k,\ell}(\delta_n x_2)|\le b_{k,\ell}$, such that $\sum_{k,\ell}b_{k,\ell}<\infty$.
\end{itemize}
\end{assumption}

Note that if $\delta_n=\delta_\infty\in (0,\infty)$ for all $n$, then the conditions (a)-(f) follow from $\sum_{x=-\infty}^\infty \|C(\delta_\infty x)\|_\tr <\infty$ and $\sum_{x=-\infty}^\infty \|\check C(\delta_\infty x)\|_\tr <\infty$. For $\delta_\infty=0$, the conditions (a) and (c) in the above assumption are related to the notion of directly Riemann integrability (dRi) \citep[cf., e.g.,][]{feller2008introduction}; if, in addition, $C(x)$ and $\check C(x)$ are functions in $\bbC$, then the dRi of $C(x)e^{\ii x\theta}$ and $\check C(x)e^{-\ii x\theta}$ also implies (b) and (d) respectively.

The following modified assumption on the kernel $K$ is also needed.

\renewcommand\theassumption{K$^{\boldsymbol{\prime}}$}
\begin{assumption} \label{a:kernel'}
The nonnegative kernel $K$ has compact support, is symmetric about $0$, and is of bounded variation. 
\end{assumption}

The following result is a central limit theorem for $\hat f_n(\theta)$, where the weak convergence is defined in the space $\bbX$ of Hilbert-Schmidt operators on $\bbH$.

\begin{thm}\label{thm:CLT} 
Consider the stationary zero-mean Gaussian process $\{X(t)$, $t\in\bbR\}$ and assume that Assumptions \ref{a:clt} and \ref{a:kernel'} hold. 
Define
\begin{align*}
\CT_n(\theta):= \sqrt{\frac{n\delta_n}{\Delta_n}}\left[\hat f_n(\theta)-\mathbb{E}\hat f_n(\theta)\right], \  \ \theta\in\bbR,
\end{align*}
where $\hat f_n(\theta)$ is given in \eqref{def: f^ts_1}.
Then, for any $\theta \in [-\pi/\delta_\infty, \pi/\delta_\infty]$, which is taken as $\bbR$ if $\delta_\infty=0$,
\begin{align*}
\CT_n(\theta) \stackrel{d}{\to}\CT(\theta) \ \mbox{ in $\bbX$},
\end{align*}
where $\CT(\theta)$ is a zero-mean Gaussian element of $\bbX$, such that for every finite collection $\{g_{\ell},\ell=1,\hdots,m\},$ and positive numbers $\{a_\ell, \ell=1,\hdots,m\},$ 
\begin{align} \label{e:clt_var}
\begin{split}
& {\rm Var}\left(\sum_{\ell=1}^ma_{\ell}\left\langle\CT(\theta) g_\ell,g_\ell\right\rangle\right) \\
& =\|K\|_2^2 \sum_{\ell_1,\ell_2=1}^ma_{\ell_1}a_{\ell_2}\left[\left|\left\langle f(\theta;\delta_{\infty})g_{\ell_2},g_{\ell_1}\right\rangle\right|^2 +c(\theta)\left|\left\langle \check f(\theta;
\delta_{\infty})\overline{g_{\ell_2}},g_{\ell_1}\right\rangle\right|^2\right],
\end{split}
\end{align}
where $\|K\|_2^2 = \int K^2(x) dx$, and $c(\theta)=I_{(\theta=0)}$ if $\delta_{\infty}=0$ and $I_{(\theta=0,\pm\pi/\delta_{\infty})}$ if $\delta_{\infty}>0$.
\end{thm}

\begin{remark} 
\begin{itemize}
\item[1.] Observe that the quantity $\sum_{\ell=1}^ma_{\ell}\left\langle\CT(\theta) g_\ell,g_\ell\right\rangle$ in \eqref{e:clt_var} is real since $K$ is assumed symmetric.

\item[2.] 
The variances in \eqref{e:clt_var} for all choices of $\{a_{\ell}\}$ and $\{g_{\ell}\}$ completely characterize the distribution of $\CT$. 
The expression $\left\langle \check f(\theta)\overline{g_{\ell_2}},g_{\ell_1}\right\rangle$ in \eqref{e:clt_var} does not depend on the choice of real CONS, since
\begin{align*} 
\langle \check C(t,s) \overline g, h\rangle = \E \langle X(t), h\rangle \langle \overline g, \overline{X(s)} \rangle = \E \langle X(t), h\rangle \langle X(s), g\rangle, \
g,h\in\H.
\end{align*}
\end{itemize}
\end{remark}

The proof of this result, given in Section \ref{s:clt_proof} in Supplement, is based on verifying the convergence of ``all moments'' of the estimator together with a tightness condition.

The previous result does not provide an explicit representation of the limit. In what follows, we obtain such an explicit, stochastic representation of $\CT(\theta)$ for $c(\theta)=0,$ where $c(\theta)$ as in \eqref{e:clt_var}. 
Define the complex Gaussian random variables $Z_{i,j}$'s as follows:
\begin{equation}\label{e:def Zij}
Z_{i,j} = \xi_{i,j} + \mathbbm{i} \eta_{i,j},\ \ i<j,
\end{equation}
where $\xi_{i,j}$ and $\eta_{i,j}$ are iid $N(0,1/2)$ and $Z_{j,i} := \overline{Z_{i,j}}$.  For $i=j$, we have
the $Z_{i,i}$'s are real and $N(0,1)$, independent from the $Z_{i,j}$'s, for $i\not=j$.
Then, one obtains that the $Z_{i,j}$'s are zero-mean complex Gaussian variables such that
\begin{equation}\label{e:Zij}
Z_{i,j} = \overline{Z_{j,i}} \ \ \ \mbox{ and }\ \  \ \E[ Z_{i,j} \overline {Z_{i',j'}} ] = \delta_{(i,j), (i',j')}.
\end{equation}

\begin{cor}\label{cor:CLT stoch rep} Let $c(\theta)=0$ in \eqref{e:clt_var} and assume the conditions of Theorem \ref{thm:CLT}. Let $\{e_i(\theta)\}$ be the (not necessarily real) CONS 
diagonalizing $f(\theta)$, i,.e.,
\begin{align*}f(\theta)=\sum_i\lambda_i (\theta) e_i(\theta)\otimes e_i(\theta).\end{align*}
The random variable $\CT(\theta)$ has the stochastic representation 
\begin{align}\label{e:KL_T}
\CT(\theta) \stackrel{d}{=} \|K\|_2 \sum_{i,j}\sqrt{\lambda_i(\theta)\lambda_j(\theta)}Z_{i,j}e_i(\theta)\otimes e_j(\theta),\end{align} 
where $Z_{i,j}$ as defined in \eqref{e:def Zij}.
In particular, the covariance operator of $\CT(\theta)$ is
\begin{align*}\mathbb{E}[ \CT(\theta)\otimes_{\rm HS}\CT(\theta) ] = \|K\|_2^2 \sum_{i,j}\lambda_i(\theta)\lambda_j(\theta)(e_i(\theta)\otimes e_j(\theta) )\otimes_{\rm HS}(e_i(\theta)\otimes e_j(\theta)).\end{align*}
\end{cor}

\begin{proof} Let $g_\ell,\ \ell = 1,\cdots,m$ be arbitrary in $\bbH$ and suppose
$$
g_{\ell} = \sum_{i} x_i(\ell) e_i,\ \ x_i(\ell) \in \mathbb C.
$$
Then, by Theorem \ref{thm:CLT}, it is enough to verify that the representation of $\CT$ in \eqref{e:KL_T} satisfies
\begin{align}\label{e:Var}
{\rm Var}\Big( \sum_{\ell =1}^m a_\ell \langle \CT g_\ell, g_\ell \rangle \Big) = \|K\|_2^2 \sum_{\ell_1,\ell_2} a_{\ell_1}a_{\ell_2}
|\langle f(\theta) g_{\ell_2} ,g_{\ell_1}\rangle |^2,
\end{align}
for real constants $a_{\ell}\in\bbR, \ell=1,\hdots,m.$ Observe that 
$$
 \langle \CT g_\ell, g_\ell \rangle = \|K\|_2 \sum_{i,j} \sqrt{\lambda_i\lambda_j} Z_{i,j} x_i(\ell) \overline{x_{j}(\ell)}.
$$
Thus, in view of \eqref{e:Zij}, the LHS of \eqref{e:Var} equals
\begin{align}\label{e:Var-1} 
\begin{split}
& \|K\|_2^2 \sum_{\ell_1,\ell_2} a_{\ell_1}a_{\ell_2}\sum_{i,i',j,j'} 
x_i(\ell_1) \overline{x_{j}(\ell_1)} \overline{x_{i'}(\ell_2)} {x_{j'}(\ell_2)} \sqrt{\lambda_i\lambda_j \lambda_{i'}\lambda_{j'}}
 \E [Z_{i,j} \overline{Z_{i',j'}}]\\
 & \quad\quad = \|K\|_2^2\sum_{\ell_1,\ell_2,i,j} a_{\ell_1}a_{\ell_2}\lambda_i\lambda_j x_i(\ell_1)x_j(\ell_2) \overline{x_j(\ell_1)} \overline{x_i(\ell_2)}.
 \end{split}
\end{align}
The latter expression is the RHS of \eqref{e:Var}.  On the other hand,
\begin{align*}
\langle f(\theta) g_{\ell_2} ,g_{\ell_1}\rangle = \sum_{i} \lambda_i x_i(\ell_1)\overline{x_i(\ell_2)}.
\end{align*}
Thus, it is easy to see that that the right-hand sides of \eqref{e:Var} and \eqref{e:Var-1} are the same.
\end{proof}

We end this section with the following remark.

\begin{remark} Observe that $\CT(\theta),$   for $c(\theta)=0$ in \eqref{e:clt_var}, is a zero-mean random element in the Hilbert space $\mathbb X$ of Hilbert-Schmidt operators.  Therefore, Relation \eqref{e:KL_T} provides
its Karhunen-Lo\'eve type representation. That is, the covariance operator of $\CT(\theta)$ is diagonalized in the basis 
$e_{i,j}(\theta):= e_i(\theta)\otimes e_j(\theta),\ (i,j)\in \mathbb{N}^2$, where $\{e_i(\theta)\}$ is the CONS of $\H$ diagonalizing the operator $f(\theta)$.  
The eigenvalues of the covariance operator $\E[ \CT(\theta) \otimes \CT(\theta)]$ are precisely $\lambda_{i,j}(\theta) := \lambda_i(\theta) \lambda_j(\theta)$, where the $\lambda_i(\theta)$'s are the eigenvalues of $f(\theta)$.   
\end{remark}

\section{An RKHS formulation based on discretely-observed functional data} \label{s:rkhs}
 In this section, we specialize the obtained results for an abstract Hilbert space to the case where
$\mathbb H$ is a space of functions.  In real-data applications complete functions 
are not available and instead each of the functional data $X(t_i)$ is observed on a finite set of points. A natural space for this setting may be when $\bbH$ is a reproducing kernel Hilbert space (RKHS). Unlike the more commonly considered space $L^2[a,b]$, an RKHS $\bbH$ allows us to view $\bbH$-valued random elements as bona fide functions,
since the point-evaluation functionals are well-defined and continuous. This enables a seamless interface between the theory that we have developed up to this point and applications based on discretely observed data. The literature on RKHS is extremely rich. For a quick overview on the role of RKHS in functional data analysis, the reader is referred to \cite{hsing2015theoretical}. 

Let $\bbH$ be a reproducing kernel Hilbert space (RKHS) containing functions on a compact set $E$, where the kernel $R(\cdot,\cdot)$ is continuous on $E\times E$. The reproducing property states that
$$
g(u) = \langle g, R(u,\cdot)\rangle_\bbH, \ u \in E.
$$ 
Now, let $\{X(t),\ t\in \R^d\}$ be a stationary $\bbH$-valued process with covariance function $C$ and spectral density
$f$.  Then, it can be viewed as a bivariate stochastic process 
$
\{X(u, t):=\langle X(t), R(u,\cdot)\rangle_\bbH, u\in E, t \in\R^d\}. 
$
We have
\begin{align} \label{e:cross_sdf}
\begin{split}
\cov(X(u,t+h), X(v,t))  &= \langle C(h) R(u,\cdot), R(v,\cdot) \rangle_{\bbH} \\
&= \int_{\mathbb{R}^d}e^{-\mathbbm{i}h^\top\theta}  f_{u,v}(\theta) d\theta,
\end{split}
\end{align}
where
\begin{align*} 
f_{u,v}(\theta) = \langle f(\theta) R(u,\cdot), R(v,\cdot) \rangle_{\bbH}.
\end{align*}
In view of \eqref{e:cross_sdf}, it may be convenient to refer to $f_{u,v}(\theta)$ as a spectral density. However,
there is no guarantee that it is nonnegative for $u\not=v$.
By the Cauchy-Schwartz inequality, our estimation rates on the operator $f(\theta)$ translate immediate to 
$f_{u,v}(\theta)$ for all $u,v$.

Assume that the process is observed on a common discrete set of points $D_n = \{u_{n,j}, j = 1, \ldots, m_n\}$ for all $t\in\mathbb{T}_n$. To relate the partially observed functional data to complete functional data in $\bbH$, a possible approach is the following. Assume that the matrix 
\begin{equation}\label{e:Rn-matrix}
\boldsymbol{R}_n:=\{R(u_{n,i}, u_{n,j})\}_{i,j=1}^{m_n}
\end{equation} is invertible for each $n$. Let $\bbH_n$ be the subspace of $\bbH$
spanned by $\{R(u, \cdot), u\in D_n\}$ and $\Pi_n$ is the projection operator onto $\bbH_n$. Then, for any $g\in\bbH$,
$$
\tilde g :=\Pi_n g
$$
interpolates $g$ at the points in $D_n$ and is in fact the minimum norm interpolant of $g$ on $D_n$; see \cite{wahba:1990} 
or Proposition \ref{prop:rkhs}.

The covariance of the stationary process $\{\wt X(t)\}$ is $\wt C(h) := \Pi_n C(h)\Pi_n$. First note that 
\begin{align*} 
\|\wt C(h)\|_{\rm tr}  \le \|C(h)\|_{\rm tr}. 
\end{align*}
This follows from Lemma \ref{le:trace norm alt def} (i), since $\langle \wt C(h), {\cal W} \rangle_{\rm HS} = \langle C(h),\wtilde {\cal W} \rangle_{\rm HS}$, where $\wtilde {\cal W} = \Pi_n {\cal W} \Pi_n$ is unitary for every unitary ${\cal W}$.
Thus, the condition $\int \|C(h)\|_{\rm tr} dh < \infty$ ensures that the spectral density $\tilde f$ of $\{\wt X(t)\}$ is well defined, and satisfies
$$
\tilde f(\theta) = \frac{1}{(2\pi)^d}\int_{\mathbb{R}^d}e^{\mathbbm{i}h^{\top}\theta}\wt C(h)dh.
$$ 
Following the approach in \eqref{def:f^K(theta)} based on the data $\wt X(t_i)$, define 
$$
\tilde f_n(\theta) = \Pi_n \hat f_n(\theta) \Pi_n.
$$
Consider the estimation of $\tilde f$ by $\tilde f_n$. To keep the presentation simple we focus on the Gaussian case. 
The following result follows readily from Theorem  \ref{thm:consistency f^K}.

\begin{thm}\label{cor:consistency for RKHS Gaussian}
Let the process $\left\{X(t), t\in\mathbb{R}^d\right\}$ be a zero-mean stationary Gaussian process taking values in $\bbH$. Suppose that Assumptions \ref{a:cov}, \ref{a:kernel}, and \ref{a:sampling} hold. 
If, additionally we have 
\begin{align*}
\Delta_n\cdot S_K\subset T_n-T_n \hspace{.2cm}\mbox{for all $n$.}
\end{align*}
Then, for any bounded set $\Theta$,
\begin{align} \label{e:f_tilde_rate}
\sup_{\theta\in \Theta}\Big(\mathbb{E}\Big\|\tilde f_n(\theta)- \tilde f(\theta){\Big\|}_{\rm HS}^2 {\Big)}^{1/2}
 = \mathcal{O}\left(\delta_n^{\gamma}+B_1(\Delta_n)+B_2(\Delta_n)+\sqrt{\frac{\Delta_n^d}{|T_n|}}\right),
\end{align} 
as $n\to\infty$. 
\end{thm}

Note that, in \eqref{e:f_tilde_rate}, we bounded $\wt B_1(\Delta_n), \wt B_2(\Delta_n)$, the counterparts of $B_1(\Delta_n), B_2(\Delta_n)$ where $C(h)$ therein is replaced by $\wt C(h)$, by $B_1(\Delta_n), B_2(\Delta_n)$, respectively. This is achieved using the simple fact that $\|\mathcal{T}_1\mathcal{T}_2\|_{\rm HS} \le \|\mathcal{T}_1\|\|\mathcal{T}_2\|_{\rm HS}$ where $\|\mathcal{T}_1\|$ stands for the operator norm of $\mathcal{T}_1$.
In view of Theorem \ref{cor:consistency for RKHS Gaussian}, to find the rate of 
$\mathbb{E} \|\tilde f_n(\theta)- f(\theta)\|_{\rm HS}^2$, it is sufficient to consider the bias $\|\tilde f(\theta)-f(\theta)\|_{\rm HS}$, 
which must be evaluated case by case, depending on the type 
of a RKHS being considered. Below, we consider an example that leads to a specific rate.

Consider the Sobolev space $\bbH=W_1[0,1]$ which consists of functions on the interval $[0,1]$ of the form $c + \int_0^1 (t\wedge u) h(u)du, c\in\R$ and $h$ integrable \citep[cf.][]{wahba:1990}. The inner-product in this space is
$\langle f,g\rangle_{\bbH}:= f(0)g(0) + \int_0^1 f'(t) g'(t) dt$, yielding the norm
$$
\|g\|_\bbH^2 = g(0)^2 + \int_0^1 \left(g'(t)\right)^2 dt.
$$
In context, we can state the following result for $\mathbb{E} \|\tilde f_n(\theta)- f(\theta)\|_{\rm HS}^2$.

\begin{thm} \label{thm:RKHS_bias}
Let the positive trace-class operator $f(\theta)$ have the eigen decomposition:
$$
f(\theta) = \sum_{j=1}^\infty \nu_j \phi_j\otimes\phi_j,
$$
where the eigenvalues $\nu_j$ are summable (since $f(\theta)\in\bbT_+$). Assume that, for each $j$, the derivative
$\phi_j'$ is Lipschitz continuous with $|\phi_j'(s)-\phi_j'(t)|\le C_j|s-t|$ for some finite constant $C_j$ where $\sum_{j=1}^\infty C_j \nu_j^2 < \infty$. Also, assume that the sampling design is $u_{n,i} = i/m_n, 0\le i\le m_n$. Then, 
$$
\|\tilde f(\theta)-f(\theta)\|_{\rm HS} = {\cal O}(m_n^{-1/2}).
$$
\end{thm}

The proof of Theorem \ref{thm:RKHS_bias} is given in Section \ref{s:RKHS_proof}.

\section{Related work and discussions}\label{sec:comparisons}
In this section we highlight the approaches in \cite{panaretos2013fourier} and \cite{zhu2020higher} focusing on the time-series setting and we explain how they relate to our approach. 

\subsection{Relation to flat-top kernel estimators} \label{subsec:flat-top}
The flat-top kernel estimators have been advocated in the works of \cite{politis2011higher,zhu2020higher}, among others. 
According to Relation (15) of \cite{zhu2020higher} the alternate estimator proposed in Section 3.1 therein takes the form 
$$ \frac{1}{2\pi}\sum_{|u|<T}\lambda(B_Tu)\hat r_u(\tau,\sigma)e^{-\mathbbm{i}\omega u},$$
where 
\begin{align*}
\hat r_u(\tau,\sigma)  = \frac{1}{T}\sum_{0\leq t,t+u\leq T-1}X_{t+u}(\tau)X_t(\sigma)\ \ \mbox{ and }\ \ 
\lambda(s)  = \int_{-\infty}^{\infty}\Lambda(x)e^{-\mathbbm{i}sx}dx
\end{align*}
for some $\Lambda(x)$. 
In the time-series setting with $d=1$, an asymptotically equivalent adaptation of our estimator in 
\eqref{def: f^ts} is given by:
\begin{align}\label{e:f-ts}
\hat f_T(\theta)
 & = \frac{1}{2\pi}\cdot\frac{1}{T}\sum_{|u|<T}K\left(\frac{u}{\Delta_T}\right)\cdot e^{-\mathbbm{i}u\theta}\sum_{0\le t,t+u\le T-1}X_{t+u}\otimes X_t.    
\end{align}
See Remark \ref{re:Guyon parametrization}.
Thus, the two estimators are essentially the same, with $\omega$ corresponds to $\theta$,  $B_T$ to $1/\Delta_T$, and $\lambda$ to $K$. 
\cite{zhu2020higher} focuses on $p$-times differentiable flat-top kernels $\lambda$ with $\lambda(t) = 1$, for all $\|t\|\le \epsilon$, for some $\epsilon>0$, where $p$ is adapted to the tail decay of the covariance function. Such kernels reduce the bias of the kernel spectral density estimator in essentially the same way as do the kernels $K$ satisfying \eqref{e:Kernel differentiability} in the present paper. One can get a rough idea about that by the crude calculations in \eqref{e:bias_estimate}.

 Moreover, in Section 5 of \cite{zhu2020higher}, an effective data-dependent choice of the bandwidth parameter $B_T$ is developed. The authors base their 
 selection on the functional version of correlogram/cross-correlogram.  Using this quantity, an empirical rule is proposed for the choice of $B_T$.  In practice, we recommend using
 flat-top kernels and a similar methodology for the selection of $\Delta_T=1/B_T$. The thorough investigation of the data-driven, adaptive choice of $\Delta_T$ in our setting
 of irregularly sampled data, however, merits further theoretical and methodological investigation.
 
\subsection{Periodogram-based estimators for functional time series}\label{subsec:panaretos-tavakoli}
The seminal work of \cite{panaretos2013fourier} considers function-valued time series, taking values in $(L^2[0,1],\mathbb{R})$.  They develop comprehensive theory
and methodology for inference of the spectral density operator extending the classic periodogram-based approach to the functional time series setting.  
The proposed estimator therein is: 
\begin{align}\label{def: panaretos-tavakoli estimator.}
f_{\omega}^{(T)}(\tau,\sigma) = \frac{2\pi}{T}\sum_{s=1}^{T-1}W^{(T)}\left(\omega-\frac{2\pi s}{T}\right)p_{2\pi s/T}^{(T)}(\tau,\sigma),    
\end{align} 
where $$W^{(T)}(x) = \sum_{j\in\mathbb{Z}}\frac{1}{B_T}W\left(\frac{x+2\pi j}{B_T}\right),$$
with $W$ being a taper weight function of bounded support. Here,
$$
p_{\omega}^{(T)}(\tau,\sigma) = \widetilde X_{\omega}^{(T)}(\tau)\widetilde X_{-\omega}^{(T)}(\sigma)
$$
is the periodogram, where
$$
\widetilde X_{\omega}^{(T)}=\frac{1}{\sqrt{2\pi T}}\sum_{t=0}^{T-1}X_t(\tau)e^{-\mathbbm{i} \omega t},
$$
is the discrete Fourier transform (DFT). This is referred to as the smoothed periodogram estimator \citep[cf.][]{robinson1983review}.

The asymptotic properties of these periodogram-based estimators are studied using the following
general cumulant-based assumptions: \\

\noindent\textit{Condition $C(\ell,k)$.} For each $j=1,\hdots,k-1,$ 
$$
 \sum_{t_1,\hdots,t_{k-1}=-\infty}^{\infty}(1+|t_j|^{\ell})\|{\rm cum}(X_{t_1},\hdots,X_{t_{k-1}},X_0)\|_2<\infty.
$$
For example, by Theorem 3.6 in \cite{panaretos2013fourier}, if $C(1,2)$ and $C(1,4)$ hold, the mean 
squared error of  $f_{\omega}^{(T)}(\cdot,\cdot)$ for $\omega\not=0,\pm\pi$
is:
$$
\mathbb{E}\|{\cal F}_{\omega}^{(T)}-{\cal F}_{\omega}\|_{\rm HS}^2 =  \mathcal{O}\Big(B_T^2 + B_T^{-1}T^{-1} \Big),
$$
where ${\cal F}_\omega^{(T)}$ and ${\cal F}_\omega$ are the operators with kernels $f_{\omega}^{(T)}$ and $f_{\omega}$, respectively. The rate-optimal choice
of $B_T$ is $T^{-1/3}$, which yields the bound on the rate of consistency of the estimator ${\cal O}(T^{-1/3})$.  \

Our results provide more detailed estimates on the rates under simple structural assumptions on the covariances. 
Indeed, observe that the condition $C(1,2)$ corresponds to our condition ${\cal P}_D(\beta,L)$ with $\beta=1$ in \eqref{def:PL minimax class discrete time}. 
Our Theorem \ref{prop:PL bias time series} (see Relation \eqref{e:mse_ts_PLD} with $d=1$) yields the rate 
of consistency bound of ${\cal O}(T^{-\beta/(2\beta+1)})$, which for $\beta=1$ matches the rate-optimal bound in \cite{panaretos2013fourier}.  
Our condition \eqref{def:PL minimax class discrete time}, however,  allows for a wider range of covariance structures than Condition $C(1,2)$, where we allow 
for $\beta>0$ to be less than $1$. As discussed in Section \ref{sec:minimax}, the rate  ${\cal O}(T^{-\beta/(2\beta +1)})$ is minimax optimal in the class ${\cal P}_D(\beta,L)$.

As observed in Section 3 of \cite{zhu2020higher}, one can relate the time-domain and frequency-domain (periodogram-based) estimators.  Indeed,
one can argue that our estimator in \eqref{e:f-ts} corresponds asymptotically to the periodogram-based estimator in \eqref{def: panaretos-tavakoli estimator.} with taper 
\begin{equation*}
 W(x) = \frac{1}{2\pi} \int_{-\infty}^\infty e^{\ii t x} K(t) dt,
\end{equation*}
where $\Delta_T\sim 1/B_T$.  In this case, we have  $2\pi \|W\|_2^2 = \|K\|_2^2$ and the asymptotic covariances of
the estimators in \eqref{e:f-ts} and \eqref{def: panaretos-tavakoli estimator.} are identical \citep[compare, e.g., Theorem 3.7 of][and our Corollary \ref{cor:CLT stoch rep}]{panaretos2013fourier}.  Theorem 3.7 in \cite{panaretos2013fourier} establishes the asymptotic normality of the periodogram-based estimators under conditions $C(1,2)$ and $C(1,4)$, as well as $C(0,k)$, for all $k\ge 2$.  In Theorem \ref{thm:CLT} we adopt the stronger assumption that the underlying process is Gaussian. We establish, however, the asymptotic normality of our estimators under milder tail-decay conditions on the operator 
covariance and pseudo-covariance functions.


%
%

\section*{Acknowledgements}
The authors would like to thank Zakhar Kabluchko for his help in understanding the properties of the intrinsic volumes. The first author was 
supported by the Onassis Foundation - Scholarship ID: F ZN 028-1 /2017-2018.
%
%
 


\bibliographystyle{imsart-nameyear} 
\bibliography{references.bib}       

\addcontentsline{toc}{section}{Supplement}
\vskip.5cm\noindent
\begin{center}
{\sc Supplement} 
\end{center}

\renewcommand{\theequation}{S.\arabic{section}.\arabic{equation}}
\renewcommand{\thesection}{S.\arabic{section}}
\renewcommand{\theremark}{S.\arabic{section}.\arabic{remark}}
\allowdisplaybreaks

The Supplement contains the detailed proofs and some remarks. 
In order to differentiate sections and results in the supplement from  those in the main paper, we add a character ``S" in front of sections, 
	lemmas, etc., in the Supplement.  
 
 For convenience, we will use the notation $a_n \le_c b_n$ to mean that there is a finite positive constant $B$ for which $a_n\le B b_n$ for all $n$.


\setcounter{section}{0}

\section{Proofs for Section \ref{s:asymp}}

We begin by recalling some key notation.  The spectral density of the $\mathbb{H}$-valued second order
stationary process $X = \{X(t),\ t\in \R^d\}$ is:
\begin{align}\label{s:def:H-f(theta)}
f(\theta):=\frac{1}{(2\pi)^d}\int_{\mathbb{R}^d}e^{\mathbbm{i}h^{\top}\theta}C(h)dh,\ \ \theta\in \R^d,
\end{align}
where the last integral is understood in the sense of Bochner, and where $C(t) = \E[ X(t)\otimes X(0) ]$ is the operator auto-covariance
function of $X$.

The estimator of the spectral density is defined as:
\begin{align}\label{s:def:f^K(theta)}
\begin{split}
    \hat f_n(\theta) & = \frac{1}{(2\pi)^d}\sum_{t\in\mathbb{T}_n}\sum_{s\in\mathbb{T}_n}e^{\mathbbm{i}(t-s)^{\top}\theta}\frac{X(t)\otimes X(s)}{|T_n\cap(T_n-(t-s))|} \\
&  \hspace{2cm}  \cdot K\left(\frac{t-s}{\Delta_n}\right)\cdot|V(t)|\cdot|V(s)|. 
\end{split}
\end{align}
Introduce also the auxiliary, idealized estimator based on the continuously sampled path $\{X(t),\ t\in T_n\}$:
\begin{align}\label{s:def:hat f_n}
    g_n(\theta)=\frac{1}{(2\pi)^d}\int_{t\in T_n}\int_{s\in T_n}e^{\mathbbm{i}(t-s)^\top\theta}\frac{X(t)\otimes X(s)}{|T_n\cap(T_n-(t-s))|} 
    K\left(\frac{t-s}{\Delta_n}\right)dtds.
\end{align}

\subsection{Proof of Theorem \ref{thm:expectation f_n^K-f_n}} \label{s:proof_of_bias}

We begin by recalling the statement.

\begin{thm}[Theorem \ref{thm:expectation f_n^K-f_n}]\label{s:thm:expectation f_n^K-f_n}
Let Assumptions \ref{a:cov}, \ref{a:kernel}, and \ref{a:sampling} hold and suppose $\delta_n\vee |T_n|^{-1}\to 0$.
Choose $\Delta_n\to\infty$ such that 
\begin{align*} 
\Delta_n\cdot S_K\subset T_n-T_n \hspace{.2cm}\mbox{for all $n$, }
\end{align*}
where $A-B:=\{a-b:\ a\in A,b\in B\}$ for sets $A,B\subset \mathbb{R}^d$.
Then, for any bounded set $\Theta \subset \R^d$, we have 
\begin{align}\label{s:e:thm:expectation f_n^K-f_n rate}
    \sup_{\theta\in\Theta}\left\Vert\mathbb{E}\hat f_n(\theta) - f(\theta)\right\Vert_{\rm HS} 
     = \mathcal{O}\left(\delta_n^{\gamma} +B_1(\Delta_n)+B_2(\Delta_n)\right),
\end{align}
where \begin{align*}
    B_1(\Delta_n) & := \left\Vert\int_{h\in\Delta_n\cdot S_K}e^{\mathbbm{i}h^{\top}\theta}C(h)\left(1-K\left(\frac{h}{\Delta_n}\right)\right)dh\right\Vert_{\rm HS}, \\
    B_2(\Delta_n)&:= \left\Vert\int_{h\not\in \Delta_n\cdot S_K}e^{\mathbbm{i}h^{\top}\theta}C(h)dh\right\Vert_{\rm HS}.
\end{align*}
\end{thm}

\begin{proof} By the triangle inequality,
\begin{align*}
    \left\Vert\mathbb{E}\hat f_n(\theta)-f(\theta)\right\Vert_{\rm HS}\leq \left\Vert\mathbb{E}\hat f_n(\theta)-\mathbb{E}g_n(\theta)\right\Vert_{\rm HS}+\left\Vert\mathbb{E}g_n(\theta)-f(\theta)\right\Vert_{\rm HS}.
\end{align*}
It is immediate from \eqref{s:def:H-f(theta)} for $f$ and the inclusion
$\Delta_n\cdot S_K\subset T_n-T_n$, that
\begin{equation*} 
\|\mathbb{E}g_n(\theta)-f(\theta)\|_{\rm HS} \le B_1(\Delta_n) + B_2(\Delta_n).
\end{equation*}
To complete the proof one needs to show that
\begin{equation}\label{s:e:Efhat-Eg}
 \|\mathbb{E}\hat f_n(\theta)-\mathbb{E}g_n(\theta)\|_{\rm HS} = {\cal O}(\delta_n^\gamma).
 \end{equation}
To evaluate $ \|\mathbb{E}\hat f_n(\theta)-\mathbb{E}g_n(\theta)\|_{\rm HS}$, first denote the integrand in 
\eqref{s:def:hat f_n} by
\begin{align*} 
h_n(t,s;\theta):=e^{\mathbbm{i}(t-s)^{\top}\theta}\frac{X(t)\otimes X(s)}{|T_n\cap(T_n-(t-s))|}K\left(\frac{t-s}{\Delta_n}\right).
\end{align*}
In view of \eqref{s:def:f^K(theta)} and since $|V(w)|\cdot |V(v)| = \int_{t\in V(w)}\int_{s\in V(v)} \mathbbm{1}_{\left(t\in V(w),s\in V(v)\right)} dtds$,
this allows us to write:
\begin{align*} 
\begin{split}
& g_n(\theta)-\hat f_n(\theta) \\
& =\frac{1}{(2\pi)^d}\sum_{w\in\mathbb{T}_n}\sum_{v\in\mathbb{T}_n}\int_{t\in V(w)}\int_{s\in V(v)} \left(h_n(t,s;\theta) - h_n(w,v;\theta)\right)   dtds.
\end{split}
\end{align*}
This implies that  
\begin{align}\label{s:ineq: Exp dif before kernel} 
\begin{split}
& \left\Vert\mathbb{E} g_n(\theta)-\mathbb{E}\hat f_n(\theta)\right\Vert_{\rm HS} \\
& \le \frac{1}{(2\pi)^d}\sum_{w\in\mathbb{T}_n}\sum_{v\in\mathbb{T}_n}  \mathop{\int}_{t\in V(w)}\mathop{\int}_{s\in V(v)}\|\mathbb{E}h_n(t,s;\theta)-\mathbb{E}h_n(w,v;\theta)\|_{\rm HS}dtds.
\end{split}
\end{align}
In the rest of the proof we will make use of the smoothness of $K$ and $C$, and routine but technical analysis to 
show that the last sum is of order ${\cal O}(\delta_n^\gamma)$. This will 
yield \eqref{s:e:Efhat-Eg} and complete the proof of \eqref{s:e:thm:expectation f_n^K-f_n rate}.

Recall that $S_K$ denotes the bounded support of the kernel function $K$. This means that \begin{equation*}
K\left(\frac{t-s}{\Delta_n}\right)=0,\ \textrm{whenever}\ t-s\not\in \Delta_n\cdot S_K.    
\end{equation*} 
In each integral in the sums of \eqref{s:ineq: Exp dif before kernel} we have that $t\in V(w)$ and $s\in V(v).$ Thus, $$t-s=t-w+w-v+v-s\in w-v+B(0,2\delta_n),$$
where we used that $\max\{\|t-w\|,\|s-v\|\}\le \delta_n$, by the definition of $\delta_n$ \eqref{def:delta_n} and $B(0,r) = \{x\in\R^d\, :\, \|x\|_2<r\}$.
 
By \eqref{s:ineq: Exp dif before kernel}, we have
\begin{align}\label{e:In-Jn}
\|\E g_n(\theta) - \E \hat f_n(\theta) \|_{\rm HS} & \le 
\frac{1}{(2\pi)^d}\sum_{w\in\mathbb{T}_n}\sum_{v\in\mathbb{T}_n} \iint_{t,s\in T_n} 
 \|C(t-s) - C(w-v)\|_{\rm HS} | L_n(w-v)|  dtds \nonumber\\
 & + \frac{1}{(2\pi)^d}\sum_{w\in\mathbb{T}_n}\sum_{v\in\mathbb{T}_n} \iint_{t,s\in T_n} 
  \|C(t-s)\|_{\rm HS} |L_n(t-s) - L_n(w-v)|dsdt\nonumber\\
 &=: I_n + J_n,
\end{align}
where 
$$
L_n(x):= \frac{e^{\ii x^\top \theta}}{|T_n\cap(T_n-x)|} K\Big(\frac{x}{\Delta_n}\Big),\ \ x\in\R^d.
$$
Observe that since $K(x/\Delta_n) =0$ for $x\not\in \Delta_n\cdot S_K$,
and $\Delta_n^d/|T_n|\to 0$, Lemma \ref{le:rate for geometry quantity} implies that 
$|T_n\cap(T_n-x)|\sim |T_n|$ uniformly in $x\in \Delta_n\cdot S_K$.  This  
and the boundedness of the kernel $K$ imply
\begin{equation}\label{e:Ln}
\sup_{x\in \R^d} | L_n(x) | ={\cal O}\Big(\frac{1}{|T_n|}\Big).
\end{equation}

Recall that by \eqref{e:C-L1-Holder} in Assumption \ref{a:cov}, we have
\begin{equation}\label{e:L1-loc-Lipshits} 
\int_{\R^d} \sup_{y\in B(x,\delta)} \|C(y) - C(x)\|_{\rm tr} dx \le \vertiii{C}_\gamma \cdot \delta^\gamma,\ \ \delta\in (0,1)
\end{equation}
Thus, for the term $I_n$ in \eqref{e:In-Jn}, using Relations \eqref{e:Ln} and \eqref{e:L1-loc-Lipshits}, and the change 
of variables $x:= t-s$, 
we obtain
\begin{align*}
I_n &\le \frac{1}{(2\pi)^d} \int_{t\in T_n} \int_{s\in T_n}\sup_{\substack{\|t'-t\|\le \delta_n\\ \| s'-s\|\le \delta_n}}
 \|C(t-s) -  C(t'-s')\|_{\rm HS}\cdot \sup_{\tau\in \R^d} |L_n(\tau)|  dt ds\\
  &\le_c \frac{1}{|T_n|}  \int_{s\in T_n} \Big(\int_{x\in s+T_n} \sup_{\ \ y\, :\, \|x-y\|\le 2\delta_n} \| C(x) - C(y)\|_{\rm HS} dx \Big) ds\\
  & \le_c \frac{1}{|T_n|} \int_{s\in T_n} \vertiii{C}_\gamma \cdot \delta_n^\gamma ds = {\cal O}(\delta_n^\gamma).
\end{align*}

Next, focus on the term $J_n$ in \eqref{e:In-Jn}. We will show below that 
\begin{equation}\label{e:Lnx-Lny}
\sup_{\| x- y\| \le 2\delta_n} |L_n(x) - L_n(y)| = {\cal O}\Big(\frac{\delta_n}{|T_n|}\Big).
\end{equation}
Thus, recalling that $\|t-s - (w-v) \|\le 2\delta_n$, whenever $t\in V(w)$ and $s\in V(v)$, Relation \eqref{e:Lnx-Lny} for
the term $J_n$ in \eqref{e:In-Jn} implies
$$
J_n\le_c \frac{\delta_n}{|T_n|} \int_{t\in T_n} \int_{s\in T_n} \|C(t-s)\|_{\rm HS} dt ds = {\cal O}(\delta_n),
$$
where the last relation follows from a change of variables $x:=t-s$ and Assumption \ref{a:cov} (a).\\

To complete the proof, it remains to establish \eqref{e:Lnx-Lny}.  By adding and subtracting terms, we obtain
\begin{align}\label{e:Ln-ABC}
|L_n(x) - L_n(y)| &\le |K(x/\Delta_n)| \Big| \frac{1}{|T_n\cap (T_n-x)|} - \frac{1}{|T_n\cap (T_n-y)|} \Big| \\
& + \frac{|K(x/\Delta_n) - K(y/\Delta_n)|}{|T_n\cap (T_n-y)|}   +  \frac{|K(y/\Delta_n)|}{|T_n\cap (T_n-y)|} \Big| e^{\ii x^\top \theta} - e^{\ii y^\top \theta} \Big|  \\
& =: A + B + C.
\end{align}
Note that $K(y/\Delta_n)$ and $K(x/\Delta_n)$ vanish whenever $x$ and $y$ are outside $\Delta_n\cdot S_K$.  Therefore, since
$\delta_n\to 0$ and $\Delta_n\to\infty$, the right-hand side of \eqref{e:Ln-ABC}  vanishes 
for all $\|x-y\|\le 2\delta_n$ such that $\|y\| \ge {\rm const}\cdot \Delta_n$.  Therefore,
the supremum in \eqref{e:Lnx-Lny} does not change if it is taken over the set 
$$
 {\cal I}_n:= \{(x,y)\, :\, \|x-y\|\le 2\delta_n,\ \|y\|\le {\rm const} \cdot \Delta_n\}.
$$ 
Thus, we restrict our attention to $(x,y)\in {\cal I}_n$.  
By Lemma \ref{le:rate for geometry quantity}, we have
$|T_n\cap (T_n - y)|\sim |T_n|$, uniformly in $(x,y)\in {\cal I}_n$.  This fact and the Lipschitz property of the complex exponentials and the kernel $K$  (by (c)
 of Assumption \ref{a:kernel}), immediately imply that
$$
B \le_c \frac{\delta_n/\Delta_n}{|T_n|} \ \ \mbox{ and }\ \ C\le_c \frac{\delta_n}{|T_n|},
$$
uniformly in $(x,y)\in {\cal I}_n$.

Now, for term $A$, exploiting the boundedness of the kernel and the fact that $\|x-y\|\le 2\delta_n$, we obtain
\begin{align*}
A &\le \|K\|_\infty \frac{\Big||T_n\cap (T_n-x)| - |T_n\cap (T_n-y)|\Big|}{|T_n\cap (T_n-x)|\cdot |T_n\cap (T_n-y)|}\\
 &\le \|K\|_\infty \frac{|T_n+B(0,2\delta_n)|-|T_n|}{|T_n\cap (T_n-x)|\cdot |T_n\cap (T_n-y)|} \\
    &\le_c \frac{|T_n+B(0,2\delta_n)|-|T_n|}{|T_n|^2} =\mathcal{O}\left(\frac{\delta_n}{|T_n|^{1+1/d}}\right),
\end{align*}
where the last inequality follows from Lemma \ref{le:rate for geometry quantity} and Assumption \ref{a:sampling}. Note that the last
bound is uniform in $(x,y)\in {\cal I}_n$.  Combining the above bounds on the terms $A, B,$ and $C$, we obtain
\eqref{e:Lnx-Lny}.  This completes the proof.
\end{proof}

\subsection{Proof of Theorem \ref{thm:var for f^K}} \label{s:var_proof}

For easy referencing, we begin by recalling the statement of Theorem \ref{thm:var for f^K}.
\begin{thm}[Theorem \ref{thm:var for f^K}]\label{sthm:var for f^K}
Let $X=\left\{X(t), t\in\mathbb{R}^d\right\}$ be a zero-mean, strictly stationary real $\H$-valued process. 
Suppose that Assumptions \ref{a:cov}, \ref{a:kernel}, \ref{a:sampling}, and \ref{a:var} hold. Also, assume that $\Delta_n$ satisfies
\begin{align*} 
\Delta_n\cdot S_K\subset T_n-T_n \hspace{.2cm}\quad \mbox{ for all $n$, and}\hspace{.2cm}\quad
 \delta_n + \Delta_n^d /|T_n|  \to 0 \mbox{ as $n\to\infty$}.
\end{align*}
Then
\begin{equation*}
\sup_{\theta\in\Theta}\mathbb{E}\|\hat f_n(\theta)-\mathbb{E}\hat f_n(\theta)\|_{\rm HS}^2 = \mathcal{O}\left(\frac{\Delta_n^d}{|T_n|} \right),
\ \text{as}\ n\to\infty.
\end{equation*}
\end{thm}

\begin{proof}
In what follows we will use $\Delta$ and $T$ instead of $\Delta_n$ and $T_n$ respectively. 
Recall \eqref{e:var_1} and \eqref{e:var_3} in the main paper.  Namely, we have
\begin{align} \label{s:e:var_1}
\begin{split}
& \mathbb{E}\|\hat f_n(\theta)-\mathbb{E}\hat f_n(\theta)\|_{\rm HS}^2 \\
& = \frac{1}{(2\pi)^{2d}}\sum_{t\in\mathbb{T}_n}\sum_{s\in\mathbb{T}_n}\mathop{\sum\sum}_{\substack{h\in[\Delta\cdot S_K]\cap(\mathbb{T}_n-t) \\h'\in[\Delta\cdot S_K]\cap(\mathbb{T}_n-s)}}e^{\mathbbm{i}(h-h')^{\top}\theta}K\left(\frac{h}{\Delta}\right)K\left(\frac{h'}{\Delta}\right)\\ 
& \qquad\cdot|V(t+h)|\cdot|V(t)| \cdot|V(s+h')|\cdot|V(s)| \\
& \qquad\cdot \frac{{\rm Cov}\left( X(t+h)\otimes X(t),X(s+h')\otimes X(s)\right)}{|T\cap(T-h)||T\cap(T-h')|}
\end{split}
\end{align}
where
\begin{align*}
\begin{split}
& {\rm Cov}\left( X(t+h)\otimes X(t),X(s+h')\otimes X(s)\right) \\
& := \mathbb{E}\left\langle X(t+h)\otimes X(t)-C(h),X(s+h')\otimes X(s)-C(h')\right\rangle_{\rm HS}.
\end{split}
\end{align*}
By Definition \ref{def:cum} (main paper), since $X$ is real $C = \check C$, and 
\begin{align} \label{s:e:var_3}
\begin{split}
& {\rm Cov}\left( X(t+h)\otimes X(t),X(s+h')\otimes X(s)\right) \\
    &= \mathbb{E}\langle X(t+h), X(s+h')\rangle_{\mathbb{H}}\cdot\mathbb{E}\langle X(s), X(t)\rangle_{\mathbb{H}} \\
    & \hspace{.5cm} +  \langle  C(t-s+h),  C(s+h'-t)\rangle_{\rm HS}\\
    & \hspace{.5cm} + {\rm cum}\left(X(t+h),X(t),X(s+h'),X(s)\right).  
\end{split}
\end{align}
For simplicity of notation, write ${\rm cum}(s,t,u,v) = {\rm cum}\left(X(s),X(t),X(u),X(v)\right)$.
We fix a real CONS $\{e_j\}$ and use the representation  in Proposition \ref{prop:Cov for cross product} (see also 
\eqref{e:cum_basis} in the main paper). Next, we split the sum on the right-hand side of \eqref{s:e:var_1} into three terms corresponding to the decomposition \eqref{s:e:var_3}.  Namely, we define 
\begin{align*}
    A & := \sum_{t\in\mathbb{T}_n}\sum_{s\in\mathbb{T}_n}\mathop{\sum\sum}_{\substack{h\in[\Delta\cdot S_K]\cap(\mathbb{T}_n-t) \\h'\in[\Delta\cdot S_K]\cap(\mathbb{T}_n-s)}}e^{\mathbbm{i}(h-h')^{\top}\theta}K\left(\frac{h}{\Delta}\right)K\left(\frac{h'}{\Delta}\right)\cdot|V(t+h)|\cdot|V(t)|\\ 
    & \qquad \cdot|V(s+h')|\cdot|V(s)|\cdot\frac{\mathbb{E}\left\langle X(t+h),X(s+h')\right\rangle_\bbH\cdot \mathbb{E}\left\langle X(s),X(t)\right\rangle_\bbH}{|T\cap(T-h)||T\cap(T-h')|},\\
    B & := \sum_{t\in\mathbb{T}_n}\sum_{s\in\mathbb{T}_n}\mathop{\sum\sum}_{\substack{h\in[\Delta\cdot S_K]\cap(\mathbb{T}_n-t) \\h'\in[\Delta\cdot S_K]\cap(\mathbb{T}_n-s)}}e^{\mathbbm{i}(h-h')^{\top}\theta}K\left(\frac{h}{\Delta}\right)K\left(\frac{h'}{\Delta}\right)\cdot|V(t+h)|\cdot|V(t)|\\ 
    & \qquad \cdot|V(s+h')|\cdot|V(s)|\cdot\frac{\left\langle  C(t-s+h),  C(s-t+h')\right\rangle_{\rm HS}}{|T\cap(T-h)||T\cap(T-h')|}\\
    \text{and}\\
    C & := \sum_{t\in\mathbb{T}_n}\sum_{s\in\mathbb{T}_n}\mathop{\sum\sum}_{\substack{h\in[\Delta\cdot S_K]\cap(\mathbb{T}_n-t) \\h'\in[\Delta\cdot S_K]\cap(\mathbb{T}_n-s)}}e^{\mathbbm{i}(h-h')^{\top}\theta}K\left(\frac{h}{\Delta}\right)K\left(\frac{h'}{\Delta}\right)\cdot|V(t+h)|\cdot|V(t)|\\ 
    &\qquad \cdot|V(s+h')|\cdot|V(s)|\frac{{\rm cum}\left(X(t+h),{X(t)},{X(s+h')},X(s)\right)}{|T\cap(T-h)||T\cap(T-h')|}.
\end{align*}
Thus,
\begin{equation}\label{e:ABC}
 (2\pi)^d\mathbb{E}\|\hat f_n(\theta)-\mathbb{E}\hat f_n(\theta)\|_{\rm HS}^2 = A+B+C.
\end{equation}
In the sequel, the bounds we shall obtain are based on the summation of the absolute values of the summands. Therefore, 
in view of Lemma \ref{le:rate for geometry quantity} (below) and the assumption $\Delta^d = o(|T|)$, the denominators in 
$A, B, C$ can be replaced by $|T|^{-2}.$  

We start with the term $C$. Lemma \ref{le:for the variance cumulants} entails that 
\begin{equation}\label{s:e:C}
 C =\mathcal{O}\left(\frac{{\cal N}(\Delta_n \cdot S_K,\bbT_n)}{|T_n|}\right) = \mathcal{O}\left(\frac{\Delta_n^d}{|T_n|}\right),
 \end{equation}
where the last relation follows from Lemma \ref{le:suppl-N(Delta,V)}.

The term $B$ is bounded above by 
\begin{align*}
    |B| & \leq \frac{1}{|T_n|^2}\sum_{\substack{t,s\in\mathbb{T}_n\\ u\in\mathbb{T}_n\cap(t+\Delta\cdot S_K)\\v\in\mathbb{T}_n\cap(s+\Delta\cdot S_K)}}\|  C(u-s)\|_{\rm HS}\|  C(v-t)\|_{\rm HS}\cdot |V(t)|\cdot |V(s)|  \cdot |V(u)|\cdot |V(v)|,
\end{align*}
where we have implemented the change of variables $u=t+h$ and $v=s+h'.$ Applying Lemma \ref{le: for the variance B}, 
we immediately obtain that 
\begin{equation}\label{s:e:B}
B=\mathcal{O}\left(\frac{\mathcal{N}(\Delta_n\cdot S_K,\mathbb{T}_n)^2}{|T_n|^2}\right) = {\cal O}\left(\frac{\Delta_n^{2d}}{|T_n|^2} \right),
\end{equation}
where we applied Lemma \ref{le:suppl-N(Delta,V)}.

Finally, we steer our attention to the term $A$. Observe that 
\begin{align} \label{e: EXY_trace}
\left|\mathbb{E}\langle X,Y\rangle\right| = \left|\mathbb{E} \ {\rm trace}\left(X\otimes Y\right)\right| = \left|{\rm trace}\left(\mathbb{E}[X\otimes Y]\right)\right|
\le \|\mathbb{E}[X\otimes Y]\|_{\rm tr}
\end{align}
by (iii) of Lemma \ref{le:trace norm alt def}. Thus, 
\begin{align*}
|A| & \leq_c \frac{\|K\|_{\infty}^2}{|T|^2}\sum_{t\in\mathbb{T}_n}\sum_{s\in\mathbb{T}_n}\mathop{\sum\sum}_{\substack{h\in[\Delta\cdot S_K]\cap(\mathbb{T}_n-t) \\h'\in[\Delta\cdot S_K]\cap(\mathbb{T}_n-s)}}\left\Vert \mathbb{E} [X(t+h)\otimes X(s+h')]\right\Vert_{\rm tr}\\ 
& \qquad\qquad \cdot \left\Vert \mathbb{E} [X(s)\otimes X(t)]\right\Vert_{\rm  tr} \cdot|V(t+h)|\cdot|V(t)|\cdot|V(s+h')|\cdot|V(s)| \\ 
& = \frac{\|K\|_{\infty}^2\mathcal{N}(\Delta\cdot S_K,\mathbb{T}_n)}{|T|} \cdot A_1\cdot A_2,
\end{align*}
where
\begin{align*}
A_1 = \frac{1}{|T|}\sum_{t\in\mathbb{T}_n}\sum_{s\in\mathbb{T}_n}  \left\Vert C(s-t)\right\Vert_{\rm tr}\cdot|V(t)|\cdot|V(s)|
\end{align*}
and
\begin{align*}
A_2 & = \frac{1}{\mathcal{N}(\Delta\cdot S_K,\mathbb{T}_n)}\mathop{\sum\sum}_{\substack{h\in[\Delta\cdot S_K]\cap(\mathbb{T}_n-t) \\h'\in[\Delta\cdot S_K]\cap(\mathbb{T}_n-s)}}\left\Vert C(t-s+h-h')\right\Vert_{\rm tr} \cdot|V(t+h)|  \cdot|V(s+h')|.
\end{align*}
Now
$$A_2\leq \frac{1}{\mathcal{N}(\Delta\cdot S_K,\mathbb{T}_n)}\max_{t,s\in\mathbb{T}_n}\sum_{\substack{u\in \mathbb{T}_n\cap(t+\Delta\cdot S_K)\\v\in\mathbb{T}_n\cap(s+\Delta\cdot S_K)}}\|C(u-v)\|_{\rm tr}\cdot|V(u)|\cdot|V(v)|.$$
By Lemma \ref{le: for the variance A} we obtain that $A_2=\mathcal{O}\left(1\right).$ Moreover, a close inspection of the proof of Lemma \ref{le: for the variance A} shows that $A_1$ is also of the order $\mathcal{O}\left(1\right)$. Keeping only the dominating bounds for $A$, we have that 
\begin{equation}\label{s:e:A}
   A  =\mathcal{O}\left(\frac{\mathcal{N}(\Delta\cdot S_K,\mathbb{T}_n)}{|T|}
\right) = {\cal O}\left( \frac{\Delta_n^d}{|T_n|}\right),
\end{equation}
by Lemma \ref{le:suppl-N(Delta,V)}.
In view of \eqref{e:ABC}, gathering all the bounds in \eqref{s:e:C}, \eqref{s:e:B}, and \eqref{s:e:A}, we complete the proof of the theorem.
\end{proof}

\subsection{Lemmas used in the proofs of Theorems  \ref{thm:expectation f_n^K-f_n} and \ref{thm:var for f^K}}\label{appendix:kernel-based estimator}

For the next lemmas, we need to define the quantity:
\begin{equation}\label{def:N(Delta,V)}
\mathcal{N}(A,\mathbb{T}_n):=\max_{w\in\mathbb{T}_n}\left|\cup_{v\in\mathbb{T}_n,w-v\in A +B(0,2\delta_n)}V(v)\right|.   
\end{equation}
This is the maximum volume over $w$ of the unions of all tessellation cells for which the representatives $v$'s
 are in the $2\delta_n$ inflated $A$ neighborhood of $w$.

\begin{lemma}\label{le:suppl-N(Delta,V)}
Let $\mathcal{N}(\Delta_n\cdot S_K,\mathbb{T}_n)$ be defined as in \eqref{def:N(Delta,V)} and suppose that (a) of Assumption \ref{a:sampling} holds. Then 
$$\mathcal{N}(\Delta_n\cdot S_K,\mathbb{T}_n)=\mathcal{O}\left(\Delta_n^d\right),\ \ \mbox{ as }\Delta_n\to\infty.$$
\end{lemma}
\begin{proof}
Indeed, let $Z(t):= \cup_{s\in\mathbb{T}_n,t-s\in \Delta_n\cdot S_K +B(0,2\delta_n)}V(s)$ and $t_0\in \mathbb{T}_n.$ We will show that 
\begin{equation}
\label{e:Zt0}
  Z(t_0)\subseteq B(t_0,3\delta_n+\Delta_n\cdot M_K),
\end{equation} 
where $M_k=\sup_{h\in S_K}\|h\|_2.$ Suppose $u\in Z(t_0)$. Then, there is $s_u\in\mathbb{T}_n$ such that $$ \|t_0-s_u\|_2\leq\Delta_n\cdot M_K + 2\delta_n$$ with $u\in V(s_u)$. Thus, $$\|u-t_0\|_2\leq \|u-s_u\|_2+\|s_u-t_0\|_2\leq \delta_n+ \Delta_n M_K+2\delta_n=3\delta_n+\Delta_n M_K,$$
which implies \eqref{e:Zt0}. This entails that, 
\begin{align*}
& \mathcal{N}(\Delta_n\cdot S_K,\mathbb{T}_n)=\max_{t\in\mathbb{T}_n}|Z(t)|\leq|B(0,3\delta_n+\Delta_nM_K)| \\
& \quad\quad =\mathcal{O}\left((\Delta_n+\delta_n)^d\right)=\mathcal{O}\left(\Delta_n^d\right),
\end{align*}
as $\Delta_n\to\infty$, where the last relation follows from (a) of Assumption \ref{a:sampling}.
\end{proof}

\begin{lemma}\label{le:rate for geometry quantity}
Under Assumption \ref{a:sampling}, for $\|h\|_2\le|T_n|^{1/d}$ we have that 
$$\frac{|T_n|-|T_n\cap(T_n-h)|}{|T_n|}=\mathcal{O}\left(\frac{\|h\|_2}{|T_n|^{1/d}}\right),\ \textrm{as}\ n\to\infty.$$
Consequently, if $\sup_{h\in A_n} \|h\|^d =o\left(|T_n|\right),$ we have that 
\[
 \sup_{h\in A_n} \frac{|T_n\cap(T_n-h)|}{|T_n|}=\mathcal{O}(1).
 \]
\end{lemma}
\begin{proof}
We will make critical use of the Steiner formula from convex analysis \citep[see, e.g.][]{gruber2007convex}. We have that 
$$
\frac{|T_n|-|T_n\cap(T_n-h)|}{|T_n|}\leq \frac{|T_n+B(0,\|h\|_2)|-|T_n|}{|T_n|}.
$$
An application of Steiner formula to the convex set $T_n$ entails that 
$$
|T_n+B(0,\|h\|_2)| = \sum_{j=0}^d\mu_j(T_n)\|h\|_2^{d-j},
$$
where $\mu_j(\cdot)$ denote the intrinsic volumes of order $j$. Note that $\mu_d(T_n)=|T_n|.$ Thus,
\begin{align*}
    \frac{|T_n+B(0,\|h\|_2)|-|T_n|}{|T_n|} & = \sum_{j=0}^{d-1}\frac{\mu_j(T_n)\|h\|_2^{d-j}}{|T_n|}\\ & = \sum_{j=0}^{d-1} \mu_j\left(\frac{T_n}{|T_n|^{1/d}}\right)\cdot\left(\frac{\|h\|_2}{|T_n|^{1/d}}\right)^{d-j},
\end{align*}
where the last equality follows from the homogeneity of the intrinsic volumes. 

Assumption \ref{a:sampling}, part (b), along with the continuity of the intrinsic volumes in the Hausdorff metric on the set of convex bodies 
see, e.g., Section 1.2.2 in \cite{lotz2018concentration}  or  Theorem 6.13(iii) in \cite{gruber2007convex} and the fact that 
$\|h\|_2\leq|T_n|^{1/d}$ complete the proof.    
\end{proof}

The following remark shows that the order of the bounds in Lemma \ref{le:rate for geometry quantity}
 obtained using the Steiner formula cannot be improved. 

\begin{remark}\label{re:suppl-Steiner formula}
For the case $d=2$ and $d=3$, when $T_n$ is a circle and a sphere respectively, we can evaluate the desired volume exactly. Indeed, for $d=2,$ we have that \begin{align*}
       & \frac{|T_n|-|T_n\cap(T_n-h)|}{|T_n|} \\
       & =1-\frac{2}{\pi}\arccos\left(\frac{\|h\|_2}{2|T_n|^{1/2}}\right)+\frac{1}{2|T_n|^{1/2}}\|h\|_2\sqrt{4-\left(\frac{\|h\|_2}{|T_n|^{1/2}}\right)^2}
        \\ & = \mathcal{O}\left(\frac{\|h\|_2}{|T_n|^{1/2}}\right),
    \end{align*}
    and for $d=3$ we have that 
    \begin{align*}
        \frac{|T_n|-|T_n\cap(T_n-h)|}{|T_n|}& = \frac{3}{4}\frac{\|h\|_2}{|T_n|^{1/3}}-\frac{1}{16}\left(\frac{\|h\|_2}{|T_n|^{1/3}}\right)^3 = \mathcal{O}\left(\frac{\|h\|_2}{|T_n|^{1/3}}\right).
    \end{align*}

These two cases provide evidence that the application of Steiner formula is not giving us a loose upper bound, at least when $T_n$ is an $n$-dimensional ball.

Now, when $T_n$ is a square, and assuming that $S_n$ is the side of the square, we have that 
\begin{align*}
\max\frac{|T_n|-|T_n\cap(T_n-h)|}{|T_n|} = \frac{S_n\cdot\|h\|_2\cdot\sqrt{2}-\|h\|_2^2/2}{S_n^2}=\mathcal{O}\left(\frac{\|h\|_2}{|T_n|^{1/2}}\right).    
\end{align*}
Finally, when $T_n$ is a cube of side $S_n,$ we have
\begin{align*}
\max\frac{|T_n|-|T_n\cap(T_n-h)|}{|T_n|}  = \frac{\frac{3}{2}\cdot S_n^2\cdot\|h\|_2-\frac{\sqrt{2}}{2}\cdot\|h\|_2^3}{S_n^3} =\mathcal{O}\left(\frac{\|h\|_2}{|T_n|^{1/3}}\right).    
\end{align*}
This suggests that using $n$-dimensional cubes leads indeed to the same rates as for the $n$-dimensional balls. 
\end{remark}

\begin{lemma}\label{le: for the variance A}
Let $\{A_n\}$ be a growing sequence of open sets such that $A_n\uparrow \R^d$,
as $n\to\infty$. Moreover, let $\mathbb{T}_n$ be the set of representatives of a tessellation of $T_n,$ with the diameter $\delta_n\to 0$, as $n\to\infty$,
 where $T_n$ is as in \eqref{def:f^K(theta)}. Also, let Assumptions \ref{a:cov}, \ref{a:kernel} and \ref{a:sampling} hold. 
Then 
\begin{align*}
& \frac{1}{\mathcal{N}(A_n,\mathbb{T}_n)}\max_{t,s\in\mathbb{T}_n}\sum_{\substack{u\in \mathbb{T}_n\cap(t+A_n)\\
v\in\mathbb{T}_n\cap(s+A_n)}}\|C(u-v)\|_{\rm tr}\cdot|V(u)|\cdot|V(v)| \\
& =\mathcal{O}\left(\delta_n^{\gamma}+1\right) = {\cal O}(1),
\end{align*}
 as $n\to\infty$, where $\mathcal{N}(\cdot,\cdot)$ is defined in \eqref{def:N(Delta,V)}.
\end{lemma}

\begin{proof} Using the inequality 
$$
\Big|\max_{i=1,\cdots,m} a_i - \max_{j=1,\cdots,m} b_j\Big |\le \max_{i=1,\cdots,m} |a_i-b_i|,
$$
valid for all $a_i, b_i \in \R,\ i=1,\cdots,m$, we obtain
\begin{align*}
    & \Bigg|\max_{t,s\in\mathbb{T}_n}\sum_{\substack{u\in \mathbb{T}_n\cap(t+A_n)\\v\in\mathbb{T}_n\cap(s+A_n)}}\|C(u-v)\|_{\rm tr}\cdot|V(u)|\cdot|V(v)| 
    \\ & \hspace{3cm} - \max_{t,s\in\mathbb{T}_n}\iint_{\substack{h\in\cup_{u\in \mathbb{T}_n\cap(t+A_n)}V(u)\\h'\in\cup_{v\in \mathbb{T}_n\cap(s+A_n)}V(v)}}\|C(h-h')\|_{\rm tr}dh'dh\Bigg|\\ &\leq\max_{t,s\in\mathbb{T}_n}\Bigg|\sum_{\substack{u\in \mathbb{T}_n\cap(t+A_n)\\v\in\mathbb{T}_n\cap(s+A_n)}}\|C(u-v)\|_{\rm tr}\cdot|V(u)|\cdot|V(v)|\\ 
    & \hspace{3cm} - \iint_{\substack{h\in\cup_{u\in \mathbb{T}_n\cap(t+A_n)}V(u)\\h'\in\cup_{v\in \mathbb{T}_n\cap(s+A_n)}V(v)}}\|C(h-h')\|_{\rm tr}dh'dh\Bigg|\\ 
    & =\max_{t,s\in\mathbb{T}_n}\left|\sum_{\substack{u\in \mathbb{T}_n\cap(t+A_n)\\v\in\mathbb{T}_n\cap(s+A_n)}} \iint_{\substack{h\in V(u)\\h'\in V(v)}}\|C(u-v)\|_{\rm tr}-\|C(h-h')\|_{\rm tr}dh'dh\right|\\ 
    &\leq \max_{s\in\mathbb{T}_n}\sum_{\substack{v\in\mathbb{T}_n\cap(s+A_n)}} \int_{h'\in V(v)}
    \int_{x\in \R^d}\Bigg( \sup_{y\, :\, \|x-y\|\le 2\delta_n} \Big|\|C(y)-C(x)\|_{\rm tr}\Big|dx \Bigg) dh' \\ & \leq \vertiii{C}_{\gamma}(2\delta_n)^\gamma \cdot \max_{s\in\mathbb{T}_n}\sum_{\substack{v\in\mathbb{T}_n\cap(s+A_n)}} \int_{\substack{h'\in V(v)}}dh' \\
 &   \leq  \vertiii{C}_{\gamma}(2\delta_n)^\gamma \mathcal{N}(A_n,\mathbb{T}_n),
\end{align*}
where we made the change of variables $x:=h-h'$ and enlarged the domain of integration over $x\in \R^d$. 
The last two inequalities follow from \eqref{e:C-L1-Holder} and definition of ${\cal N}(\cdot,\cdot)$ in \eqref{def:N(Delta,V)}.

To complete the proof, we show that $$\frac{1}{\mathcal{N}(A_n,\mathbb{T}_n)}\max_{t,s\in\mathbb{T}_n}\iint_{\substack{h\in\cup_{u\in \mathbb{T}_n\cap(t+A_n)}V(u)\\h'\in\cup_{v\in \mathbb{T}_n\cap(s+A_n)}V(v)}}\|C(h-h')\|_{\rm tr}dh'dh =\mathcal{O}\left(1\right).$$
With the change of variables $x=h-h',$ we have that the aforementioned term is equal to 
\begin{align*}
    & \frac{1}{\mathcal{N}(A_n,\mathbb{T}_n)}\max_{t,s\in\mathbb{T}_n}\iint_{\substack{x\in[\cup_{u\in \mathbb{T}_n\cap(t+A_n)}V(u)-\cup_{v\in \mathbb{T}_n\cap(s+A_n)}V(v)]\\ h'\in[\cup_{v\in \mathbb{T}_n\cap(s+A_n)}V(v)]\cap[\cup_{u\in \mathbb{T}_n\cap(t+A_n)}V(u)-x]}}\|C(x)\|_{\rm tr}dh'dx \\ 
    & \leq \max_{t,s\in\mathbb{T}_n}\mathop{\int}_{\substack{x\in[\cup_{u\in \mathbb{T}_n\cap(t+A_n)}V(u)\\ 
    \ \ \ -\cup_{v\in \mathbb{T}_n\cap(s+A_n)}V(v)]}}\|C(x)\|_{\rm tr}\frac{\left|[\cup_{v\in \mathbb{T}_n\cap(s+A_n)}V(v)]\cap[\cup_{u\in \mathbb{T}_n\cap(t+A_n)}V(u)-x]\right|}{\mathcal{N}(A_n,\mathbb{T}_n)}dx \\ & \leq \int_{w\in\mathbb{R}^d}\|C(x)\|_{\rm tr}dx = {\cal O}(1),
\end{align*}
by Assumption \ref{a:cov} (a). The proof is complete. 
\end{proof}

\noindent The next lemma is similar to Lemma \ref{le: for the variance A}. It is used for the term $B$ in the proof of Theorem \ref{sthm:var for f^K}.
\begin{lemma}\label{le: for the variance B} 
Let all the assumptions of Lemma \ref{le: for the variance A} hold. Then
 \begin{align} \label{s:e:CC_HS}
 \begin{split}
     & \frac{1}{|T_n|^2}\sum_{\mathclap{\substack{t,s\in\mathbb{T}_n\\ u\in\mathbb{T}_n\cap(t+A_n)\\v\in\mathbb{T}_n\cap(s+A_n)}}}\|C(u-s)\|_{\rm HS}\| C(t-v)\|_{\rm HS}\cdot |V(t)| \cdot |V(s)|\cdot |V(u)|\cdot |V(v)| \\ 
     & = \mathcal{O}\left(\frac{\mathcal{N}(A_n,\mathbb{T}_n)^2}{|T_n|^2}\right),
     \end{split}
\end{align}
as $n\to\infty$, where $\mathcal{N}(\cdot,\cdot)$ is defined in \eqref{def:N(Delta,V)}.
\end{lemma}

\begin{proof} 
For any $w$ in $T_n$, let $\tau_w$ denote the point $t_{n,i}\in\bbT_n$ that is in the same cell as $w$; if $w$ is on the boundary of a cell, then let $\tau_w$ be any of the $t_{n,i}\in\bbT_n$ in adjacent cells. Thus, $\|w-\tau_w\|_2 \le \delta_n$.
It follows that
\begin{align*}
    & \hspace{.6cm} \sum_{\mathclap{\substack{t,s\in\mathbb{T}_n\\ u\in\mathbb{T}_n\cap(t+A_n)\\v\in\mathbb{T}_n\cap(s+A_n)}}}\|C(u-s)\|_{\rm HS}\| C(t-v)\|_{\rm HS}\cdot |V(t)| \cdot |V(s)|\cdot |V(u)|\cdot |V(v)| \\
&=  \mathop{\iiiint}_{\substack{w,\, x\in T_n\\ h\in\cup_{u\in \mathbb{T}_n\cap(\tau_w+A_n)}V(u)\\h'\in\cup_{v\in \mathbb{T}_n\cap(\tau_x+A_n)}V(v)}}
\|C(\tau_h-\tau_x)\|_{\rm HS}\| C(\tau_w-\tau_{h'})\|_{\rm HS} dh'dhdxdw \\
& \le \mathop{\iiiint}_{\substack{w,\, x\in T_n\\ h\in w+A_n + B(0,2\delta_n) \\h'\in x+A_n + B(0,2\delta_n)}}
 \|C(\tau_h-\tau_x)\|_{\rm HS}\| C(\tau_w-\tau_{h'})\|_{\rm HS} dh'dhdxdw \\
& \le \mathop{\iiiint}_{\substack{w,\, x\in T_n\\ h\in w+A_n + B(0,2\delta_n) \\h'\in x+A_n + B(0,2\delta_n)}}
\sup_{\substack{\lambda_i\in B(0,2\delta_n),\\i=1,2}}
 \|C(\lambda_1+h-x)\|_{\rm HS}\| C(\lambda_2+w-{h'})\|_{\rm HS} dh'dhdxdw \\ 
 &= \mathop{\iiiint}_{\substack{w,\, x\in T_n\\ \tilde h\in A_n + B(0,2\delta_n) \\ \tilde h'\in A_n + B(0,2\delta_n)}}
\sup_{\substack{\lambda_i\in B(0,2\delta_n),\\i=1,2}}
 \|C(\lambda_1+\tilde h+w-x)\|_{\rm HS}\| C(\lambda_2+w-\tilde h'-x)\|_{\rm HS} d\tilde h'd\tilde hdxdw \\ 
 &\le \mathop{\iint}_{\substack{\tilde h\in A_n + B(0,2\delta_n) \\ \tilde h'\in A_n + B(0,2\delta_n)}}
 \left(\int_{x\in \bbR^d} \sup_{\lambda\in B(0,2\delta_n)} 
\|C(\lambda+x)\|_{\rm HS}dx\right)^2  d\tilde h'd\tilde h \\ 
& = |A_n + B(0,2\delta_n) |^2 \left(\int_{x\in \bbR^d} \sup_{\lambda\in B(0,2\delta_n)} 
\|C(\lambda+x)\|_{\rm HS}dx\right)^2
\end{align*}
which in view of Assumption \ref{a:cov} implies \eqref{s:e:CC_HS} and completes the proof, since 
\begin{equation}\label{e:An_asymp_N}
    |A_n + B(0,2\delta_n) |\asymp \mathcal{N}(A_n,\mathbb{T}_n),
\end{equation}
because $\delta_n/|A_n|\to0$(recall \eqref{def:N(Delta,V)}.
\end{proof}

Finally, we state a lemma to handle term $C$ in the proof of Theorem \ref{sthm:var for f^K}.

\begin{lemma}\label{le:for the variance cumulants}
Let the assumptions of Lemma \ref{le: for the variance A} and Assumption \ref{a:var} hold. Moreover, assume that the process $\{X(t)\}$ is strictly stationary.
Then, \begin{align*}
    & \frac{1}{|T_n|^2}\sum_{\mathclap{\substack{t,s\in\mathbb{T}_n\\ u\in\mathbb{T}_n\cap(t+A_n)\\v\in\mathbb{T}_n\cap(s+A_n)}}}\quad\Big|{\rm cum}\left(X(u),{X(t)},{X(v)},X(s)\right)\Big| 
    \cdot |V(t)|\cdot |V(s)|\cdot |V(u)|\cdot |V(v)|
\end{align*}
is of the order ${\cal O}( {\mathcal{N}(A_n,\mathbb{T}_n)}/{|T_n|} ),$ where $\mathcal{N}(\cdot,\cdot)$ is defined in \eqref{def:N(Delta,V)}.
\end{lemma}

\begin{proof} 
Proceeding as in Lemma \ref{le: for the variance B}, 
It follows that
\begin{align*}
    & \hspace{.6cm} \sum_{\mathclap{\substack{t,s\in\mathbb{T}_n\\ u\in\mathbb{T}_n\cap(t+A_n)\\v\in\mathbb{T}_n\cap(s+A_n)}}}\quad 
    |{\rm cum}(u,t,v,s)|\hspace{.5cm}\prod_{\mathclap{\tau\in\{t,s,v,u\}}}|V(\tau)| \\
    & = \mathop{\iiiint}_{\substack{w,\, x\in T_n\\ h\in\cup_{u\in \mathbb{T}_n\cap(\tau_w+A_n)}V(u)\\h'\in\cup_{v\in \mathbb{T}_n\cap(\tau_x+A_n)}V(v)}}
  |{\rm cum}(\tau_h,\tau_w,\tau_{h'},\tau_x)| dh'dhdxdw \\ 
  & \le \mathop{\iiiint}_{\substack{w,\, x\in T_n\\ h\in \tau_w+A_n + B(0,\delta_n) \\h'\in \tau_x+A_n + B(0,\delta_n)}}
  |{\rm cum}(\tau_h,\tau_w,\tau_{h'},\tau_x)| dh'dhdxdw \\ 
  & \le \mathop{\iiiint}_{\substack{w,\, x\in T_n\\ h\in w+A_n + B(0,2\delta_n) \\h'\in x+A_n + B(0,2\delta_n)}}
  |{\rm cum}(\tau_h,\tau_w,\tau_{h'},\tau_x)| dh'dhdxdw.
\end{align*}  
Applying (b) of Assumption \ref{a:var}, the last expression becomes
\begin{align*}
  & \mathop{\iiiint}_{\substack{w,\, x\in T_n\\ h\in w+A_n + B(0,2\delta_n) \\h'\in x+A_n + B(0,2\delta_n)}}
  |{\rm cum}(\tau_h-\tau_x,\tau_w-\tau_x,\tau_{h'}-\tau_x,0)| dh'dhdxdw \\ 
  & \le \mathop{\iiiint}_{\substack{w,\, x\in T_n\\ h\in w+A_n + B(0,2\delta_n) \\h'\in x+A_n + B(0,2\delta_n)}}
  \sup_{\substack{\lambda_i\in B(0,2\delta_n),\\i=1,2,3}}
  |{\rm cum}(\lambda_1+h-x,\lambda_2+w-x,\lambda_3+h'-x,0)| dh'dhdxdw \\ 
    & = \mathop{\iiiint}_{\substack{w,\, x\in T_n\\ \tilde h, \tilde h' \in A_n + B(0,2\delta_n)}} 
  \sup_{\substack{\lambda_i\in B(0,2\delta_n),\\i=1,2,3}}
  |{\rm cum}(\lambda_1+\tilde h+w-x,\lambda_2+w-x,\lambda_3+\tilde h',0)| d\tilde h'd\tilde hdxdw \\ 
  & = \mathop{\iiint}_{\substack{y \in T_n-T_n\\ \tilde h, \tilde h' \in A_n + B(0,2\delta_n)}} |T_n \cap (T_n-y)|
  \sup_{\substack{\lambda_i\in B(0,2\delta_n),\\i=1,2,3}}
  |{\rm cum}(\lambda_1+\tilde h+y,\lambda_2+y,\lambda_3+\tilde h',0)| d\tilde h'd\tilde hdy \\ 
 & \le |T_n| \mathop{\int}_{\tilde h'\in A_n+B(0,2\delta_n)} \sup_{z\in\mathbb{R}^d}
 \mathop{\iint}_{\substack{y\in T_n-T_n\\ \tilde h\in A_n+B(0,2\delta_n)}}
  \sup_{\substack{\lambda_i\in B(0,2\delta_n),\\i=1,2,3}}
  |{\rm cum}(\lambda_1+\tilde h+y,\lambda_2+y,\lambda_3+z,0)| d\tilde h'd\tilde hdy \\ 
  & = \mathcal{O}(\mathcal{N}(A_n,\mathbb{T}_n)\cdot|T_n|),
\end{align*}
where the last relation is justified by Assumption \ref{a:var} (b) and \eqref{e:An_asymp_N}. 
Note that we have applied two changes of variables; first $\tilde h = h-\tau_w$ and $\tilde h' =h'-x$, and second $v =\tilde w-x$. This completes the proof of the lemma.
\end{proof}

\section{Proofs for Section \ref{sec:discrete-time}}\label{appendix:bounds on the rates}

We start this section obtaining rates on the variance of $\hat f_n(\theta)$ in Section \ref{sec:discrete-time}. 
We first establish a result that is more general than what is needed for the proofs of Section \ref{sec:discrete-time}. We will use it to evaluate the variance of $\hat f_n(\theta)$ in the time-series setting where $\delta_n\equiv 1$.

\begin{prop}\label{thm: variance of hat f}
Let the process $\{X(t)\}_{t\in \delta_n\cdot \Z^d}$ be strictly stationary and suppose that Assumptions \ref{a:cov-prime}, \ref{a:sampling}, 
and  \ref{a:var_d} hold. Then, for the estimator $\hat f_n(\theta)$ defined in \eqref{def: f^ts}, we have the following upper bound on the rate of the variance
$$
\sup_{\theta\in\Theta}\mathbb{E}\|\hat f_n(\theta)-\mathbb{E} \hat f_n(\theta)\|_{\rm HS}^2=\mathcal{O}\left(\frac{\Delta_n^d}{|T_n|}\right), \quad{\rm as}\quad n\to\infty,
$$
where $T_n =  \delta_n \cdot [0,n]^d$ and $|T_n| = (n \delta_n)^d$.
\end{prop}
\begin{proof} As before, we will use that $\|{\cal A}\|_{\rm HS}\le \|{\cal A}\|_{\rm tr}$ throughout.  Recall that 
$\bbT_n = \delta_n \cdot\{1,\cdots,n\}^d$ is a discrete set of $n^d$ samples, while $T_n =  \delta_n \cdot [0,n]^d$ is a hypercube of
side $n\delta_n$.

We start with 
\begin{align*}
    & \hat f_n(\theta)-\mathbb{E}\hat f_n(\theta)  \\
    & = \frac{\delta_n^{2d}}{(2\pi)^d}\sum_{t\in \bbT_n}\sum_{s\in \bbT_n}e^{\ii(t-s)^{\top}\theta}\frac{X(t) \otimes X(s)-C(t-s)}{|T_n\cap(T_n-(t-s))|}K\left(\frac{t-s}{\Delta_n}\right)\\ & = 
    \frac{\delta_n^{2d}}{(2\pi)^d}\sum_{t\in \bbT_n}\sum_{h\in \Delta_n S_K\cap \delta_n\cdot \mathbb{Z}^d}e^{\ii h^{\top}\theta}\frac{X(t+h) \otimes X(t)-C(h)}{|T_n\cap(T_n-h)|}K\left(\frac{h}{\Delta_n}\right)\cdot\mathbbm{1}(h+t\in \delta_n\cdot \mathbb{Z}^d).
\end{align*}
This means that  
\begin{align*}
& \hat f_n(\theta)-\mathbb{E}\hat f_n(\theta) \\
& = \frac{\delta_n^{2d}}{(2\pi)^d}\sum_{t\in \bbT_n}\sum_{h\in \Delta_n S_K\cap(\bbT_n-t)}e^{\ii h^{\top}\theta}\frac{X(t+h) \otimes X(t)-C(h)}{|T_n\cap(T_n-h)|}K\left(\frac{h}{\Delta_n}\right).
\end{align*}
Then, the variance becomes
\begin{align*}
    & \mathbb{E}\|\hat f_n(\theta)-\mathbb{E}\hat f_n(\theta)\|_{\rm HS}^2 \\
    & = \frac{\delta_n^{4d}}{(2\pi)^{2d}}\sum_{t\in \bbT_n}\sum_{s\in \bbT_n}\mathop{\sum\sum}_{\substack{h\in[\Delta_n\cdot S_K]\cap(\bbT_n-t) \\h'\in[\Delta_n\cdot S_K]\cap(\bbT_n-s)}}e^{\mathbbm{i}(h-h')^{\top}\theta}K\left(\frac{h}{\Delta_n}\right)K\left(\frac{h'}{\Delta_n}\right)\\ & \qquad\qquad\qquad\qquad\qquad\qquad\cdot\frac{{\rm Cov}\left( X(t+h)\otimes X(t),X(s+h')\otimes X(s)\right)}{|T_n\cap(T_n-h)|\cdot |T_n\cap(T_n-h')|}.
    \end{align*}

By Proposition \ref{prop:Cov for cross product} we obtain that ${\rm Cov}\left( X(t+h)\otimes X(t),X(s+h')\otimes X(s)\right)$ is equal to
\begin{align*}
    &\sum_{i\in I}\sum_{j\in I}{\rm cum}\left(X_i(t+h),{X_j(t)},{X_i(s+h')},X_j(s)\right)\\ & \qquad+ \mathbb{E}\left\langle X(t+h),X(s+h')\right\rangle_\bbH\cdot \mathbb{E}\left\langle X(t),X(s)\right\rangle_\bbH+\left\langle   C(t-s+h),  C(s-t+h')\right\rangle_{\rm HS}.
\end{align*}

In an analogous manner to the proof of Theorem \ref{thm:var for f^K}, we define the quantities
\begin{align*}
    A & := \delta_n^{4d} \sum_{t\in \bbT_n}\sum_{s\in \bbT_n}\mathop{\sum\sum}_{\substack{h\in[\Delta_n\cdot S_K]\cap(\bbT_n-t) \\h'\in[\Delta_n\cdot S_K]\cap(\bbT_n-s)}}e^{\mathbbm{i}(h-h')^{\top}\theta}K\left(\frac{h}{\Delta_n}\right)K\left(\frac{h'}{\Delta_n}\right)\\ & \qquad\qquad\qquad\qquad\qquad\qquad \cdot\frac{\mathbb{E}\left\langle X(t+h),X(s+h')\right\rangle_\bbH\cdot \mathbb{E}\left\langle X(s),X(t)\right\rangle_\bbH}{|T_n\cap(T_n-h)|\cdot |T_n\cap(T_n-h')|},\\
    B & := \delta_n^{4d} \sum_{t\in \bbT_n}\sum_{s\in \bbT_n}\mathop{\sum\sum}_{\substack{h\in[\Delta_n\cdot S_K]\cap(\bbT_n-t) \\h'\in[\Delta_n\cdot S_K]\cap(\bbT_n-s)}}e^{\mathbbm{i}(h-h')^{\top}\theta}K\left(\frac{h}{\Delta_n}\right)K\left(\frac{h'}{\Delta_n}\right)\\ & \qquad\qquad\qquad\qquad\qquad\qquad \cdot\frac{\left\langle  C(t-s+h),  C(s-t+h')\right\rangle_{\rm HS}}{|T_n\cap(T_n-h)|\cdot |T_n\cap(T_n-h')|}\\
    \text{and}\\
    C & := \delta_n^{4d} \sum_{t\in \bbT_n}\sum_{s\in \bbT_n}\mathop{\sum\sum}_{\substack{h\in[\Delta_n\cdot S_K]\cap(\bbT_n-t) \\h'\in[\Delta_n\cdot S_K]\cap(\bbT_n-s)}}e^{\mathbbm{i}(h-h')^{\top}\theta}K\left(\frac{h}{\Delta_n}\right)K\left(\frac{h'}{\Delta_n}\right)\\ &\qquad\qquad\qquad\qquad\qquad \cdot\sum_{i\in I}\sum_{j\in I}\frac{{\rm cum}\left(X_i(t+h),{X_j(t)},{X_i(s+h')},X_j(s)\right)}{|T_n\cap(T_n-h)|\cdot |T_n\cap(T_n-h')|}.
\end{align*}

We start with term $A$. Bounding terms by their norm (using \eqref{e: EXY_trace}) and changing variables, we obtain 
\begin{align*}
    |A|& \le \delta_n^{4d} \sum_{t\in \bbT_n}\sum_{s\in \bbT_n}\mathop{\sum\sum}_{\substack{h\in[\Delta_n\cdot S_K]\cap(\bbT_n-t) \\h'\in[\Delta_n\cdot S_K]\cap(\bbT_n-s)}} \frac{|\mathbb{E}\left\langle X(t+h),X(s+h')\right\rangle_\bbH\cdot \mathbb{E}\left\langle X(s),X(t)\right\rangle_\bbH|}{|T_n\cap(T_n-h)|\cdot |T_n\cap(T_n-h')|}
    \\ & \le \delta_n^{4d} \sum_{t\in \bbT_n}\sum_{s\in \bbT_n}\mathop{\sum\sum}_{\substack{h\in[\Delta_n\cdot S_K]\cap(\bbT_n-t) \\h'\in[\Delta_n\cdot S_K]\cap(\bbT_n-s)}} \frac{\|C(h-h'+t-s)\|_{\rm tr}\cdot \|C(s-t)\|_{\rm tr}}{|T_n\cap(T_n-h)|\cdot |T_n\cap(T_n-h')|}\\ 
    & = \delta_n^{4d} \sum_{t\in \bbT_n}\sum_{s\in \bbT_n}\|C(s-t)\|_{\rm tr}\mathop{\sum\sum}_{\substack{h\in[\Delta_n\cdot S_K]\cap(\bbT_n-t) \\h'\in[\Delta_n\cdot S_K]\cap(\bbT_n-s)}} \frac{\|C(h-h'+t-s)\|_{\rm tr} }{|T_n\cap(T_n-h)|\cdot |T_n\cap(T_n-h')|}\\ 
    & = \delta_n^{4d}
    \sum_{w\in \bbT_n-\bbT_n}\|C(w)\|_{\rm tr}\sum_{x\in \bbT_n\cap(\bbT_n-w)}\mathop{\sum\sum}_{\substack{u\in[\Delta_n\cdot S_K\cap(\bbT_n-(w+x))]-[\Delta_n\cdot S_K\cap(\bbT_n-x)] \\v\in[\Delta_n\cdot S_K\cap(\bbT_n-x)]\cap\{[\Delta_n\cdot S_K\cap(\bbT_n-(w+x))]-u\}}} \\ 
    & \qquad\qquad\qquad\qquad\qquad\qquad\qquad\frac{\|C(u+w)\|_{\rm tr} }{|T_n\cap(T_n-(u+v))|\cdot |T_n\cap(T_n-v)|},
\end{align*}
where the last equality is obtained through the change of variables $w=s-t,x=s,u=h-h',v=h'.$

By Lemma \ref{le:rate for geometry quantity}, in view of Assumption \ref{a:sampling}(ii), and inflating slightly the sums by dropping the intersections in the summations of $u,v$ we obtain that  
\begin{align*}
    |A| &\le_{\rm c} \frac{\delta_n^{2d}\times \delta_n^{2d}}{|T_n|^2}\sum_{w\in \bbT_n-\bbT_n}\|C(w)\|_{\rm tr}\sum_{x\in \bbT_n\cap(\bbT_n-w), }\sum_{u\in \Delta_n(S_K-S_K)} \|C(u+w)\|_{\rm tr} \\
    & \hskip6cm  \sum_{v\in(\Delta_n S_K)\cap (\Delta_n S_K-u) \cap \delta_n\cdot \mathbb{Z}^d}1 \\ 
    & \le_{\rm c}\frac{\delta_n^{2d}}{|T_n|^2} \sum_{w\in \bbT_n-\bbT_n}\|C(w)\|_{\rm tr}\cdot|T_n\cap(T_n-w)|\\ 
    & \qquad\qquad\qquad\qquad \times \sum_{u\in \Delta_n(S_K-S_K)\cap \delta_n\cdot {\mathbb Z}^d } \|C(u+w)\|_{\rm tr}\cdot |(\Delta_n S_K)\cap (\Delta_n S_K-u)|\\ 
    & \le_{\rm c}\frac{\Delta_n^d}{|T_n|} \times \delta_n^d \sum_{w\in  \delta_n\cdot \mathbb{Z}^d}\|C(w)\|_{\rm tr} \times \delta_n^d \sum_{u\in \delta_n\cdot \mathbb{Z}^d} \|C(u)\|_{\rm tr} =\mathcal{O}\left(\frac{\Delta_n^d}{|T_n|}\right),
\end{align*}
by Assumption \ref{a:cov-prime}, where we used that 
$\delta_n^d \sum_{t} 1_{\{t\in \bbT_n\}}  \sim |T_n|$ 
and 
\begin{equation*}
\delta_n^d \sum_{t} 1_{\{t\in (\Delta_n\cdot S_K) \cap (\Delta_n\cdot S_K-u)\cap \delta_n \cdot \mathbb Z^d\}} \le 2 |\Delta_n S_K| ={\cal O}(\Delta_n^d).
\end{equation*}
 
Now, we shift to term $B$. Using the change of variables $w:= t-s$, the Cauchy-Schwartz inequality, we obtain
\begin{align*}
    |B| & \le \delta_n^{4d} \sum_{w\in \bbT_n-\bbT_n} \sum_{s} 1_{\{ \bbT_n\cap(\bbT_n-w)\} }(s) 
    \mathop{\sum\sum}_{\substack{h\in[\Delta_n\cdot S_K]\cap\delta_n\cdot \mathbb{Z}^d \\h'\in[\Delta_n\cdot S_K]\cap\delta_n\cdot \mathbb{Z}^d}}
    K\left(\frac{h}{\Delta_n}\right)K\left(\frac{h'}{\Delta_n}\right)\\ & \qquad\qquad\qquad \cdot
    \frac{ \|C(h+w)\|_{\rm HS} \|C(h'-w)\|_{\rm HS} }{|T_n\cap(T_n-h)|\cdot |T_n\cap(T_n-h')|} \\
    &\le_c \frac{\delta_n^{3d}}{|T_n|}\sum_{w\in \bbT_n-\bbT_n}
     \mathop{\sum\sum}_{\substack{h\in[\Delta_n\cdot S_K]\cap\delta_n\cdot \mathbb{Z}^d\  \\h'\in[\Delta_n\cdot S_K]\cap\delta_n\cdot \mathbb{Z}^d}}  \|C(h+w)\|_{\rm HS} \|C(h'-w)\|_{\rm HS},
\end{align*}
where we used that $\delta_n^d \sum_s 1_{\{ \bbT_n\cap(\bbT_n-w)\} }(s) = {\cal O}(|T_n|)$, and Lemma \ref{le:rate for geometry quantity}
to conclude that $|T_n\cap (T_n - h)| \sim |T_n\cap (T_n -h')| \sim |T_n|$, uniformly in $h,h'\in \Delta_n \cdot S_K$.  Now,
with the change of variables $u:= h+w$, and expanding the range of summation, we further 
obtain
\begin{align*}
    |B| & \le_{\rm c} \frac{\delta_n^{3d}}{|T_n|}\sum_{u\in \delta_n\cdot \mathbb{Z}^d}\| C(u)\|_{\rm HS}\sum_{h'\in\Delta_nS_K\cap \delta_n\cdot \Z^d\ }\sum_{h\in\Delta_nS_K\cap(u-(\bbT_n-\bbT_n))\ }\| C(h'-u+h)\|_{\rm HS}
    \\ & \le\frac{1}{|T_n|}\Big(\sum_{u\in \delta_n\cdot \mathbb{Z}^d} \delta_n^d \| C(u)\|_{\rm HS} \Big) 
    \Big( \delta_n^d \sum_{h'\in\Delta_nS_K\cap \delta_n\cdot \Z^d\ }
   1 \Big)  \Big(\sum_{h\in\delta_n\cdot \mathbb{Z}^d}\delta_n^d \| C(h)\|_{\rm HS}\Big) \\ 
    & = \mathcal{O}\left(\frac{\Delta_n^d}{|T_n|}\right),
\end{align*}
in view of Assumption \ref{a:cov-prime}.

Finally, we look at term $C$. An application of Lemma \ref{le:rate for geometry quantity}, again gives us that 
\begin{align*}
    |C| & \le_{\rm c}\frac{\delta_n^{4d} }{|T_n|^2}\sum_{t\in \bbT_n}\sum_{s\in \bbT_n}\mathop{\sum\sum}_{\substack{h\in[\Delta_n\cdot S_K]\cap(\bbT_n-t) \\h'\in[\Delta_n\cdot S_K]\cap(\bbT_n-s)}} 
       \left|\sum_{i\in I}\sum_{j\in I}{\rm cum}\left(X_i(t+h),{X_j(t)},{X_i(s+h')},X_j(s)\right)\right|\\ &\le
    \frac{\delta_n^{4d}}{|T_n|^2}\sum_{t\in \bbT_n}\sum_{s\in \bbT_n}\mathop{\sum\sum}_{\substack{h\in\Delta_n\cdot S_K \cap \delta_n\cdot\Z^d \\h'\in \Delta_n\cdot S_K \cap \delta_n\cdot\Z^d }} 
       \left|\sum_{i\in I}\sum_{j\in I}{\rm cum}\left(X_i(h+t-s),{X_j(t-s)},{X_i(h')},X_j(0)\right)\right|
    \\ &\le
    \frac{\delta_n^{3d} }{|T_n|}\sum_{w\in \bbT_n-\bbT_n}\mathop{\sum\sum}_{\substack{h\in\Delta_n\cdot S_K\cap \delta_n\cdot\Z^d \\h'\in \Delta_n\cdot S_K \cap \delta_n\cdot\Z^d}} 
        \left|\sum_{i\in I}\sum_{j\in I}{\rm cum}\left(X_i(h+w),{X_j(w)},{X_i(h')},X_j(0)\right)\right|
    \\ &\le_{\rm c}
    \frac{\Delta_n^d}{|T_n|}\sup_{h'\in\delta_n\cdot \mathbb{Z}^d} \delta_n^{2d} \sum_{w\in \delta_n\cdot \mathbb{Z}^d\ \ }\sum_{h\in\delta_n\cdot \mathbb{Z}^d}
       \left| \sum_{i\in I}\sum_{j\in I}{\rm cum}\left(X_i(h+w),{X_j(w)},{X_i(h')},X_j(0)\right)\right| \\
    & = \mathcal{O}\left(\frac{\Delta_n^d}{|T_n|}\right),
\end{align*}
where we used that $\delta_n^d \sum_{h'\in \Delta_n\cdot S_K\cap \delta_n\cdot \Z^d} 1= {\cal O}(|\Delta_n|^d)$, 
the fact that $\delta_n^d \sum_{s\in \bbT_n \cap \bbT_n} = {\cal O}(|T_n|)$, and 
Assumption \ref{a:var_d}.  This completes the proof.
\end{proof}

{\bf Proof of Theorem \ref{prop:PL bias time series}:} As indicated, by following the proof of Proposition \ref{thm: variance of hat f}, 
we see that the variance bound ${\cal O}(\Delta_n^d/|T_n|)$ is uniform in $f\in {\cal P}_D(\beta,L)$.  Therefore, to prove
\eqref{e:bias_ts_PLD}, by Relation \eqref{e:bias-bound-time-series} in Theorem \ref{prop: consistency f^ts}, it is enough to 
bound terms $B_1(\Delta_n)$ and $B_2(\Delta_n)$, uniformly in $f\in {\cal P}_D(\beta,L)$.
Let $M_k$ and $m_k$ be the radii of the smallest ball that contains $S_K$ and the largest ball contained in $S_K$ respectively. Starting with term $B_2$ we have,   
\begin{align*}
\sup_{f\in {\cal P}_D(\beta,L)} B_2(\Delta_n)& \le \sum_{\|h\|_2\ge\Delta_n m_k} \sup_{f\in {\cal P}_D(\beta,L)} \|C(h)\|_{\rm HS}\|h\|_2^{\beta} \|h\|_2^{-\beta} \\
& \le (\Delta_n m_k)^{-\beta}\sum_{\|h\|_2\ge\Delta_n m_k}\sup_{f\in {\cal P}_D(\beta,L)} \|C(h)\|_{\rm HS}\|h\|_2^{\beta}\\ 
& \le  (\Delta_n m_k)^{-\beta} \cdot L = {\cal O} (\Delta_n^{-\beta}),    
\end{align*}
in view of \eqref{def:PL minimax class discrete time}.  Recalling that $K(0)=1,$ we have that for $\lambda+1>\beta,$
\begin{align*}
    \sup_{f\in {\cal P}_D(\beta,L)} B_1(\Delta_n) &\le \sum_{0<\|h\|_2\le\Delta_nM_K}
    \sup_{f\in {\cal P}_D(\beta,L)} \left\|C(h)\right\|_{\rm HS}\left[1-K\left(\frac{h}{\Delta_n}\right)\right]\\ 
    &\le \tilde c\sum_{0<\|h\|_2\le\Delta_nM_K}
   \sup_{f\in {\cal P}_D(\beta,L)}  \left\|C(h)\right\|_{\rm HS}\frac{\|h\|_2^{\lambda+1}}{\Delta_n^{\lambda+1}}\\ 
   & \le \tilde c \Delta_n^{-\beta}M_K^{\lambda+1-\beta}\sum_{0<\|h\|_2\le\Delta_nM_K}
    \sup_{f\in {\cal P}_D(\beta,L)} \left\|C(h)\right\|_{\rm HS}\|h\|_2^{\beta}= \mathcal{O}(\Delta_n^{-\beta}),
\end{align*}
where the second inequality follows from the multivariate Taylor Theorem since \eqref{e:Kernel differentiability} holds. Indeed, under this condition and using that $K(0)=1$, we obtain
 $$K\left(\frac{h}{\Delta_n}\right)=1+R_{0,\lambda}\left(\frac{h}{\Delta_n}\right),$$ 
where $|R_{0,\lambda}(h)|\leq \frac{L_1}{(\lambda+1)!}\|h\|_2^{\lambda+1}.$

Collecting the bounds for $B_1(\Delta_n)$ and $B_2(\Delta_n)$, we obtain that the bias is of order ${\cal O}(\Delta_n^\beta)$, uniformly over the class ${\cal P}_D(\beta,L)$.
Now, by Theorem \ref{prop: consistency f^ts}, the variance is of order ${\cal O}(\Delta_n^d/|T_n|)$ and picking 
$\Delta_n= |T_n|^{1/(2\beta+d)}$, we obtain the rate-optimal bound in \eqref{e:mse_ts_PLD}.
\qed

\bigskip
\vskip.3cm
\textbf{Proof of Theorem \ref{prop:PL bias cont time}.}
In view of Theorem \ref{s:thm:expectation f_n^K-f_n}, one only needs to bound the 
terms  $B_1(\Delta_n)$ and $B_2(\Delta_n)$ appropriately.  Starting with term $B_2$, if $m_K$ denotes the 
radius of the largest ball contained in $S_K,$ we have
\begin{align*}
B_2(\Delta_n) & \le
\int_{x\not\in\Delta_n\cdot S_K}  \|C(x)\|_{\rm HS} dx \\
& \le (\Delta_n\cdot m_K)^{-\beta}\cdot \int_{ \|x\| >\Delta_n\cdot m_K} \|x\|_2^{\beta}\cdot \|C(x)\|_{\rm HS} dx \\
& \le (\Delta_n\cdot m_K)^{-\beta} \cdot L = {\cal O}(\Delta_n^{-\beta}).
\end{align*}

Next, recall that $$B_1(\Delta_n) := \left\|\int_{h\in\Delta_n\cdot S_K}e^{-\mathbbm{i}h^{\top}\theta}C(h)\left(1-K\left(\frac{h}{\Delta_n}\right)\right)dh\right\|_{\rm HS}.$$ 
Since $K(0)=1$ and \eqref{e:Kernel differentiability} holds, by the Taylor theorem, we have that $$K\left(\frac{h}{\Delta_n}\right)=1+R_{0,\lambda}\left(\frac{h}{\Delta_n}\right),$$ 
where $|R_{0,\lambda}(h)|\leq \frac{L_1}{(\lambda+1)!}\|h\|_2^{\lambda+1}.$
Thus, with  $M_K$ denoting the radius of the smallest ball centered at the origin that contains $S_K$, 
the term $B_1$ is bounded by 
\begin{align*}
B_1(\Delta_n) &\le \frac{L_1}{(\lambda+1)!} \int_{h\in \Delta_n\cdot S_K} \|C(h)\|_{\rm HS} \Big(\frac{\|h\|_2}{\Delta_n} \Big)^{\lambda+1} dh\\
& \le \frac{L_1\cdot (M_K\cdot \Delta_n)^{\lambda+1 -\beta}}{(\lambda+1)! \cdot \Delta_n^{\lambda+1}} 
\int_{h\in \Delta_n\cdot S_K} \|C(h)\|_{\rm HS} \|h\|_2^{\beta}dh = {\cal O}(\Delta_n^{-\beta}),
\end{align*}
since $\lambda+1>\beta$ and in view of \eqref{def:CTPL minimax class}.

Collecting the bounds for $B_1$ and $B_2$, we get  $B_1(\Delta_n)+B_2(\Delta_n)  = \mathcal{O}\left(\Delta_n^{-\beta}\right).$  
Now, the optimal choice of $\Delta_n$ is the one which balances the last bound with the 
rate of the variance, that is, $\Delta_n^d/|T_n| \sim \Delta_n^{-2\beta}$.  This is achieved with 
$$
 \Delta_n:= |T_n|^{1/(2\beta+d)} \equiv (n\delta_n)^{d/(2\beta +d)},
 $$
which upon substitution yields the rate $\delta_n^{\gamma} \vee \Delta_n^{-\beta} = \delta_n^{\gamma}\vee (n\delta_n)^{-\beta d/(2\beta + d)}$ in \eqref{e:rate-bounds}.
\qed

\section{Proofs for Section \ref{sec:minimax}} \label{s:minimax_proof}
For any member $e_i$ of the real CONS, consider the (scalar) real-valued process \[X_{e_i}(t) := \langle X(t), e_i\rangle_\bbH\] and let $C_{e_i}(x)$ and $f_{e_i}(\theta)$ be its stationary covariance and spectral density, respectively. If $f\in {\cal P}_D(\beta,L)$ then 
$$
\int_{\R^d} (1+\|x\|_2^\beta) |C_{e_i}(x)| dx \le L\quad\mbox{and} \quad |\hat f_{e_i} (\theta_0) - f_{e_i}(\theta_0)| \le \|\hat f(\theta_0) - f(\theta_0)\|_{\rm HS}. 
$$
These follow from the simple fact that $|\langle {\cal A} \phi, \phi\rangle_\bbH|  \le \|{\cal A}\|_{\rm op}$ for any bounded linear operator ${\cal A}$ and unitary $\phi\in\bbH$. Thus, it suffices to prove Theorems \ref{thm:pointwise rate time series} and \ref{thm:PL minimax cont time} for  scalar, real-valued processes, which we do below. 
\vskip.3cm

\begin{proof}[{\bf Proof of Theorem \ref{thm:pointwise rate time series}.}]
Let $\|\cdot\|$ denote the Euclidean norm in $\mathbb{R}^d$ and $C_g$ be the covariance that corresponds to the spectral density $g$. Fix an interior point $\theta_0\in(-\pi,\pi)^d$ and 
let $f_{0,n}(\theta)=L/(2\cdot(2\pi)^d)\cdot \mathbbm{1}(\theta\in[-\pi,\pi]^d).$ Then, 
\begin{equation}\label{e:C_f0-1}
C_{f_{0,n}}(k) = \int_{\theta\in[-\pi,\pi]^d} e^{-\ii \theta^{\top}k} \frac{L}{2\cdot(2\pi)^d} d\theta = \mathbbm{1}(k=0) L/2,
\end{equation}
and therefore
\begin{equation}\label{e:C_f0-2}
\sum_{k\in\mathbb{Z}^2}|C_{f_{0,n}}(k)|(1+\|k\|^{\beta}) = C_{f_{0,n}}(0) = L/2 < L.
\end{equation}
Let for $\theta=(\theta_i)_{i=1}^d\in\R^d$,
\begin{align} \label{e:g_def}
g(\theta) = \epsilon \cdot \prod_{i=1}^d \varphi(\theta_i),\ \ \mbox{ where } \varphi(x) =  \exp\left(-\frac{1}{1-(x/\pi)^2}\right)\mathbbm{1}(|x| < \pi),\ \ (x\in \R)
\end{align}
for some $\epsilon>0$, which is to be determined.  The function $\varphi$ is a type of a ``bump'' function that belongs to 
$C_0^{\infty}(\mathbb{R})$ (the class of infinitely differentiable functions with compact support). The support of $\varphi$ is
the compact interval $[-\pi,\pi]$. Hence $g\in C_0^{\infty}(\mathbb{R}^d)$ and its support is $[-\pi,\pi]^d$. Consider the function 
\begin{align} \label{e:gn_def}
g_n(\theta) = h_n^{\beta}g\left(\frac{\theta-\theta_0}{h_n}\right),
\end{align}
where $0 < h_n\le 1$ and tends to $0$ at a rate to be determined later. Observe that since $\theta_0\in (-\pi,\pi)^d$, the support of
$g_n$ is included in $\theta_0 + h_n\cdot[-\pi,\pi]^d \subset (-\pi,\pi)^d$, for all sufficiently small $h_n$.

Now, consider the ``alternative'' spectral density models:
\begin{align*}
    f_{1,n}(\theta) & = f_{0,n}(\theta)+ [g_n(\theta)+g_n(-\theta)] =: \frac{L}{2}\cdot\mathbbm{1}_{(-\pi,\pi)^d}(\theta)+r_n(\theta),\ \ \theta\in [-\pi,\pi]^d.
\end{align*}
We will choose the sequence $h_n$ and the constant $\epsilon>0$ such that the following three properties hold.

\medskip
\noindent{\em Properties:}
\begin{enumerate}
    \item[(1)] $f_{0,n},f_{1,n}\in {\cal P}_D(\beta,L)$, where the class ${\cal P}_D(\beta,L)$ is defined in \eqref{def:PL minimax class discrete time}.
    
    \item[(2)] For all $n$ large enough, we have
    \begin{equation}\label{e:Property-2} 
     f_{1,n}(\theta_0)-f_{0,n}(\theta_0) = h_n^{\beta}[g(0)+g(2\theta_0/h_n)]=g(0) (1+\mathbbm{1}(\theta_0=0)) \cdot h_n^{\beta}.
     \end{equation}
    
    \item[(3)] ${\rm KL}(\mathbb{P}_{1n},\mathbb{P}_{0n})\le C<\infty$, where ${\rm KL}$ stands for the Kullback-Leibler distance and $\mathbb{P}_{0n}$ and $\mathbb{P}_{1n}$ are probability distributions of the data $\{X(k),\ k\in \{1,\cdots,n\}^d\}$
   under $f_{0,n}$ and $f_{1,n}$ respectively.
\end{enumerate}

\medskip
\noindent{\em Proof of Property (1).} We have already shown that $f_{0,n} \in {\cal P}_D(\beta,L)$. Recalling \eqref{def:PL minimax class discrete time}, and in view of \eqref{e:C_f0-1} and \eqref{e:C_f0-2},
to prove $f_{1,n}\in {\cal P}_D(\beta,L)$, it is enough to show that 
\begin{equation}\label{e:C_gn_inequality}
\sum_{k\in \mathbb Z^d} |C_{g_n}(k)| (1+\|k\|^\beta) < \frac{L}{4},
\end{equation}
where $C_{g_n}(k) = \int_{\theta\in[-\pi,\pi]^d}e^{-\ii \theta^{\top}k} g_n(\theta)d\theta$. 

We have that for all $k = (k_i)_{i=1}^d\in \mathbb Z^d$,
\begin{align} \label{e:C_gn_bound}
\begin{split}
    C_{g_n}(k) & =\int_{\theta\in[-\pi,\pi]^d}e^{-\ii \theta^{\top}k}h_n^{\beta}g\left(\frac{\theta-\theta_0}{h_n}\right)d\theta \\ 
    &= h_n^{\beta+d}\cdot e^{-\ii \theta_0^{\top}k}\int_{x\in [-\pi,\pi]^d}e^{-\ii x^{\top}k h_n}g(x)dx \\
    & =  \epsilon\cdot h_n^{\beta+d}\cdot e^{-\ii \theta_0^{\top}k} \prod_{i=1}^d \what \varphi(k_i h_n),
    \end{split}
\end{align}
where we used the change of variables $x = (\theta-\theta_0)/h_n$ and the fact that $\theta_0+ h_n\cdot[-\pi,\pi]^d \subset  (-\pi,\pi)^d$, for all 
sufficiently small $h_n$.  The last relation follows from \eqref{e:g_def}, where $\what \varphi(x) := \int_{-\pi}^\pi e^{-\ii x u}\varphi(u)du$ denotes
the Fourier transform of the bump function $\varphi$.  Now using the fact that the derivatives of $\varphi$ vanish at $\pm \pi$, i.e., 
$\varphi^{(\ell)}(\pm \pi) = 0$, for all $\ell = 0,1,\hdots$, integration by parts yields
\begin{equation*} 
\what \varphi(x) = \frac{1}{(-i x)^{\ell}} \int_{-\pi}^{\pi} e^{-\ii x u}\varphi^{(\ell)}(u)du,\ \ \ell =0,1,\ldots
\end{equation*}
Indeed, for all $\ell$, the derivative $\varphi^{(\ell)}(x)$ is continuous and supported on $[-\pi,\pi]$, and thus
$$
|\what \varphi(x) | \le c_0 \wedge (|x|^{-\ell} c_\ell),\ \ \mbox{ where } \ c_\ell:= \int_{-\pi}^\pi |\varphi^{(\ell)}(u)| du.
$$
In view of \eqref{e:C_gn_bound}, we have
\begin{equation}\label{e:C_gn_bound_via_ell}
|C_{g_n}(k)| \le \epsilon \cdot h_n^{\beta+d} \prod_{i=1}^d  \left( c_0 \wedge  \frac{c_{\ell}}{|k_i h_n|^{\ell}} \right).
\end{equation}
We will choose $\ell\ge 2$ and $\epsilon>0$ to satisfy \eqref{e:C_gn_inequality}
for all sufficiently small $h_n$. Notice that $\|k\|^\beta \le  d^{(\beta-1)\vee 0} \sum_{i=1}^d |k_i|^\beta$. Hence, \eqref{e:C_gn_inequality} follows from
$$
d^{(\beta-1)\vee 0} \sum_{i=1}^d  \sum_{k\in \mathbb Z^d}  |C_{g_n}(k)| ( 1+ |k_i|^\beta)  = d^{\beta\vee 1} \sum_{k\in \mathbb Z^d} |C_{g_n}(k)| ( 1+ |k_1|^\beta) < \frac{L}{4},
$$
where $k=(k_i)_{i=1}^d$.  Indeed, this follows from the observation that, by \eqref{e:C_gn_bound}, we have
$$
\sum_{i=1}^d \sum_{k\in\Z^d} |C_{g_n}(k_1,\cdots,k_d)| |k_i|^\beta = d \sum_{k\in\Z^d} |C_{g_n}(k_1,\cdots,k_d)| |k_1|^\beta.
$$
 Thus, it suffices to show that 
\begin{align} \label{e:C_gn}
\begin{split}
\sum_{k\in\mathbb{Z}^d}|C_{g_n}(k)|(1+|k_1|^{\beta}) & \le |C_{g_n}(0)| + 2\sum_{k=(k_i)_{i=1}^d\in\mathbb{Z}^d}|C_{g_n}(k)| |k_1|^{\beta} \\
& \le \frac{L}{4(d^{\beta\vee 1})}.
\end{split}
\end{align} 
Since $h_n\in (0,1)$, \eqref{e:C_gn_bound_via_ell} readily implies that 
\begin{align*} 
|C_{g_n}(0)|\le \epsilon \cdot c_0^d.
\end{align*}
Also, applying \eqref{e:C_gn_bound_via_ell},
\begin{align} \label{e:Csum} 
\begin{split}
    & \sum_{k=(k_i)_{i=1}^d\in\mathbb{Z}^d}|C_{g_n}(k)| |k_1|^{\beta} \\
    & \le  \epsilon \cdot  \sum_{k\in \mathbb Z^d } \left[ 
      h_n^{\beta+1} |k_1|^\beta \left( c_0 \wedge \frac{c_{\ell}}{|k_1 h_n|^{\ell}} \right) 
     \cdot h_n^{d-1} \prod_{i=2}^d  \left( c_0  \wedge \frac{c_{\ell}}{|k_i h_n|^{\ell}} \right) \right]\\
      & = \epsilon  \cdot \left [  h_n \cdot \sum_{j\in\mathbb Z} |j h_n|^{\beta}  \left( c_0 \wedge \frac{c_{\ell}}{|j h_n|^{\ell}} \right)  \right] \times \left[ h_n \sum_{j\in\Z} 
      \left( c_0  \wedge \frac{c_{\ell}}{|j h_n|^{\ell}} \right) \right ]^{d-1}\\
      &=: \epsilon  \times A_n \times (B_n)^{d-1}.
    \end{split}
\end{align}
Observe that
$A_n$ and $B_n$ are Riemann sums for the integrals
$$
A:= \int_{x\in\R} |x|^{\beta} \left( c_0 \wedge \frac{c_{\ell}}{|x|^{\ell}}\right) dx\ \ \ \mbox{ and }\ \ \ B:= \int_{x\in\R} \left( c_0 \wedge \frac{c_{\ell}}{|x|^{\ell}}\right) dx,
$$
which are clearly finite for $\ell \ge \lfloor \beta \rfloor +2$.  Taking such a value of $\ell$ and using 
the fact that $A_n\to A$ and $B_n\to B$, as $h_n\to 0$, we obtain that the right hand side of \eqref{e:Csum} is bounded above by 
$2 \epsilon \times A \times B^{d-1}$ for all sufficiently small $h_n$.
Therefore, we can ensure that \eqref{e:C_gn} holds by picking $\epsilon>0$ such that 
$$
0<\epsilon \cdot \Big[ c_0^d + 4 A\times B^{d-1}\Big] \le \frac{L}{4(d^{\beta\vee 1})}.
$$
This shows that $f_{1,n}\in {\cal P}_D(\beta,L)$ and completes the proof of Property (1).\\

\noindent {\em Proof of Property (2).}  This is immediate. Relation \eqref{e:Property-2} holds for all sufficiently large $n$ since 
 $g(\theta_0/h_n) \to g(0)\mathbbm{1}(\theta_0=0),$ as $h_n\to 0$, by the fact that $g$ is 
supported on $[-\pi,\pi]$.\\

\noindent
{\em Proof of Property (3).} Let $D_n$ and $B_{n,\xi}$  be the covariance matrices of the data $X(t), t \in \{1,\ldots,n\}^d$, that correspond to, 
respectively, the spectral densities $r_n(\theta) = g_n(\theta) + g_n(-\theta)$ and
$f_{0,n}(\theta)+\xi r_n(\theta),$ for some  $\xi\in[0,1].$ By Lemma \ref{le:mean value theorem}, 
\begin{equation}\label{ineq: Samarov's argument}
    \mathrm{KL}(\mathbb{P}_{1n},\mathbb{P}_{0n})\le \frac{1}{2}\|D_n\|_{\rm F}^2\|B_{n,\xi}^{-1}\|_{\rm op}^2,
\end{equation} 
where $\|\cdot\|_{\rm F}$ is the Frobenius norm and $\|\cdot\|_{\rm op}$ is the matrix operator norm induced by the Euclidian vector norm.   
It follows from part (iii) of Lemma \ref{le:Samarov-lemma2} applied to $A_n:= D_n$ that
\begin{align*}
 \frac{1}{n^d}\|D_n\|_{\rm F}^2  \le (2\pi)^d \int_{\theta\in[-\pi,\pi]^d}r_n^2(\theta)d\theta.
\end{align*}
Thus, recalling $r_n(\theta) = g_n(\theta)+g_n(-\theta)$, Relation \eqref{e:gn_def},
and using a change of variables, we obtain:
\begin{align}\label{ineq: D matrix}
    \frac{1}{n^d}\|D_n\|_{\rm F}^2 \le (2\pi)^d h_n^{2\beta+d}\left\{4\int_{\theta\in[-\pi,\pi]^d}g^2(\theta)d\theta\right\} = 4\cdot (2\pi)^d \cdot \epsilon^2 h_n^{2\beta+d}\|\varphi\|_{L^2}^{2d},
\end{align}
where we used \eqref{e:g_def}. Applying (i) and (ii) of Lemma \ref{le:Samarov-lemma2}, we obtain 
\begin{equation}\label{ineq: B matrix}
\|B_{n,\xi}^{-1}\|_{\rm op} \le \frac{1}{(2\pi)^d}\sup_{\theta \in [-\pi,\pi]^d} \left[ f_{0n}(\theta)+\xi r_n(\theta)\right]^{-1}
\le \frac{2}{L},    
\end{equation}
since $r_n(\theta)\ge 0$ and $ f_{0n}(\theta) = L/(2\cdot(2\pi)^d),\ \theta\in[-\pi,\pi]$.
Combining \eqref{ineq: Samarov's argument} - \eqref{ineq: B matrix}, 
\begin{align*}
{\rm KL}(\mathbb{P}_{1n}, \mathbb{P}_{0n}) & \le \frac{1}{2} \|D_n\|_F^2 \|B_{n,\xi}^{-1}\|_{\rm op} ^2
\le \left(\frac{8\cdot(2\pi)^d\epsilon^2\|g\|_{L^2}^2}{L^2} \right)n^d h_n^{2\beta+d},
\end{align*}
which is bounded, if we set
\[
h_n = M \cdot n^{-d/(2\beta+d)}.
\]
By picking $M = M_{\theta_0}$ so that $[ g(0)+ g(0)\mathbbm{1}(\theta_0=0)]\cdot M^{\beta}=1$, we have that for all
sufficiently large $n$,
\[
|f_{1n}(\theta_0)-f_{0n}(\theta_0)| = [ g(0)+ g(0)\mathbbm{1}(\theta_0=0)]\cdot M^{\beta}\cdot n^{-\frac{d\beta}{2\beta+d}}=n^{-\frac{d\beta}{2\beta+d}},
\]
which is a lower bound in the estimation error. The proof is complete by appealing to Theorem 2.5(iii) of \cite{tsybakov2008introduction}.
\end{proof}
\vskip.3cm

\noindent
\begin{proof}[{\bf Proof of Theorem \ref{thm:PL minimax cont time}}]
As argued in the proof of Theorem \ref{thm:pointwise rate time series}, it suffices to focus on the case of scalar-valued processes 
$\{X(t),\ t\in \R^d\}$.  As for the discrete-time case, we will introduce two models with spectral densities $f_{0,n}(\theta)$ and $f_{1,n}(\theta)$,
and corresponding auto-covariances $C_{0,n}(t)$ and $C_{1,n}(t)$.  Consider the function:
\begin{equation}\label{e:f0n-cont-def}
f_{0,n}(\theta) := \epsilon \cdot \prod_{i=1}^d \phi(\delta_n \theta_i),\ \ \ \theta = (\theta_i)_{i=1}^d\in\R^d,
\end{equation}
where $\phi(z)= e^{-z^2/2}/\sqrt{2\pi} ,\ z\in\R$ is the standard Normal density.

With a straightforward change of variables, we obtain:
\begin{equation}\label{e:C0n}
C_{0,n}(x) =  \int_{\R^d} e^{-\ii x^\top \theta} f_{0,n}(\theta) d\theta = \epsilon (2\pi)^{d/2} \cdot \delta_n^{-d}
\prod_{i=1}^d \phi(x_i/\delta_n),
\end{equation}
where we used the fact that $ \int_{\R} e^{-\ii x u} \phi(u)du = \sqrt{2\pi} \phi(x)$. 

 As in the time-series setting, let
\begin{equation}\label{e:f1n-grid-def}
f_{1,n}(\theta) = f_{0,n}(\theta) + h_n^\beta \left[ g\left(\frac{\theta-\theta_0}{h_n}\right) + g\left(\frac{\theta+\theta_0}{h_n}\right) \right],
\end{equation}
where $g$ is as in \eqref{e:g_def}.   Following 
the proof of Theorem \ref{thm:pointwise rate time series}, we will verify the following.

\medskip
\noindent {\em Properties:}

\begin{enumerate}
 \item[(1)] $f_{0,n},\ f_{1,n}\in {\cal P}_C(\beta,L)$, where the class ${\cal P}_C(\beta,L)$ is defined in  \eqref{def:CTPL minimax class}.
 \item[(2)] The functions $f_{0,n}$ and $f_{1,n}$ satisfy Relation  \eqref{e:Property-2}.
 \item[(3)] The KL-divergence is bounded, i.e., $\sup_n {\rm KL}(\P_{1n},\P_{0n})<\infty$, where $\P_{i,n}$ are the probability
 distributions of the data $\{X(\delta_n k),\ k\in \{1,\cdots,n\}^d\}$ under the models $f_{i,n},\ i=0,1$.
\end{enumerate}

Property (2) above is immediate by definition since the difference $f_{0,n}(\theta) - f_{1,n}(\theta)$ is
constructed as in the proof of Theorem \ref{thm:pointwise rate time series}.\\

{\em Proof of Property (1):} The fact that $f_{0,n}\in {\cal P}_C(\beta,L)$ is straightforward.  Indeed, by \eqref{e:C0n}, we have
\begin{align}\label{e:C0n-L2-bound}
\begin{split}
&\int_{\R^d} (1+\|x\|^\beta) |C_{0,n}(x)| dx  \le_c \epsilon  \cdot  \delta_n^{-d} \int_{\R^d} (1+\|x\|^\beta) e^{-\|x\|^2/2\delta_n^2 }dx \\
&\quad  = \epsilon \int_{\R^d} (1+ \| \delta_n \cdot u\|^\beta) e^{-\|u\|^2/2} du \le \epsilon \int_{\R^d} (1+ \| u\|^\beta) e^{-\|u\|^2/2} du \le L/2, 
\end{split}
\end{align}
for all  $\delta_n\in (0,1)$ and for a sufficiently small $\epsilon >0$.  This follows from the fact that with $\delta_n\in (0,1)$, we have
$\|\delta_n u\|^\beta \le \|u\|^\beta$ and the fact that $\int_{\R^d} (1+ \| u\|^\beta) e^{-\|u\|^2/2} du <\infty$.
This ensures that  \eqref{def:CTPL minimax class} holds  with $C$ replaced by $C_{0,n}$ and $L$ by $L/2$. 
That is, $f_{0,n}\in {\cal P}_C(\beta,L)$.

Now, we show that $f_{1,n}$ defined in \eqref{e:f1n-grid-def} belongs to ${\cal P}_C(\beta,L)$, by perhaps lowering the value of $\epsilon>0$.
 Let
$$
C_{g_n}(x):= \int_{\R^d} e^{-\ii \theta^\top x}g_n(\theta)d\theta,
$$
where $g_n$ is as in \eqref{e:gn_def}.
As argued in the proof of Theorem  \ref{thm:pointwise rate time series}, in view of \eqref{e:C0n-L2-bound},
it suffices to show that
\begin{equation}\label{e:Cgn-L4}
\int_{\R^d} \left(1+\|x\|^\beta\right) |C_{g_n}(x)| dx \le \frac{L}{4}.
\end{equation}
Note that Relation \eqref{e:C_gn_bound_via_ell} remains valid if $k\in \Z^d$ therein is 
replaced with $x\in \R^d$.  Therefore, \eqref{e:Cgn-L4} follows by picking a possibly smaller value of $\epsilon>0$, provided
$$
 h_n^{\beta+d} \int_{\R^d}  \left(1+\|x\|^\beta\right) \prod_{i=1}^d \left( c_0 \wedge \frac{c_{\ell}}{|h_n x_i|^{\ell_i}}\right)dx  <\infty,
$$
for some $\ell\in \mathbb N,\ i=1,\cdots,d$. Notice that for $h_n \in (0,1]$, we have 
$h_n^\beta (1+\|x\|^\beta)\le \left(1+\|h_n x\|^\beta\right)$, and hence the last integral is bounded 
above by
\begin{align*}
h_n^d \int_{\R^d}  \left(1+\|h_n x\|^\beta\right) \prod_{i=1}^d \left( c_0 \wedge \frac{c_{\ell}}{|h_n x_i|^{\ell}}\right) dx 
= \int_{\R^d} \left(1+ \| u\|^\beta\right)  \prod_{i=1}^d \left( c_0 \wedge \frac{c_{\ell}}{|u_i|^{\ell}}\right) du,
\end{align*}
where we used the change of variables $u:= h_n x$.  Clearly, the last integral is finite
provided $\ell\ge \lfloor \beta \rfloor + 2$.  This implies that \eqref{e:Cgn-L4} holds with a suitably chosen 
$\epsilon>0$, showing that $f_{1,n}\in {\cal P}_C(\beta,L)$ and completing the proof of Property (2).\\

{\em Proof of Property (3):}  Now, as in the proof of Theorem \ref{thm:pointwise rate time series} we will bound the KL-divergence
${\rm KL}(\P_{1,n},\P_{0,n})$, where $\P_{i,n}$ is the law of the Gaussian vector $\{X_i(\delta_nk),\ k\in \{1,\cdots,n\}^d \}$ under the model
$f_{i,n},\ i=0,1$. 

Observe that the $\Z^d$-indexed stationary process $\{X_i(\delta_nk),\ k\in\Z^d\}$ has the so-called folded 
spectral density
\begin{equation}\label{e:folded-spec-dens}
\wtilde f_{i,h}(\theta):=  \delta_n^{-d}  \sum_{\ell \in \Z^d} f_{i,n} \left(\frac{\theta + 2\pi \ell}{\delta_n} \right),
\ \ \theta\in [-\pi,\pi]^d,\ i=0,1.
\end{equation}
We shall apply the same argument as in the proof of Theorem \ref{thm:pointwise rate time series} based on
Samarov's Lemmas \ref{le:mean value theorem} and  \ref{le:Samarov-lemma2} applied to the {\em folded spectral densities}.

For $\xi\in [0,1]$, let $D_n$ and $B_{n,\xi}$ be the covariance matrices of zero-mean Gaussian vectors
having spectral densities $\wtilde r_n(\theta):= \wtilde f_{1,n}(\theta) - \wtilde f_{0,n}(\theta)$ 
and $\wtilde f_{0,n}(\theta) + \xi \wtilde r_n(\theta),\ \theta\in [-\pi,\pi]^d$, respectively, where 
\begin{align*}
\wtilde r_n(\theta) = h_n^{\beta} \delta_{n}^{-d} \sum_{\ell\in \Z^d} \left[  g\left(\frac{\theta+2\pi\ell }{h_n\delta_n} -\frac{\theta_0}{h_n} \right) + 
g\left(\frac{\theta + 2\pi\ell }{h_n\delta_n} +\frac{\theta_0}{h_n} \right)\right].
\end{align*}
Then, by Lemma \ref{le:mean value theorem},  we have
\begin{equation}\label{e:KL-D-cont-time}
{\rm KL}(\P_{1,n},\P_{0,n})\le \frac{1}{2} \|D_n\|_{\rm F}^2 \|B_{n,\xi}^{-1}\|_{\rm op}^2,
\end{equation}
for some $\xi\in [0,1]$.  As in \eqref{ineq: D matrix} from Lemma \ref{le:Samarov-lemma2}(iii) applied to $A_n:=D_n$, we obtain
\begin{align}\label{e:D-bound-cont-time}
\begin{split}
\|D_n\|_{\rm F}^2 & \le (2\pi)^d\cdot n^d\int_{[-\pi,\pi]^d} \wtilde r_n(\theta)^2 d\theta \\ 
& \le 4\cdot (2\pi)^d\cdot n^d h_n^{2\beta} \delta_n^{-2d} \int_{[-\pi,\pi]^d} \sum_{\ell\in\Z^d} g\left( \frac{\theta+2\pi\ell }{h_n\delta_n} -\frac{\theta_0}{h_n} \right)^2 d\theta \\
& = 4\cdot (2\pi)^d\cdot n^d h_n^{2\beta} \delta_n^{-d} \int_{\R^d} g\left( \frac{u}{h_n}  -\frac{\theta_0}{h_n} \right)^2  du  \\
& =  4\cdot(2\pi)^d\cdot n^d h_n^{2\beta+d} \delta_n^{-d} \|g^2\|_{L^2}^2 = 4\epsilon\cdot(2\pi)^d\cdot n^d h_n^{2\beta+d} \delta_n^{-d} \|\varphi\|_{L^2}^{2d},
\end{split}
\end{align}
where in the last two integrals we made changes of variables, and the last relation follows from the definition of $g$ in \eqref{e:g_def}.

Now, we deal with bounding $\|B_{n,\xi}^{-1}\|_{\rm op}$.  Notice that $B_{n,\xi}$ is the covariance matrix of
a Gaussian vector $\{X_\xi (\delta_n k),\ k\in \{1,\cdots,n\}^d\}$ coming from a stationary process $Y(k) = X_\xi(\delta_n k),\ k\in\Z^d$
with spectral density $\wtilde f_{0,n}(\theta) + \xi \wtilde r_n(\theta),\ \theta\in [-\pi,\pi]^d$, where $\wtilde r_n(\theta)\ge 0$ and $\xi\in [0,1]$.
By Lemma \ref{le:Samarov-lemma2}(ii), we then have that 
$$
\|B_{n,\xi}^{-1}\|_{\rm op} \| \le \sup_{\theta \in [-\pi,\pi]^d} \left[ \wtilde f_{0n}(\theta) + \xi \wtilde r_n(\theta) \right]^{-1} \le 
\sup_{\theta \in [-\pi,\pi]^d} \left[ \wtilde f_{0n}(\theta)\right]^{-1}.
$$
Recalling the definition of $f_{0,n}$ in \eqref{e:f0n-cont-def} and the {\em folded} spectral density in \eqref{e:folded-spec-dens},
we obtain
$
\wtilde f_{0,n}(\theta) \ge \epsilon \delta_n^{-d} \prod_{i=1}^d \phi(\theta_i) \ge \epsilon \delta_n^{-d} e^{-d\pi^2/2}/(2\pi)^{d/2}
$ 
for $\theta \in  [-\pi,\pi]^d$.  Hence
\begin{equation}\label{ineq: B matrix cont time}
\|B_{n,\xi}^{-1}\|_{\rm op} \le 
\frac{1}{(2\pi)^d} \sup_{\theta \in [-\pi,\pi]^d} \left[ \wtilde f_{0n}(\theta)\right]^{-1} \le \frac{e^{d\pi^2/2} }{(2\pi)^{d/2}\cdot \epsilon}\cdot \delta_n^d
\end{equation}
Finally, by \eqref{e:KL-D-cont-time}, \eqref{e:D-bound-cont-time}, and \eqref{ineq: B matrix cont time}, 
\begin{align*}
{\rm KL}(\mathbb{P}_{1n}, \mathbb{P}_{0n}) & \le \frac{1}{2} \|D_n\|_{\rm F}^2 \|B_{n,\xi}^{-1}\|_{\rm op} ^2
\le c \cdot n^d h_n^{2\beta+d}\delta_n^{-d} \cdot \delta_n^{2d} = c \cdot (n\delta_n)^d h_n^{2\beta+d},
\end{align*}
where $c = 2 \epsilon^{-1} \|\varphi\|_{L^2}^{2d} e^{d\pi^2}$.
Thus, the KL-divergence is uniformly bounded if we set
\[
h_n = M \cdot (n\delta_n)^{-d/(2\beta+d)}.
\]
Picking $M$ so that $g(0) (1+ \mathbbm{1}(\theta_0=0)) \cdot M^{\beta}=1$, we have that
\[
|f_{1n}(\theta_0)-f_{0n}(\theta_0)| = g(0) (1+ \mathbbm{1}(\theta_0=0)) \cdot M^{\beta}\cdot (n\delta_n)^{-\frac{d\beta}{2\beta+d}}=(n\delta_n)^{-\frac{d\beta}{2\beta+d}},
\]
which, by appealing to Theorem 2.5(iii) of \cite{tsybakov2008introduction}, 
yields the desired lower bound in estimation error in \eqref{e:thm:PL minimax cont time}. 
\end{proof}

The technical lemmas needed in the above proofs come from \cite{samarov1977lower}. The first is a slight
extension to the $d$-dimensional case. We provide proofs below for the sake of completeness.

\begin{lemma}\label{le:Samarov-lemma2} 
Let $a_j,\ j\in \Z^d$ be a sequence of numbers 
such that $\sum_{j\in\Z^d}|a_j|^2<\infty$ and $a_j=a_{-j}$. Let also $A_n$ be a matrix of dimensions $n^d\times n^d$, whose 
$(j,k)$-th element equals $a_{j-k}$, where $j$ and $k$ are multi-indices that belong to $[0:n-1]^d:=\{0,1,\cdots,n-1\}^d$ (i.e., the $(j,k)$-th element based on a natural ordering of the multi-indices of $[0:n-1]^d$). Finally, define
$\alpha(\lambda)=(2\pi)^{-d}\sum_{j\in \Z^d} a_je^{\ii j^\top \lambda}$, for $\lambda\in [-\pi,\pi]^d$. 
Then, for the norms of $A_n$, the following claims are true.
\begin{enumerate}
    \item [(i)] $\|A_n\|_{\rm op}\le (2\pi)^{d}\cdot \sup_{\lambda\in[-\pi,\pi]^d}|\alpha(\lambda)|.$
    
    \item [(ii)] If $A_n$ is positive definite, then $\|A_n^{-1}\|_{\rm op}\le (2\pi)^{-d}\cdot\sup_{\lambda\in[-\pi,\pi]^d}|1/\alpha(\lambda)|.$
    
    \item [(iii)] $n^{-d}\|A_n\|_{\rm F}^2\le \sum_{j\in\Z^d} |a_j|^2=(2\pi)^{d}\int_{[-\pi,\pi]^d} \alpha^2(\lambda)d\lambda.$
\end{enumerate}
\end{lemma}
\begin{proof} Let $N:= n^{d}$ and use the notation $[0:n-1]^d := \{0,1,\cdots,n-1\}^d.$ 
We will follow the arguments in  \cite{samarov1977lower}.
\begin{enumerate} 
    \item [(i)] Since $A_n$ is a symmetric matrix, we have that 
    \[
    \|A_n\|_{\rm op} = \sup\left\{|x^{\top} A_ny|:\ \|x\|_2=\|y\|_2=1,\ x,y\in\R^{N}\right\},
    \]
    where is the matrix (operator) norm induced by the Euclidian vector norm.   
    Now, let $x=(x_i)_{i\in [0:n-1]^d}$ and $y =(y_i)_{i\in [0:n-1]^d}.$ By Fourier inversion, we have that 
    $$
     a_{j-k} = \int_{[-\pi,\pi]^d} e^{-\ii(j-k)^\top\lambda} \alpha(\lambda) d\lambda.
    $$
    Therefore, 
    \begin{align*}
        |x^{\top}A_ny|& =  
        \left|\sum_{j \in [0:n-1]^d} \sum_{k\in [0:n-1]^d}
        \int_{[-\pi,\pi]^d} x_j e^{-\ii(j-k)^\top\lambda}y_k \alpha(\lambda) d\lambda \right|\\
         & \le \int_{[-\pi,\pi]^d }\left|\sum_{j\in [0:n-1]^d}x_je^{\ii j^\top \lambda}\right|\cdot|\alpha(\lambda)|
           \cdot\left|\sum_{k\in [0:n-1]^d}e^{\ii k^\top \lambda}y_k\right|d\lambda\\ 
           & \le \sup_{\lambda\in[-\pi,\pi]^d}|\alpha(\lambda)| \cdot 
           \left(\int_{[-\pi,\pi]^d}\left|\sum_{j\in [0:n-1]^d}x_je^{\ii j^\top\lambda}\right|^2d\lambda\right)^{1/2} \\
       & \hskip5cm  \times  \left(\int_{[-\pi,\pi]^d}\left|\sum_{k \in [0:n-1]^d}y_k e^{\ii k^\top \lambda}\right|^2d\lambda\right)^{1/2} \\ 
        & = (2\pi)^{d}\cdot\sup_{\lambda\in[-\pi,\pi]^d}|\alpha(\lambda)| \cdot \|x\|_2\cdot\|y\|_2,
    \end{align*}
    The second inequality follows from the Cauchy-Schwarz inequality and the last equality follows by 
    Parseval's identity, since the functions $\varphi_k(\lambda):= (2\pi)^{-d/2} e^{\ii k^\top \lambda},\ \lambda \in [-\pi,\pi],\ k\in [0:(n-1)]^d$,
    are orthonormal in $L^2([-\pi,\pi]^d;\C)$.
        
    \item [(ii)] If $A_n$ is also positive definite, then $A_n$ is invertible. 
     \noindent Since $\|\cdot\|_{\rm op}$ is the spectral norm, we have that \[
    \|A_n^{-1}\|_{\rm op}=\max\sigma_i(A_n^{-1}) = \max\frac{1}{\sigma_i(A_n)} = \sup\Big\{\frac{1}{x^{\top}A_nx}:\ \|x\|_2=1,x\in\R^N\Big\}
    \]
    where the last equality follows from Rayleigh quotient optimization results, when $A_n$ is positive definite. 
    
    As in part (1), with $z^\star$ denoting the complex conjugate of $z$, we have that 
    \begin{align*}
        |x^{\top}A_nx| & =x^{\top}A_nx =  
        \int_{[-\pi,\pi]^d}\left(\sum_{j\in [0:n-1]^d}x_je^{\ii j^\top\lambda}\right)\alpha(\lambda)\left(\sum_{k\in [0:n-1]^d}e^{\ii k^\top \lambda}y_k\right)d\lambda \\ 
        & = \int_{[-\pi,\pi]^d}\left|\sum_{j\in [0:n-1]^d}x_je^{-\ii j^\top\lambda}\right|^2\alpha(\lambda)d\lambda 
          \\ & \ge \inf_{\lambda\in[-\pi,\pi]^d}|\alpha(\lambda)|\int_{[-\pi,\pi]^d}\left|\sum_{j\in [0:n-1]^d}x_je^{-\ii j\lambda}\right|^2d\lambda
        \\ & = \Big(\sup_{\lambda\in[-\pi,\pi]^d}|\alpha^{-1}(\lambda)|\Big)^{-1}\cdot\|x\|_2^2\cdot (2\pi)^{d},
    \end{align*}
    and the result follows. 
    
    \item [(iii)] It follows that 
    \begin{equation*}
        \|A_n\|_{\rm F}^2 =\sum_{i \in [0:n-1]^d}\sum_{j \in [0:n-1]^d} |a_{i-j}|^2 = \sum_{k\in [-(n-1):(n-1)]^d} \prod_{i=1}^d (n-|k_i|)|a_{k}|^2,
    \end{equation*}
    where $[-(n-1):(n-1)]^d:= \{-(n-1),\cdots,n-1\}^d$.
    Thus, 
    \begin{equation*}
        \frac{1}{n^d}\|A_n\|_{\rm F}^2 = \sum_{k\in [-(n-1):(n-1)]^d} \prod_{i=1}^d \Big(1 -\frac{|k_i|}{n} \Big) |a_k|^2\le 
        \sum_{j\in \Z^d}|a_j|^2= (2\pi)^d \int_{[-\pi,\pi]^d}\alpha^2(\lambda)d\lambda,
    \end{equation*}
    by Parseval's identity.
\end{enumerate}
\end{proof}

\begin{lemma}\label{le:mean value theorem} Let  $B_0$ and $B_1$ be symmetric, positive definite $n\times n$ matrices such that $D:= B_1-B_0$ is 
non-negative definite. Let $P_0$ and $P_1$ be the probability distributions of zero-mean Gaussian vectors with covariance matrices $B_0$ and $B_1$, respectively. 
Then, there is a $\xi\in[0,1]$ such that 
\[
 {\rm KL}(P_1,P_0)\le \frac{1}{2}\|D\|_{\rm F}^2\|B_{\xi}^{-1}\|_{\rm op}^2,
\]
where $B_{\xi} := B_0 +\xi D$, and where $\|\cdot\|_{\rm F}$ stands for the matrix Frobenius norm and $\|\cdot\|_{\rm op}$ stands for the matrix operator norm induced by the Euclidean vector norm.
\end{lemma}
\begin{proof}
Since the data is assumed Gaussian, we can immediately obtain that 
\begin{equation}\label{eq:KL div}
    {\rm KL}(P_1,P_0) = \frac{1}{2} \left\{{\rm tr}(B_1B_0^{-1}- E)+\log|B_0|-\log|B_1|\right\},
\end{equation}
where ${\rm tr}(A)$ and $|A|$ are the trace and determinant of the matrix $A$ and $E$ is the identity matrix of dimensions $n\times n$. 
Notice that  $B_{\lambda}:= B_0 +\lambda D,\ \lambda\in[0,1]$ is a positive definite covariance matrix. The expression in \eqref{eq:KL div} can 
be rewritten as 
\begin{equation}\label{eq:KL div 2}
    {\rm KL}(P_1,P_0) = \frac{1}{2} \left\{{\rm tr}[B_1(B_0^{-1}- B_1^{-1})]+\log|B_0|-\log|B_1|\right\}.
\end{equation}
We define the function $\phi(\lambda)= {\rm tr}(B_1B_{\lambda}^{-1})+\log|B_{\lambda}|$ and note that 
by the intermediate value theorem, we have 
$$
{\rm KL}(P_1,P_0) = \frac{\phi(0)-\phi(1)}{2}  = -\frac{1}{2}\cdot \phi'(\xi),
$$
for some $\xi\in [0,1]$.  Using the following differentiation rules
\begin{align*}
     \frac{d}{d\lambda} A^{-1}(\lambda) & = -A^{-1}(\lambda) \left( \frac{d}{d\lambda} A(\lambda) \right) A^{-1}(\lambda),\\
     \frac{d}{d\lambda} \log| A(\lambda) | & = {\rm tr} \left(A^{-1}(\lambda) \frac{d}{d\lambda} A(\lambda) \right),
 \end{align*}
 and the fact that $d B_\lambda/d\lambda = D$, we obtain that \eqref{eq:KL div 2} becomes:
 \begin{align}\label{eq:KL div 3}
 \begin{split}
  {\rm KL}(P_1,P_0) & =  \frac{1}{2} \cdot {\rm tr} \left[ B_1B_{\xi}^{-1} D B_{\xi}^{-1} - D B_{\xi}^{-1}\right] \\
   & = \frac{1}{2} \cdot {\rm tr}\left[ (B_1-B_{\xi})B_{\xi}^{-1} D B_{\xi}^{-1}\right],
   \end{split}
 \end{align}
 for some $\xi\in [0,1]$. To estimate the rhs of \eqref{eq:KL div 3} we will use the following inequalities \citep[see, e.g.,][] {davies1973asymptotic}:
 \begin{align}\label{ineq: KL div le}
 |{\rm tr}(AB)| \le \|A\|_{\rm F} \cdot \|B\|_{\rm F}.
 \end{align}
 If the matrices $A$ and $B$ are symmetric, then 
 \begin{align}\label{ineq: KL div le 2}
  \|B A\|_{\rm F} = \|A B\|_{\rm F} \le \|A\|_{\rm op} \|B\|_{\rm F},
 \end{align}
 where 
 $$
 \|A\|_{\rm F} = \left( \sum_{i,j} a_{ij}^2 \right)^{1/2},\ \ \|A\|_{\rm op} = \sup\{ \|Ax\|_2;\ \|x\|_2 = 1,\ x\in \mathbb R^n\}.
 $$
 From \eqref{eq:KL div 3} with the help of \eqref{ineq: KL div le} and \eqref{ineq: KL div le 2} we obtain:
 \begin{align*}
  {\rm KL}(P_1,P_0) \le \frac{1}{2} \|B_{1} - B_{\xi}\|_{\rm F} \|D\|_{\rm F} \|B_{\xi}^{-1}\|_{\rm op}^2 \le \frac{1}{2}  \|D\|_{\rm F}^2 \|B_{\xi}^{-1}\|_{\rm op}^2. 
 \end{align*}
\end{proof}

\section{Proofs for Section \ref{s:CLT}}\label{s:clt_proof}
This section provides the proofs for Theorem \ref{thm:CLT}. As stated in Section \ref{s:CLT}, $\{X(t), t\in\bbR\}$ is a stationary Gaussian process in the complex Hilbert space $\bbH$, we observe $X$ at $t=k\delta_n, k=1, \ldots, n$, where $n\delta_n\to\infty$.  We will focus on the case $\delta_n\equiv 1$, which contains 
the main ideas of the proof. The proof for the general case follows from a straightforward adaptation of the special case.

\subsection{An overview of the proof of Theorem \ref{thm:CLT}} \label{s:clt_ts} 

As mentioned, we focus on the case  $\delta_n\equiv 1$, which is essentially the time-series setting (cf.\ Section \ref{s:ts}). Thus, we consider a discrete-time stationary process $X=\{X(t),t\in\mathbb{Z}\},$ where $X(t)$ are Gaussian elements of the complex Hilbert space $\bbH$. 
We explained in Section \ref{s:CLT} of the paper that the conditions (a)-(f) in Assumption \ref{a:clt} hold in this case if
\begin{align}\label{e:cond extension}
\sum_{x=-\infty}^{\infty}\left[\|C(x)\|_{\rm tr}+\|\check C(x)\|_{\rm tr}\right]<\infty,
\end{align}
which we assume below. Note that $C$ and $\check C$ are defined in \eqref{e:CXY} and \eqref{e:CXY_2} respectively.
Recall that $\bbX$ denotes the Hilbert space of Hilbert-Schmidt operators ${\cal A}:\bbH \to \bbH$, equipped with the HS-inner product
$\langle {\cal A},{\cal B}\rangle_{\rm HS}:= {\rm trace}({\cal B}^*{\cal A}),\ {\cal A},\ {\cal B}\in \bbX$, and corresponding norm
$\|{\cal A}\|_{\rm HS} = \langle {\cal A},{\cal A}\rangle_{\rm HS}^{1/2}$. The spectral and pseudo spectral density functions in this case are given by
\begin{align*}
f(\theta) = \frac{\delta}{2\pi}\sum_{k=-\infty}^\infty e^{\ii k^\top \theta\delta} C(k), \ \check f(\theta) = \frac{\delta}{2\pi}\sum_{k=-\infty}^\infty e^{-\ii k^\top \theta\delta} \check C(k), \ \theta\in[-\pi,\pi].
\end{align*}
Also,
\begin{align*}
\CT_n(\theta) := \sqrt{\frac{n}{\Delta_n}}\left(\hat f_n(\theta)-\mathbb{E}\hat f_n(\theta)\right),
\end{align*}
where 
\begin{align*}
\hat f_n(\theta) = \frac{1}{2\pi n}\sum_{i,j=1}^ne^{\mathbbm{i}(i-j)\theta}X(i)\otimes X(j) K\left(\frac{i-j}{\Delta_n}\right), \ \theta\in [-\pi,\pi].
\end{align*}

For this special setting, we will establish the following theorem:

\begin{thm}\label{thm:CLT_ts} 
Let $X=\{X(t),t\in\mathbb{Z}\}$ be a stationary Gaussian process of the complex Hilbert space $\bbH$. Let $\Delta_n\to\infty, \Delta_n/n\to 0$, and 
assume that \eqref{e:cond extension} and Assumption \ref{a:kernel'} hold. 
Define
\begin{align*}
\CT_n(\theta):= \sqrt{\frac{n}{\Delta_n}}\left[\hat f_n(\theta)-\mathbb{E}\hat f_n(\theta)\right], \  \ \theta\in\bbR,
\end{align*}
where $\hat f_n(\theta)$ is given in \eqref{def: f^ts_1}.
Then, for any $\theta \in [-\pi, \pi]$,
\begin{align*}
\CT_n(\theta) \stackrel{d}{\to}\CT(\theta) \ \mbox{ in $\bbX$},
\end{align*}
where $\CT(\theta)$ is a zero-mean Gaussian element of $\bbX$, such that for every finite collection $\{g_{\ell},\ell=1,\hdots,m\},$ and positive numbers $\{a_\ell, \ell=1,\hdots,m\},$ 
\begin{align*} 
\begin{split}
& {\rm Var}\left(\sum_{\ell=1}^ma_{\ell}\left\langle\CT(\theta) g_\ell,g_\ell\right\rangle\right) \\
& =\|K\|_2^2 \sum_{\ell_1,\ell_2=1}^ma_{\ell_1}a_{\ell_2}\left[\left|\left\langle f(\theta)g_{\ell_2},g_{\ell_1}\right\rangle\right|^2 +I_{(\theta=0,\pm\pi)}
\left|\left\langle \check f(\theta)\overline{g_{\ell_2}},g_{\ell_1}\right\rangle\right|^2\right].
\end{split}
\end{align*}
\end{thm}

The following proposition describes the roadmap for proving this result. For simplicity of notation, we will henceforth suppress the argument $\theta$ in $\CT_n(\theta)$ since it is fixed.

\begin{prop} \label{e:weak_conv}
Let the assumptions of Theorem \ref{thm:CLT_ts} hold. Also, 
let $\{e_i, i \ge 1\}$ be a CONS of $\bbH$. 
Assume that
\begin{enumerate} [label=(\roman*)]
\item for any $\epsilon,\delta>0$, there exists $u\in\Z_+$ such that
\begin{align*}
\sup_{n\ge 1} \pr(\|({\rm I}-\Pi_u)\CT_n\|_\HS > \epsilon) < \delta,
\end{align*} 
where $\Pi_u:\bbX\to \bbX$ is the orthogonal projection operator on $\bbX_u :={\rm span}(e_i\otimes e_j, i, j\le u)$, and
\vskip.2cm
\item
for all $a_\ell\in\R$ and $g_\ell\in\bbH$, 
we have
\begin{align} \label{e:gaussian_moment}
    \mathbb{E}\left[\sum_{\ell=1}^m a_\ell \left\langle \CT_ng_{\ell},g_{\ell}\right\rangle\right]^k \sim 
    \begin{cases}
    \mathcal{O}\left(\left(\frac{\Delta_n}{n}\right)^{\frac{1}{2}}\right) &,\ k\quad\mbox{odd}, \\
     (k-1)!!\left[\sigma_{a,g}^2 \right]^{\frac{k}{2}} &,\   k \quad\mbox{even},
    \end{cases}
\end{align}
where 
\[
\sigma_{a,g}^2:= \|K\|_2^2 \sum_{\ell_1,\ell_2=1}^m a_{\ell_1}a_{\ell_2}\left(\left| \left\langle f(\theta)g_{\ell_1},g_{\ell_2}\right\rangle\right|^2+1_{\{0,\pm\pi\}} (\theta)\left|\left\langle \check f(\theta )\overline{g_{\ell_1}},g_{\ell_2}\right\rangle\right|^2\right).
\]
\end{enumerate} 
Then there exists a Gaussian process $\CT$ in $\bbX$ that fulfills the description of Theorem \ref{thm:CLT_ts}.
\end{prop}
\begin{proof}
First, we state a useful identity for a complex Hilbert space.
Write
$$
\langle \CT_n, e_i\otimes e_j\rangle_\HS = \langle \CT_ne_j,e_i\rangle_\bbH =: \CT_n(e_j,e_i),
$$
By Lemma A.8 of \cite{shen2020tangent},
\begin{align} \label{e:polar_identity}
\begin{split}
	\CT_n(e_j,e_i) = & \frac{\ii-1}{2} (\CT_n(e_j,e_j)+\CT_n(e_i,e_i) )+ \frac{1}{2} \CT_n(e_j+e_i,e_j+e_i) \\
	 & -\frac{\ii}{2} \CT_n(\ii e_j+e_i, \ii e_j+e_i).
	 \end{split}
	\end{align}
Also, recall that $\{e_i\otimes e_j, i, j \ge 1\}$ is a CONS of $\bbX$. Thus, (i) implies the flat concentration condition of Condition 1 of Theorem 7.7.4 of \cite{hsing2015theoretical}. It follows from (ii), applying \eqref{e:polar_identity} plus Markov's inequality, that Condition 2 of Theorem 7.7.4 of \cite{hsing2015theoretical} also holds. Thus, $\{\CT_n, n \ge 1\}$ is tight and hence relatively compact. To show that $\CT_n$ converges in distribution to some $\CT$, it suffices to show that if $\CT_{n'}\cid \CT$ along some subsequence $\{n'\}$, then $\CT$ does not depend on the subsequence. Now, by the continuous mapping theorem, \eqref{e:gaussian_moment} and a standard uniform integrability argument, we have
$$
\mathbb{E}\left[\sum_{\ell=1}^m a_\ell \left\langle \CT_{n'}g_{\ell},g_{\ell}\right\rangle\right]^k \to 
\mathbb{E}\left[\sum_{\ell=1}^m a_\ell \left\langle \CT g_{\ell},g_{\ell}\right\rangle\right]^k
\ \ \mbox{for all $k$},
$$
where the limiting moments entail that
$\sum_{\ell=1}^m a_\ell\langle \CT g_{\ell},g_{\ell}\rangle$ is distributed as $N(0, \sigma_{a,g}^2)$ \citep[cf. 
Theorem 30.1 of][]{billingsley2012probability}.
Relation \eqref{e:polar_identity} shows that the finite-dimensional distributions
of the real Gaussian process $\{ \langle \CT g_{\ell},g_{\ell}\rangle,\ g_\ell \in \bbH\}$ determine the finite-dimensional distributions of $\{ \langle \CT g,h \rangle,\ g,h\in\bbH\}$, which in turn characterize the law of the $\mathbb X$-valued random element $\CT$. 
The result thus follows.
\end{proof}

We will complete the proof of Theorem \ref{thm:CLT_ts} by verifying the conditions (i) and (ii) of Proposition \ref{e:weak_conv}, which will be established in Sections \ref{sec:flat-concentration} and \ref{sec:weak-conv}, respectively.

\subsection{Verifying (i) of Proposition \ref{e:weak_conv}} \label{sec:flat-concentration}

In this subsection, we establish that (i) of Proposition \ref{e:weak_conv} holds for $\{\CT_n\}$ under the assumptions of the Proposition.  This property is
known as the flat concentration of $\{\CT_n\}$. 

By Markov's inequality, 
\begin{align*}
\pr(\|({\rm I}-\Pi_u)\CT_n\|_\HS > \epsilon) \le \epsilon^{-2} \E \|({\rm I}-\Pi_u)\CT_n\|_\HS^2.
\end{align*}
Since
\begin{align} \label{e:markov}
\E \|({\rm I} -\Pi_u)\CT_n\|_\HS^2 = \sum_{k\vee\ell > u} \E |\langle \CT_n,e_k\otimes e_{\ell}\rangle_{\rm HS}|^2,
\end{align}
it is sufficient to show that
\begin{align} \label{e:akl}
\sum_{k,\ell} \sup_{n}\mathbb{E}\left|\left\langle \CT_n,e_k\otimes e_{\ell} \right\rangle_{\rm HS}\right|^2 < \infty,
\end{align}
which implies (i) of Proposition \ref{e:weak_conv} by \eqref{e:markov}.

Without loss of generality suppose that the CONS $\{e_j,\ j\in\mathbb {N}\}$ of $\bbH$ is {\em real} and thus
$X(t) = \sum_{j\in\mathbb N} X_j(t) e_j$, where $X_j(t) = \langle X(t),e_j\rangle_{\bbH}$ are complex zero-mean Gaussian random variables. 
Also, let 
\begin{align} \label{e:C_proj}
C_{k,\ell}(x) = \langle C(x),e_k\otimes e_{\ell}\rangle_{\rm HS} = \langle C(x) e_\ell, e_k\rangle.
\end{align}
It follows that
\begin{align*}
    & \left\langle \CT_n,e_k\otimes e_{\ell} \right\rangle_{\rm HS} \\
    & = \frac{(2\pi)^{-1}}{\sqrt{n\Delta_n}}\sum_{i,j=1}^n \left\langle X(i)\otimes X(j) - C(i-j),e_k\otimes e_{\ell}\right\rangle_{\rm HS}  e^{\mathbbm{i}\theta(i-j)}K\left(\frac{i-j}{\Delta_n}\right)\\ 
    & = \frac{(2\pi)^{-1}}{\sqrt{n\Delta_n}}\sum_{i,j=1}^n\left[  X_k(i) \overline{X_{\ell}(j)}-C_{k,\ell}(i-j)\right] e^{\mathbbm{i}\theta(i-j)}K\left(\frac{i-j}{\Delta_n}\right),
\end{align*}
where $C_{k,\ell}(i-j) = \mathbb{E}[X_k(i)\overline{X_{\ell}(j)}].$
By Lemma \ref{le:complex isserlis}, we have 
\begin{align*}
 \E[ X_k(i_1) \overline{X_\ell(j_1) X_k(i_2)} X_\ell(j_2) ] &= C_{k,\ell}(i_1-j_1) \overline{C_{k,\ell}(i_2-j_2)} + \check C_{k,\ell}(i_1-j_2)\overline{\check C_{k,\ell}(i_2-j_1)}\\
 & +  C_{k,k}(i_1-i_2)
 \overline{C_{\ell,\ell}(j_1-j_2)}.
 \end{align*} 
Thus,
\begin{align}\label{e:flat-concentration-check-1}
\begin{split}
& \mathbb{E}\left|\left\langle \CT_n,e_k\otimes e_{\ell} \right\rangle_{\rm HS} \right|^2 \\
& = \frac{(2\pi)^{-2}}{n\Delta_n}\sum_{i_1,j_1}\sum_{i_2,j_2} e^{\mathbbm{i}\theta (i_1-j_1-i_2+j_2)}K\left(\frac{i_1-j_1}{\Delta_n}\right)K\left(\frac{i_2-j_2}{\Delta_n}\right) \\ 
& \quad\times  \mathbb{E}\left\{\left[X_k(i_1)\overline{X_{\ell}(j_1)}-C_{k,\ell}(i_1-j_1)\right]\left[\overline{X_k(i_2)}X_{\ell}(j_2)-\overline{C_{k,\ell}(i_2-j_2)}\right]\right\}
\\  
& = \frac{(2\pi)^{-2}}{n\Delta_n}\sum_{i_1,j_1}\sum_{i_2,j_2} e^{\mathbbm{i}\theta (i_1-j_1-i_2+j_2)}K\left(\frac{i_1-j_1}{\Delta_n}\right)K\left(\frac{i_2-j_2}{\Delta_n}\right)
\\
& \qquad \qquad\qquad\qquad\qquad\times C_{k,k}(i_1-i_2)\overline{C_{\ell,\ell}(j_1-j_2)}
\\ 
& \quad +\frac{(2\pi)^{-2}}{n\Delta_n}\sum_{i_1,j_1}\sum_{i_2,j_2} e^{\mathbbm{i}\theta (i_1-j_1-i_2+j_2)}K\left(\frac{i_1-j_1}{\Delta_n}\right)K\left(\frac{i_2-j_2}{\Delta_n}\right)
\\ 
& \qquad \qquad\qquad\qquad\qquad\times \check C_{k,\ell}(i_1-j_2)\overline{\check C_{k,\ell}(i_2-j_1)}\\ 
& =: A_{k,\ell}+B_{k,\ell}.
\end{split}
\end{align}
We start with $A_{k,\ell}$. With the change of variables 
\begin{align*}
    x_1=i_1-i_2, &\ \ x_2=j_1-j_2, \\
    y_1 = i_1-j_1, & \ \ y_2 = i_1,
\end{align*}
we obtain
\begin{align*}
A_{k,\ell} & = \frac{(2\pi)^{-2}}{n\Delta_n}\sum_{x_1,x_2=1-n}^{n-1}C_{k,k}(x_1)\overline{C_{\ell,\ell}(x_2)}e^{\mathbbm{i}\theta(x_1-x_2)}\\ & \qquad\qquad\times \sum_{y_1=(-\Delta_n)\vee(1-n+x_1-x_2)}^{\Delta_n\wedge(n-1+x_1-x_2)}K\left(\frac{y_1}{\Delta_n}\right)K\left(\frac{-x_1+x_2+y_1}{\Delta_n}\right)\sum_{y_2=1\vee (1+y_1)}^{n\wedge (n+y_1)}1
\\ & = \frac{(2\pi)^{-2}}{n\Delta_n}\sum_{x_1,x_2=1-n}^{n-1}C_{k,k}(x_1)\overline{C_{\ell,\ell}(x_2)}e^{\mathbbm{i}\theta(x_1-x_2)}\\ & \qquad\qquad\times\sum_{y_1=(-\Delta_n)\vee(1-n+x_1-x_2)}^{\Delta_n\wedge(n-1+x_1-x_2)}K\left(\frac{y_1}{\Delta_n}\right)K\left(\frac{-x_1+x_2+y_1}{\Delta_n}\right)({n-|y_1|}).
\end{align*}
Thus, with $\|K\|_\infty:= \max_{t} |K(t)|$, we obtain
\begin{align*}
    |A_{k,\ell}| & \le \frac{\|K\|_{\infty}^2}{2\pi^2}\sum_{x_1,x_2=1-n}^{n-1}|C_{k,k}(x_1)||C_{\ell,\ell}(x_2)|\\
    & \le \frac{\|K\|_{\infty}^2}{2\pi^2}\sum_{x_1,x_2=-\infty}^{\infty}|C_{k,k}(x_1)||C_{\ell,\ell}(x_2)| =: \alpha_{k,\ell}.
\end{align*}
By \eqref{e:C_proj} and (ii) of Lemma \ref{le:trace norm alt def}, we have
\begin{align}  \label{e:akl-sum}
    \sum_{k,\ell}\alpha_{k,\ell}\le \frac{\|K\|_{\infty}^2}{2\pi^2}\left(\sum_{x=-\infty}^{\infty}\|C(x)\|_{\rm tr}\right)^2<\infty.
\end{align}

We now turn to $B_{k,\ell}$ in \eqref{e:flat-concentration-check-1}. With the change of variables 
\begin{align*}
    x_1 = i_1-j_2, & \ \ x_2=i_2-j_1,\\ 
    y_1 = i_1-j_1, & \ \ y_2=i_1,
\end{align*}
\begin{align*}
B_{k,\ell} & = \frac{(2\pi)^{-2}}{n\Delta_n}\sum_{x_1,x_2=1-n}^{n-1} \check C_{k,\ell}(x_1)\overline{\check  C_{k,\ell}(x_2)}e^{\mathbbm{i}\theta(-x_1-x_2+2y_1)}\\ & \qquad\qquad\times\sum_{y_1=(-\Delta_n)\vee(1-n+x_1+x_2)}^{\Delta_n\wedge(n-1+x_1+x_2)}K\left(\frac{y_1}{\Delta_n}\right)K\left(\frac{x_1+x_2-y_1}{\Delta_n}\right)\sum_{y_2=1\vee(1+y_1)}^{n\wedge (n+y_1)}1\\ 
& = \frac{(2\pi)^{-2}}{n\Delta_n}\sum_{x_1,x_2=1-n}^{n-1} \check C_{k,\ell}(x_1)\overline{\check C_{k,\ell}(x_2)}e^{\mathbbm{i}\theta(-x_1-x_2+2y_1)}\\ & \qquad\qquad\times\sum_{y_1=(-\Delta_n)\vee(1-n+x_1+x_2)}^{\Delta_n\wedge(n-1+x_1+x_2)}K\left(\frac{y_1}{\Delta_n}\right)K\left(\frac{x_1+x_2-y_1}{\Delta_n}\right)({n-|y_1|}).
\end{align*}
Thus, 
\begin{align*}
    |B_{k,\ell}| & \le \frac{\|K\|_{\infty}^2}{2\pi^2}\sum_{x_1,x_2=1-n}^{n-1}\left|\check C_{k,\ell}(x_1)\right|\left|\check C_{k,\ell}(x_2)\right| \\
    & \le \frac{\|K\|_{\infty}^2}{2\pi^2}\sum_{x_1,x_2=1-n}^{n-1}\left|\check C_{k,\ell}(x_1)\right|\left|\check C_{k,\ell}(x_2)\right| =:\beta_{k,\ell}.
\end{align*}
Applying the Cauchy-Schwarz inequality and in view of the definition of the Hilbert-Schmidt inner product,
\begin{align}\label{e:bkl-sum}
\begin{split}
    \sum_{k,\ell} \beta_{k,\ell} & \le \frac{\|K\|_{\infty}^2}{2\pi^2} \sum_{x_1,x_2=-\infty}^{\infty}
        \sum_{k,\ell}\left| \check C_{k,\ell}(x_1)\right|\left|\check C_{k,\ell}(x_2)\right|  \\
    & \le \frac{\|K\|_{\infty}^2}{2\pi^2} \sum_{x_1,x_2=-\infty}^{\infty} \sqrt{\sum_{k,\ell}\left|\check C_{k,\ell}(x_1)\right|^2}\sqrt{\sum_{k,\ell}\left|\check C_{k,\ell}(x_2)\right|^2}   \\ 
    & = \frac{\|K\|_{\infty}^2}{2\pi^2} \left(\sum_{x=-\infty}^{\infty}{\|\check C(x)\|_{\rm HS}}\right)^2\le 
      \frac{\|K\|_{\infty}^2}{2\pi^2} \left(\sum_{x=-\infty}^{\infty}{\|\check C(x)\|_{\rm tr}}\right)^2<\infty,
      \end{split}
\end{align}
Since the upper bounds $\alpha_{k,\ell}$ and $\beta_{k,\ell}$ do not depend on $n$, we have
\begin{align*}
 \sup_{n}\mathbb{E}\left|\left\langle \CT_n,e_k\otimes e_{\ell} \right\rangle_{\rm HS}\right|^2  \le \alpha_{k,\ell}+\beta_{k,\ell},
\end{align*}
and Relations \eqref{e:akl-sum} and \eqref{e:bkl-sum} imply  \eqref{e:akl}.

\subsection{Verifying (ii) of Proposition \ref{e:weak_conv}} \label{sec:weak-conv} 

The proof of (ii) is quite lengthy, and constitutes the core of the central limit theorem proof. In Section \ref{e:moment_scalar}, we will first focus on showing (ii) for the case that $X$ is scalar, i.e.,\ $X$ take values in $\bbC$. There we will take advantage of this simple setting to explain the ideas of the proof. In Section \ref{s:moment_general} , we will prove the proposition for the general case of $X\in\bbH$. 

\subsubsection{The scalar case} \label{e:moment_scalar}

In this section, we focus on a zero-mean,  stationary Gaussian time-series taking values in $\bbC$ and compute the moments of $\CT_n$. The purpose of this section is to develop
technical tools for the general moment calculations needed to prove (ii).

In this setting, $C(t-s) = \mathbb{E}[X(t)\overline{X(s)}]$, $\check C(t-s):= \E[ X(t)X(s)]$, and 
\begin{equation}\label{e:f-spec-CLT}
f(\theta) = \frac{1}{2\pi}\sum_{x=-\infty}^{\infty}C(x)e^{\mathbbm{i}x\theta}.
\end{equation}
Also, \eqref{e:cond extension} becomes
\begin{equation}\label{e:Assumption-C-prime}
\sum_{x=-\infty}^\infty \Big[ |C(x)| + |\check C(x)| \Big] <\infty.
\end{equation}
Observe that since $C(x) = \overline{C(-x)}$ in this case, we have that $f(\theta)$ is real, even though the process $\{X_t,\ t\in\mathbb Z\}$ is 
complex-valued. We shall also need the so-called pseudo-spectral density, defined as
\begin{equation}\label{e:f-pseudo-spec-CLT}
\check f(\theta) = \frac{1}{2\pi}\sum_{x=-\infty}^{\infty}\check C(x)e^{-\mathbbm{i}x\theta}.
\end{equation}
Recall also that the spectral density estimator is

\[
\hat f_n(\theta) = \frac{1}{2\pi n}\sum_{i,j=1}^nX_i\overline{X_j}K_{\theta}(i-j),
\]
where 
$$
K_\theta(y) = e^{\mathbbm{i}y\theta}K(y/\Delta_n).
$$
The following proposition gives the asymptotic expression of $\mathbb{E}(\CT_n ^k)$ for the scalar case.

\begin{prop}\label{prop: central moments}
Assume that the conditions of Proposition \ref{e:weak_conv} hold for the setting $\bbH=\bbC$. 
Then, as $n\to\infty$,
\begin{equation*}
\mathbb{E}(\CT_n^k ) =  \begin{cases}\mathcal{O}\left(\left(\frac{\Delta_n}{n}\right)^{\frac{1}{2}}\right) &,\ k\quad\mbox{odd}, \\
(1+o(1)) (k-1)!!\left[ (f(\theta)^2+ 1_{\{0,\pm \pi\}}(\theta)|\check f(\theta)|^2)\|K\|_2^2\right]^{\frac{k}{2}} &,\ k \quad\mbox{even},
\end{cases}
\end{equation*}
where $f(\theta)$ and $\check f(\theta)$ are as in \eqref{e:f-spec-CLT} and \eqref{e:f-pseudo-spec-CLT}.
\end{prop}

\begin{proof}
\vspace{.3cm}\noindent
Assume without loss of generality that the support of $K$ is $[-1,1]$.

\medskip
\noindent
\underline{The case $k=2$}
\vskip.3cm\noindent
By Lemma \ref{le:extra Isserlis} (for $N=0,\ M=2$, i.e., see \eqref{e:M=2,N=0}), 
\begin{align*}
  (2\pi)^2  \mathbb{E}(\CT_n^2)  & = \frac{1}{n\Delta_n}\sum_{i_1,j_1=1}^n\sum_{i_2,j_2=1}^nK_{\theta}\left(i_1-j_1\right)K_{\theta}\left(i_2-j_2\right)\\ 
    & \qquad\qquad\times \mathbb{E}\left[\left(X_{i_1}\overline{X_{j_1}}-C(i_1-j_1)\right)\left(X_{i_2}\overline{X_{j_2}}-C(i_2-j_2)\right)\right]\\ 
    & = \frac{1}{n\Delta_n}\sum_{i_1,j_1=1}^n\sum_{i_2,j_2=1}^nK_{\theta}\left(i_1-j_1\right)K_{\theta}\left(i_2-j_2\right)\\ & \qquad\qquad\times\left[C(i_1-j_2)C(i_2-j_1)+\check C(i_1-i_2)\overline{\check C(j_1-j_2)}\right] \\
    & =: \widetilde A_2 +  \overline{A}_2,
\end{align*}
where 
\begin{equation*} 
    \widetilde A_2 := \frac{1}{n\Delta_n}\sum_{i_1,j_1=1}^n\sum_{i_2,j_2=1}^nK_{\theta}\left(i_1-j_1\right)K_{\theta}\left(i_2-j_2\right)\cdot C(i_1-j_2)C(i_2-j_1),
\end{equation*}
and 
\begin{equation*}
    \overline{A}_2 := \frac{1}{n\Delta_n}\sum_{i_1,j_1=1}^n\sum_{i_2,j_2=1}^nK_{\theta}\left(i_1-j_1\right)K_{\theta}\left(i_2-j_2\right)\cdot \check C(i_1-i_2)\overline{\check C(j_1-j_2)}.
\end{equation*}
The two terms have somewhat different properties and we start with $\widetilde A_2$.
With the change of variables 
\begin{align*}
    x_1 &= i_1-j_2,\quad x_2 =i_2-j_1,\\
    y_1 &= i_1-j_1,\quad  y_2=i_1,
\end{align*}
we have 
\begin{align*}
    \widetilde A_2    & = \frac{1}{n\Delta_n}\sum_{x_1,x_2=-n+1}^{n-1}C(x_1)C(x_2)e^{\mathbbm{i}x_1\theta}e^{\mathbbm{i}x_2\theta}\\ 
    & \qquad \times  \sum_{y_1=(1-n)\vee (1-n+x_1+x_2)}^{(n-1)\wedge(n-1+x_1+x_2)}K\left(\frac{y_1}{\Delta_n}\right)K\left(\frac{x_1+x_2-y_1}{\Delta_n}\right) \sum_{y_2=1\vee (1+y_1)}^{n\wedge(n+y_1)}1\\ 
    & = \frac{1}{n\Delta_n} \sum_{x_1,x_2=-n+1}^{n-1}C(x_1)C(x_2)e^{\mathbbm{i}x_1\theta}e^{\mathbbm{i}x_2\theta}\\ 
    & \qquad\times\sum_{y_1=(1-n)\vee(1-n+x_1+x_2)}^{(n-1)\wedge(n-1+x_1+x_2)}K\left(\frac{y_1}{\Delta_n}\right)
     K\left(\frac{x_1+x_2-y_1}{\Delta_n}\right)\cdot ({n-|y_1|})\\
    & = \frac{1}{\Delta_n}\sum_{|x_1|\vee|x_2|\le L}C(x_1)C(x_2)e^{\mathbbm{i}x_1\theta}e^{\mathbbm{i}x_2\theta}\\ & \qquad\qquad\times \sum_{y_1=(1-n)\vee(1-n+x_1+x_2)}^{(n-1)\wedge(n-1+x_1+x_2)}K \left(\frac{y_1}{\Delta_n}\right)K \left(\frac{x_1+x_2-y_1}{\Delta_n}\right)\frac{{n-|y_1|}}{n}\\ 
    & + \frac{1}{\Delta_n}\sum_{|x_1|\vee|x_2|\ge L}C(x_1)C(x_2)e^{\mathbbm{i}x_1\theta}e^{\mathbbm{i}x_2\theta}\\ & \qquad\qquad\times\sum_{y_1=(1-n)\vee(1-n+x_1+x_2)}^{(n-1)\wedge(n-1+x_1+x_2)}K\left(\frac{y_1}{\Delta_n}\right)K\left(\frac{x_1+x_2-y_1}{\Delta_n}\right)\frac{{n-|y_1|}}{n} \\
    & =:B_1 +B_2, 
\end{align*}
for some $L=L_n\to\infty$ and $L=o(\Delta_n).$
One can easily see that 
\begin{align*}
    |B_2|& \le \frac{\|K\|_\infty}{\Delta_n}\sum_{|x_1|\vee|x_2|\ge L}|C(x_1)||C(x_2)|\sum_{y_1=-n+1}^{n-1}K\left(\frac{y_1}{\Delta_n}\right)\\ 
    & \le 2\|K\|_\infty \sum_{|x_1|\vee|x_2|\ge L}|C(x_1)| |C(x_2)|=o\left(1\right),\ \mbox{as}\ L\to\infty,
\end{align*}
by \eqref{e:Assumption-C-prime}.
Now, adding and subtracting the same term in $B_1$, one obtains that $B_1=C_1+C_2,$ where 
\begin{align*}
 C_1 & := \frac{1}{\Delta_n} \sum_{|x_1|\vee|x_2|\le L}C(x_1)C(x_2)e^{\mathbbm{i}x_1\theta}e^{\mathbbm{i}x_2\theta} 
     \sum_{y_1=(1-n)\vee(1-n+x_1+x_2)}^{(n-1)\wedge(n-1+x_1+x_2)}K^2\left(\frac{y_1}{\Delta_n}\right)\frac{{n-|y_1|}}{n}
\end{align*}
and
\begin{align*}
 C_2 & := \frac{1}{\Delta_n}\sum_{|x_1|\vee|x_2|\le L}C(x_1)C(x_2)e^{\mathbbm{i}x_1\theta}e^{\mathbbm{i}x_2\theta} \\ 
    & \qquad\times\sum_{y_1=(1-n)\vee(1-n+x_1+x_2)}^{(n-1)\wedge(n-1+x_1+x_2)}K\left(\frac{y_1}{\Delta_n}\right)\left[K\left(\frac{x_1+x_2-y_1}{\Delta_n}\right)-K\left(\frac{y_1}{\Delta_n}\right)\right] 
    \times \frac{{n-|y_1|}}{n}.
\end{align*}

We examine $C_1$ first.
Observe first that the inner sum over $y_1$ is confined to
$-\Delta_n \le y_1\le \Delta_n$, since $K$ is supported on $[-1,1]$.  Moreover, since $L=o(\Delta_n)$ and $\Delta_n=o(n)$,
for all $|x_1|\vee |x_2|\le L$, and all sufficiently large $n$, we have that $(1-n)\vee (1-n+x_1+x_2)\le -\Delta_n$ and 
$\Delta_n\le (n-1)\wedge (n-1 + x_1 + x_2)$.  This means, that the inner summation in the definitions of $C_1$ and $C_2$
is over the range $[-\Delta_n,\Delta_n]$ and it does not depend on $x_1$ and $x_2$.  That is, for all sufficiently large $n$,
\begin{align*}
 C_1 & = \sum_{|x_1|\vee|x_2|\le L}C(x_1)C(x_2)e^{\mathbbm{i}x_1\theta}e^{\mathbbm{i}x_2\theta} \times \frac{1}{\Delta_n} 
 \sum_{y_1=-\Delta_n}^{\Delta_n} K^2\Big( \frac{y_1}{\Delta_n}\Big)\\
    &  \sim 4\pi^2 f(\theta)^2\int_{-1}^{1}K^2(y)dy,\ \ \mbox{ as }n\to\infty,
\end{align*}
where the last relation follows from the Riemann integrability of $K^2$ and the fact that $\sum_{|x|\le L} C(x)e^{\ii x \theta} \to 2\pi f(\theta)$,
as $L\to\infty$.

Now, focus on the term $C_2$.  Using the facts that $K$ is an even function and 
\[|K(x)-K(y)|\le  c|x-y|, \] (since $K'$ is bounded), we get $|K((x_1+x_2 -y_1)/\Delta_n) - K(y_1/\Delta_n)|\le 2c L/\Delta$, for all $|x_1|\vee |x_2|\le L$.
Thus, by Condition \eqref{e:Assumption-C-prime} and the Riemann integrability of $K$, we obtain
\begin{align*}
    |C_2|&\le\left(\sum_{|x|\le L}|C(x)|\right)^2 \frac{1}{\Delta_n} \sum_{y=-\Delta_n}^{\Delta_n}K\left(\frac{y}{\Delta_n}\right) \frac{2Lc}{\Delta_n} \\
    & \sim\frac{2Lc}{\Delta_n}\left(\sum_{x=-\infty}^\infty|C(x)|\right)^2\int_{u=-1}^1K(u)du=o\left(1\right),
\end{align*}
since $L = o(\Delta_n)$.
Summarizing, we have that for all $\theta$ (including $\theta=0$ and $\theta = \pm \pi$)
\[\widetilde A_2 =C_1+C_2+B_2\sim 4\pi^2 f(\theta)^2\int_{-1}^{1}K^2(u)du.\]

We next consider $\overline{A}_2$. Similar to the derivation for $\widetilde A_2$, with the change of variables 
\begin{align*}
    x_1 &= i_1-i_2,\quad x_2 =j_1-j_2, \\
    y_1 &= i_1-j_1,\quad  y_2=i_1,
\end{align*}
we get 
\begin{align*}
     & \overline{A}_2 = \frac{1}{n\Delta_n}\sum_{x_1,x_2=-n+1}^{n-1}\check C(x_1)\overline {\check C(x_2)}e^{-\mathbbm{i}x_1\theta}e^{\mathbbm{i}x_2\theta}\\ 
     & \qquad\cdot \sum_{y_1=(1-n)\vee(1-n+x_1-x_2)}^{(n-1)\wedge(n-1+x_1-x_2)}K\left(\frac{y_1}{\Delta_n}\right)K\left(\frac{-x_1+x_2+y_1}{\Delta_n}\right)e^{\mathbbm{i}2y_1\theta} \sum_{y_2=1\vee(1+y_1)}^{n\wedge(n+y_1)}1\\ 
     & = \frac{1}{\Delta_n}\sum_{x_1,x_2=-n+1}^{n-1}\check C(x_1)\overline{\check C(x_2)}e^{-\mathbbm{i}x_1\theta}e^{\mathbbm{i}x_2\theta}\\ 
     & \qquad\cdot\sum_{y_1=(1-n)\vee(1-n+x_1-x_2)}^{(n-1)\wedge(n-1+x_1-x_2)}K\left(\frac{y_1}{\Delta_n}\right)K\left(\frac{-x_1+x_2+y_1}{\Delta_n}\right)e^{\mathbbm{i}2y_1\theta}\cdot\frac{{n-|y_1|}}{n}. 
\end{align*}
Observe first that for $\theta = \pm\pi$ or $\theta=0$, we have $e^{\ii 2y_1 \theta } = 1,\ y_1\in\mathbb Z$
 and for the term $\overline{A}_2$ with the same argument as 
for the term $\widetilde A_2$, we obtain
$$
\overline{A}_2  \sim 4\pi |\check f(\theta)|^2 \|K\|_2^2,\ \ \mbox{ where }
 \ \  \check f(\theta) = \frac{1}{2\pi} \sum_{x=-\infty}^\infty \check C(x) e^{-\ii x \theta},\ \ \theta\in \{0,\pm \pi\}.
$$

Suppose now $\theta\not= 0$ and $\theta\not = \pm\pi$, so that the term $e^{\ii 2y_1\theta}$ is present.
By adding and subtracting a term, we have that $\overline{A}_2=D_1+D_2,$ where $D_2$ is defined in \eqref{e:D2-term} below and
\begin{align} \label{e:D1-term}
\begin{split}
    D_1 
    & :=\frac{1}{\Delta_n}\sum_{x_1,x_2=-n+1}^{n-1}\check C(x_1) \overline{\check C(x_2)} e^{-\mathbbm{i}x_1\theta}e^{\mathbbm{i}x_2\theta}  \\
    & \hspace{2cm} \cdot \sum_{y_1=(1-n)\vee(1-n+x_1-x_2)}^{(n-1)\wedge(n-1+x_1-x_2)}K^2\left(\frac{y_1}{\Delta_n}\right)e^{\mathbbm{i}2y_1\theta}
     \frac{{n-|y_1|}}{n}  \\ 
    & = \mathcal{O}\left(\frac{1}{\Delta_n}\right).
    \end{split}
\end{align}
Indeed, write
$$
w_n(y) = K^2\left(\frac{y}{\Delta_n}\right) \frac{n-|y|}{n} 
$$
and consider, for any $c_1,c_2 \in [1,\Delta_n]$,
\begin{align*}
& \sum_{y=c_1}^{c_2} w_n(y) e^{\ii 2 y \theta} (e^{\ii 2 \theta}-1) \\
&= \sum_{y=c_1+1}^{c_2+1} w_n(y-1) e^{\ii 2 y\theta} - \sum_{y=c_1}^{c_2} w_n(y) e^{\ii 2 y\theta} \\
& = w_n(c_2) e^{\ii 2 (c_2+1)\theta} - w_n(c_1) e^{\ii 2 c_1\theta} + \sum_{y=c_1+1}^{c_2} (w_n(y-1)-w_n(y)) e^{\ii 2 y\theta}.
\end{align*} 
Focusing on the second term,
\begin{align*}
& \sum_{y=c_1+1}^{c_2} (w_n(y-1)-w_n(y)) e^{\ii 2 y\theta} \\
&= \sum_{y=c_1+1}^{c_2} \left(K^2\left(\frac{y-1}{\Delta_n}\right) \frac{n-1-(y-1)}{n}-K^2\left(\frac{y}{\Delta_n}\right) \frac{n-1-y}{n}\right) 
e^{\ii 2 y\theta} \\
&= \frac{1}{n} \sum_{y=c_1+1}^{c_2} K^2\left(\frac{y-1}{\Delta_n}\right) e^{\ii 2 y\theta} 
 + \sum_{y=c_1+1}^{c_2} \left(K^2\left(\frac{y-1}{\Delta_n}\right)-K^2\left(\frac{y}{\Delta_n}\right)\right)  \frac{n-1-y}{n}e^{\ii 2 y\theta} \\
&=: E_1+E_2.
\end{align*} 
Clearly, $E_1={\cal O}(\Delta_n/n) = o(1)$ uniformly in $c_1,c_2$. Also, it follows that
\begin{align*}
|E_2| \le \sum_{y=c_1+1}^{c_2} \left|K^2\left(\frac{y-1}{\Delta_n}\right)-K^2\left(\frac{y}{\Delta_n}\right)\right| < \infty
\mbox{ uniformly in $c_1,c_2$}
\end{align*} 
since $K^2$ is of bounded variation (recall that $K'$ is bounded and $K$ is compactly supported).  
Thus,
\begin{align*}
& \sum_{y=c_1}^{c_2} w_n(y) e^{\ii 2 y \theta}  \\
& = (e^{\ii 2\theta}-1)^{-1}\left(w_n(c_2) e^{\ii (c_2+1)\theta} - w_n(c_1) e^{\ii c_1\theta}
+ \sum_{y=c_1+1}^{c_2} (w_n(y-1)-w_n(y)) e^{\ii y\theta}\right),
\end{align*} 
which is uniformly bounded. (Note that here we used the fact that $e^{\ii 2\theta} -1 \not =0$, since $\pm\pi\not =\theta \not = 0$.)
Applying this argument, 
we see that the inner sum in \eqref{e:D1-term} is uniformly bounded, and hence by Condition \eqref{e:Assumption-C-prime}, 
we obtain that $D_1 = o(1)$.

On the other hand, for the term $D_2$, we obtain
\begin{align}\label{e:D2-term}
\begin{split}
    D_2& := \frac{1}{\Delta_n}\sum_{x_1,x_2=-n+1}^{n-1}\check C(x_1) \overline{\check C(x_2)} e^{-\mathbbm{i}x_1\theta}e^{\mathbbm{i}x_2\theta}\\
     & \qquad \times\sum_{y_1=(1-n)\vee(1-n+x_1-x_2)}^{(n-1)\wedge(n-1+x_1-x_2)}K\left(\frac{y_1}{\Delta_n}\right)\left[K\left(\frac{-x_1+x_2+y_1}{\Delta_n}\right)-K\left(\frac{y_1}{\Delta_n}\right)\right] \\
    & \hspace{3.5cm} \times e^{\mathbbm{i}2y_1\theta}\cdot\frac{{n-|y_1|}}{n}\\ 
    & = o(1),
    \end{split}
\end{align}
using the same arguments as for $B_2$ and $C_2$. Thus, 
\begin{align*}
\overline{A}_2\to 0 \quad \mbox{for}\quad \theta\not=0\ \ \mbox{ and }\ \  \theta\not=\pm\pi.
\end{align*}
This completes the derivation for $\E(\CT_n^2)$.
\vskip.3cm\noindent
\underline{The case $k\ge 3$}
\vskip.3cm
\noindent
Fix some integer  $k\ge 3$. Let $\mathcal{P}$ be the set of all possible pairings of $\{i_{\ell},j_{\ell}:\ \ell=1,\hdots,k\}$. Then a pairing $\pa\in\mathcal{P}$ iff \[\pa =\left\{\{I,J\},\{I,\tilde I\},\{J,\tilde J\}:I\neq\tilde I\in\{i_{\ell},\ell=1,\hdots,k\},J\neq\tilde J\in\{j_{\ell},\ell=1,\hdots,k\}\right\},\] where all symbols $i_{\ell},j_{\ell},\ell=1,\hdots,k$ can be used only once. 

By Lemma \ref{le:extra Isserlis}, 
\begin{align*}
    \mu_k & := (2\pi)^k\E(\CT_n^k)= \frac{1}{(n\Delta_n)^{k/2}} \sum_{\substack{i_{\ell},j_{\ell}=1\\ \ell=1,\hdots,k}}^n\mathbb{E}\left[\prod_{\ell=1}^k\left[X_{i_{\ell}}\overline{X_{j_{\ell}}}-C(i_{\ell}-j_{\ell})\right]\right]\prod_{\ell=1}^kK_{\theta}(i_{\ell}-j_{\ell})\\
    & = \frac{1}{(n\Delta_n)^{k/2}} \sum_{\substack{i_{\ell},j_{\ell}=1\\ \ell=1,\hdots,k}}^n\prod_{\ell=1}^kK_{\theta}(i_{\ell}-j_{\ell})\\ & \qquad\times\sum_{\substack{\pa \in \mathcal{P}:\\ \cup_{p=1}^k\{\{i_p,j_p\}\}\cap\pa=\emptyset}} \prod_{\{i,j\}\in\pa} \Big[C(I-J)\mathbbm{1}_{\{i,j\}=\{I,J\}}+\check C(I-\tilde I)\mathbbm{1}_{\{i,j\}=\{I,\tilde I\}} \\
    & \hskip3cm\qquad\qquad +\overline{\check C(J-\tilde J)}\mathbbm{1}_{\{i,j\}=\{J,\tilde J\}}\Big],
\end{align*}

Let now $r\le k$ and fix a subset of $2r$ indices $\{i_1',j_1',\hdots, i_r',j_r'\} \subset\{i_1,j_1,\cdots,i_k,j_k\}$, where
$(i_1' ,j_1')= (i_{\sigma_1},j_{\sigma_1}),\cdots, (i_r' ,j_r')= (i_{\sigma_r},j_{\sigma_r})$, for some $1\le \sigma_1<\cdots<\sigma_r\le k$. 
A partition of 
$\{i_1',j_1',\hdots i_r',j_r'\}$ into pairs 
will be called a sub-pairing of order $r$.  Namely, it is a partition into pairs that involves $r$ couples of $i'$ and/or $j'$ symbols
taken only from the set $\{i_1',j_1',\hdots i_r',j_r'\}$.
 
A (sub)pairing will be called {\em irreducible}, if does not have further sub-pairings, i.e., it cannot be broken up into a disjoint union of two or more sub-pairings
of lower order.   Let $C_{\mathcal{P},r,k}$ denote the set of all {\em irreducible} sub-pairings of order $r$.

  Looking at a single summand of the second sum in $\mu_k,$ one can see that every pairing 
 $\pa \in \mathcal{P}$ is the union of multiple irreducible pairings of the form $C_{\mathcal{P},r,k}$ with $r\ge 2$. We will argue below 
 that among all pairings in ${\cal P}$ only the ones involving irreducible components of order $r=2$ contribute asymptotically, 
 and the remaining pairings are of lower order, as $n\to\infty$.
 
 Let $\pa\in \mathcal{P}$ denote a pairing that shows up in the second sum of $(2\pi)^d\E(\CT_n^k)$, and suppose that 
$$
\pa  = \pa_{r_1} \cup \cdots \cup\pa_{r_m},
$$ 
where the $\pa_{r_i} \in C_{\mathcal{P},r_i,k},\ r_i\ge 2,\ i=1,\hdots,m,$ are the irreducible sub-pairings of $\pa$.
Then, 
\begin{align} \label{e:r_ge_3_moment_bound}
 \mu_k= (2\pi)^k\E(\CT_n^k) &=  \frac{1}{(n\Delta_n)^{k/2}}  \sum_{\substack{i_{\ell},j_{\ell}=1\\ \ell=1,\hdots,k}}^n \Big\{  \prod_{\ell=1}^kK_{\theta}(i_{\ell}-j_{\ell}) \\ 
 &\times  \sum_{\substack{\pa \in {\cal P}:\\ \cup_{p=1}^k\left\{\{i_p,j_p\}\right\}\cap\pa=\emptyset }} \prod_{\{i,j\}\in \pa} \Big[C(I-J)\mathbbm{1}_{\{i,j\}=\{I,J\}}+\check C(I-\tilde I)\mathbbm{1}_{\{i,j\}=\{I,\tilde I\}}\nonumber\\
  & \hspace{3cm} +\overline{\check C(J-\tilde J)}\mathbbm{1}_{\{i,j\}=\{J,\tilde J\}}\Big] \Big\} \nonumber\\ 
  & =:  \left(\frac{n}{\Delta_n}\right)^{k/2}  \sum_{\substack{\pa \in {\cal P} \\ \pa = \pa_{r_1} \cup \cdots \cup\pa_{r_m},\ m\ge 1 }}  \prod_{t=1}^m A_{\mathcal{P},\pa_{r_t},k},\nonumber
\end{align}
where $A_{\mathcal{P},\pa_{r_t},k}$ involves a product of the terms restricted to the irreducible sub-pairing $\pa_{r_t}$, and where
$r_1+ \cdots + r_m = k$, with $r_i\ge 2,\ i=1,\hdots,m$.  Namely,
assuming that the subset of indices $\{i_1',j_1',\cdots,i_r',j_r'\} = \{i_{\sigma_2},j_{\sigma_1},\cdots,i_{\sigma_r},j_{\sigma_r}\},\ r\le k$, is involved in the irreducible pairing  $A_{\mathcal{P},\pa_r,k}$
we have
\begin{align*}
    A_{\mathcal{P},\pa_r,k} & = \frac{1}{n^r}  \sum_{\substack{i_{\ell}',j_{\ell}'=1\\ \ell=1,\hdots,r}}^n
     \prod_{\ell=1}^rK_{\theta}(i_{\ell}'-j_{\ell}')\times
   \\ 
   & \prod_{\{i,j\}\in\pa_r} \Big[C(I-J)\mathbbm{1}_{\{i,j\}=\{I,J\}}+\check C(I-\tilde I)\mathbbm{1}_{\{i,j\}=\{I,\tilde I\}}\\ 
   & \qquad\qquad\qquad+\overline{\check C(J-\tilde J)}\mathbbm{1}_{\{i,j\}=\{J,\tilde J\}}\Big].
\end{align*}
Let $r\ge3,r\le k$, and apply the change of variables
\begin{align*}
    & x_{\ell} = i-j,\ {i,j}\in\pa_r \\
    & y_{\ell} = i_{\ell}'-j_{\ell}',\ell=1,\hdots,r-1, \\ 
    & y_r = i_r',
\end{align*}
where the order of $i,j$ for $x_{\ell}$ is determined by the order they appear in the $C,\check C$ and $\overline{\check C}$ terms. Note that since the kernel $K$ is non-negative and bounded,
$$
\Big| \prod_{\ell=1}^rK_{\theta}(i_{\ell}'-j_{\ell}')\Big| \le \|K\|_\infty \prod_{\ell = 1}^{r-1} K\Big(\frac{y_\ell}{\Delta_n}\Big).
$$
Then, letting $D(x):= |C(x)| \vee |\check C(x)|$, we obtain
\begin{align} \label{e:Aprk}
\begin{split}
    & |A_{\mathcal{P},\pa_r,k}| \\
    & \le \frac{\|K\|_\infty}{n^r}\sum_{\substack{x_{\ell}=-n+1\\\ell=1,\hdots,r}}^{n-1}\prod_{\ell=1}^r D(x_{\ell}) \cdot \sum_{\substack{y_{m}=-n+1\\ m=1,\hdots,r-1}}^{n-1}\prod_{m=1}^{r-1}K\Big(\frac{y_{m}}{\Delta_n}\Big) \sum_{y_r=1}^n 1\\ & \le
    \frac{\|K\|_\infty}{n^{r-1}}\left(\sum_{x\in \mathbb Z}D(x)\right)^r \cdot \sum_{\substack{y_{m}=-n+1\\ m=1,\hdots,r-1}}^{n-1}\prod_{m=1}^{r-1}K\left(\frac{y_{m}}{\Delta_n}\right) \\ & = \mathcal{O}\left(\left(\frac{\Delta_n}{n}\right)^{r-1}\right).
\end{split}
\end{align}
where  we used that by Relation \eqref{e:Assumption-C-prime}, $\sum_{x} D(x) <\infty$ and the compactness of the support of $K$.

Using \eqref{e:Aprk}, in view of \eqref{e:r_ge_3_moment_bound}, one immediately has that
\begin{align*}
  \E (\CT_n^k) =  \mathcal{O}\left( \Big( \frac{n}{\Delta_n} \Big)^{k/2}  \cdot \max_{\substack{ m = 1,\cdots, \lfloor k/2\rfloor \\ 
  r_1+\cdots+r_m =k,\ r_t\ge 2}} \Big(\frac{\Delta_n}{n}\Big)^{\sum_{t=1}^m(r_t-1)}\right)\equiv \mathcal{O}\left(\left(\frac{\Delta_n}{n}\right)^{k/2-M}\right),
\end{align*}
where 
$$
  M:= \mathop{\max}_{m=1,\cdots,\lfloor k/2\rfloor} \Big\{m\, :\, r_1+\cdots+r_m=k,\ \mbox{ where }r_t\in \{2,\cdots,k\} \Big\}.
 $$

Clearly, if $k$ is odd, then $M=(k-1)/2$, we have $k/2-M = 1/2$ and by the above bound, we obtain
$$
\E [ \CT_n^k]= {\cal O}( (\Delta_n/n)^{1/2}),
$$ 
completing the proof of Proposition \ref{prop: central moments} in this case.  Note that this moment vanishes as $n\to \infty$.  

If $k$ is even, then $M=k/2$ and $k/2-M =0$.  By the above argument, the only pairings that do not vanish asymptotically, as $n\to\infty$,
correspond to $r_1=\cdots=r_{k/2} = 2$.  That is, the indices $\{i_1,j_1,\cdots,i_k,j_k\}$ are paired into $k/2$ irreducible sub-pairings of 
order $2$ and this case algebraically reduces to the case $k=2$.  

Consider four indices $\{i_1,j_1,i_2,j_2\}$ and let
$A_{{\cal P},k}^{(\{i_1,j_2\},\{i_2,j_1\})}$ and $A_{{\cal P},k}^{(\{i_1,i_2\},\{j_1,j_2\})}$ be the terms of \eqref{e:r_ge_3_moment_bound} corresponding to the subpairings $\{\{i_1,j_2\},\{i_2,j_1\}\}$ and $\{\{i_1,i_2\},\{j_1,j_2\}\}$ respectively. Let also
\[A_{{\cal P},k}^{(\{i_1,j_1,i_2,j_2\})}:=A_{{\cal P},k}^{(\{i_1,j_2\},\{i_2,j_1\})} +A_{{\cal P},k}^{(\{i_1,i_2\},\{j_1,j_2\})}
\]
By the first part of the proof, the sum of these two order-2 irreducible subpairings that correspond to the same indices $\{i_1,j_1,i_2,j_2\}$ 
contributes the following term to the rate of the expectation:
$$
\Big( \frac{n}{\Delta_n} \Big)^{1} A_{{\cal P},k}^{(\{i_1,j_1,i_2,j_2\})} \to \sigma_f^2 (\theta):=\Big (f(\theta)^2 +  1_{\{0,\pm\pi\}}(\theta) |\check  f(\theta)|^2 \Big )\|K\|_2^2,
$$
as $n\to\infty$.  Therefore, in view of \eqref{e:r_ge_3_moment_bound},
\begin{align*} 
 \mu_k= (2\pi)^k\E(\CT_n^k) &=   \left(\frac{n}{\Delta_n}\right)^{k/2}  \sum_{\substack{\pa \in {\cal P} \\ \pa = \pa_{r_1} \cup \cdots \cup\pa_{r_m},\ m\ge 1 }}  \prod_{t=1}^m A_{\mathcal{P},\pa_{r_t},k}\\ &\asymp \left(\frac{n}{\Delta_n}\right)^{k/2}  \sum_{\substack{\pa \in {\cal P} \\ \pa = \pa_{r_1} \cup \cdots \cup\pa_{r_{k/2}}\\ r_1=\hdots=r_{k/2}=2}}  \prod_{t=1}^m A_{\mathcal{P},\pa_{r_t},k}\\ & = \left(\frac{n}{\Delta_n}\right)^{k/2}  \sum_{\substack{\mathbbm{q} \in {\cal Q}_2 \\ \mathbbm{q} = \{\{i_1',j_1',i_2',j_2'\}, \cdots, \\ \{i_{k-1}',j_{k-1}',i_k',j_k'\}\}}}  \prod_{\substack{\ell,m=1,\hdots,k\\\ell\neq m}} \left(A_{{\cal P},k}^{(\{i_{\ell}',j_m'\},\{i_m',j_{\ell}'\})} +A_{{\cal P},k}^{(\{i_{\ell}',i_m'\},\{j_{\ell}',j_m'\})}\right)\\ & = \left(\frac{n}{\Delta_n}\right)^{k/2}  \sum_{\substack{\mathbbm{q} \in {\cal Q}_2 \\ \mathbbm{q} = \{\{i_1',j_1',i_2',j_2'\}, \cdots, \\ \{i_{k-1}',j_{k-1}',i_k',j_k'\}\}}} \prod_{\substack{\ell,m=1,\hdots,k\\\ell\neq m}} A_{{\cal P},k}^{(\{i_{\ell}',j_{\ell}',i_m',j_m'\})} \\ & \to N_2(k) \times \Big[ \sigma_f^2 (\theta) \Big]^{k/2},
\end{align*}
where $N_2(k)$ denotes the number of ways one can partition the set $\{i_1,j_1,\cdots,i_k,j_k\}$ into $k/2$ sets of 4 members including both $i$ and $j$ of the same index, and ${\cal Q}_2$ denotes the collection of all those partitions.

To complete the proof of Proposition \ref{prop: central moments}, it remains to argue that $N_2(k) = |{\cal Q}_2|=(k-1)!!.$ Note that every $q\in {\cal Q}_2$ is determined by a partition into sets of $4$ indices $\{i_{\ell_1},j_{\ell_1}, i_{\ell_2},j_{\ell_2}\}$ from the $2k$ symbols $\{i_1,j_1,\cdots,i_k,j_k\}$.  Thus, determining 
    the number $N_2(k)$ is equivalent to counting the number of partitions of the set $\{i_1,\cdots,i_k\}$ into 
    $2-$point subsets $\{i_{\ell_1},i_{\ell_2}\}$.  The number of ways to pick the first pair is $\binom{k}{2}$, the second pair $\binom{k-2}{2}$, and so on.  Therefore
     $N_2(k)$ equals
    $$
    \frac{1}{(k/2)!}  \binom{k}{2}\cdot \binom{k-2}{2} \cdots \binom{2}{2} = (k-1)!!,
    $$
    where we divide by $(k/2)!$ since the order of the subsets $\{i_{\ell_1},i_{\ell_2}\}$ does not matter.
\end{proof} 

\subsubsection{The general case} \label{s:moment_general} 

The purpose of this section is to finish the verification of (ii) of Proposition \ref{e:weak_conv} for a general $\bbH$ under the assumptions of the Proposition. 
Recall that we already verified (ii) for the spatial setting $\bbH=\bbC$ in the previous subsection.
The extension from the scalar to the general case is actually quite straightforward. We illustrate this for the second moment.

Recall that $X_{g_\ell}(i) = \langle X(i), g_\ell\rangle$.
Denote $(2\pi)^2 \mathbb{E}\left[\sum_{\ell=1}^m a_\ell\left\langle\CT_ng_{\ell},g_{\ell}\right\rangle\right]^2$ by $A_{n,2}$.
By Isserlis' formula in Lemma \ref{le:complex isserlis},
\begin{align*}
    A_{n,2} & = \sum_{\ell_1,\ell_2=1}^m a_{\ell_1}a_{\ell_2} \frac{1}{n\Delta_n}\sum_{i_1,j_1=1}^n\sum_{i_2,j_2=1}^n K_{\theta}(i_1-j_1)K_{\theta}(i_2-j_2) \\ 
    & \hspace{2cm} \times  \E \Big\{\left[X_{g_{\ell_1}}(i_1)\overline{X_{g_{\ell_1}}(j_1)}-\mathbb{E}X_{g_{\ell_1}}(i_1)\overline{X_{g_{\ell_1}}(j_1)}\right] \\
& \hspace{4cm}   \times \left[X_{g_{\ell_2}}(i_2)\overline{X_{g_{\ell_2}}(j_2)}-\mathbb{E}X_{g_{\ell_2}}(i_2)\overline{X_{g_{\ell_2}}(j_2)}\right]\Big\}\\ 
    & =  \sum_{\ell_1,\ell_2=1}^m a_{\ell_1}a_{\ell_2} \frac{1}{n\Delta_n}\sum_{i_1,j_1=1}^n\sum_{i_2,j_2=1}^n K_{\theta}(i_1-j_1)K_{\theta}(i_2-j_2)\\ 
    & \hspace{2cm} \times\Big[\mathbb{E}X_{g_{\ell_1}}(i_1)\overline{X_{g_{\ell_1}}(j_1)}X_{g_{\ell_2}}(i_2)\overline{X_{g_{\ell_2}}(j_2)} \\
    & \hspace{4cm} -\mathbb{E}X_{g_{\ell_1}}(i_1)\overline{X_{g_{\ell_1}}(j_1)}\mathbb{E}X_{g_{\ell_2}}(i_2)\overline{X_{g_{\ell_2}}(j_2)}\Big]\\ 
    & =    \sum_{\ell_1,\ell_2=1}^m a_{\ell_1}a_{\ell_2} \frac{1}{n\Delta_n}\sum_{i_1,j_1=1}^n\sum_{i_2,j_2=1}^n K_{\theta}(i_1-j_1)K_{\theta}(i_2-j_2) \\
   & \hspace{4cm}  \times  \mathbb{E}X_{g_{\ell_1}}(i_1)\overline{X_{g_{\ell_2}}(j_2)}\mathbb{E}X_{g_{\ell_2}}(i_2)\overline{X_{g_{\ell_1}}(j_1)}\\ 
   & + \sum_{\ell_1,\ell_2=1}^m a_{\ell_1}a_{\ell_2}  \frac{1}{n\Delta_n}\sum_{i_1,j_1=1}^n\sum_{i_2,j_2=1}^n K_{\theta}(i_1-j_1)K_{\theta}(i_2-j_2) \\
    & \hspace{4cm}  \times \mathbb{E}X_{g_{\ell_1}}(i_1)X_{g_{\ell_2}}(i_2)\mathbb{E}\overline{X_{g_{\ell_1}}(j_1)}\overline{X_{g_{\ell_2}}(j_2)}\\ 
    & =: \widetilde A_{n,2} + \overline{A}_{n,2}.
\end{align*}
Define 
\begin{align} \label{e:Cl1l2}
C_{g_{\ell_1},g_{\ell_2}}(t):=\mathbb{E}X_{g_{\ell_1}}(t)\overline{X_{g_{\ell_2}}(0)}.
\end{align}
Start with $\theta\not\in\{0,\pm\pi\}$. By the same arguments as in Proposition \ref{prop: central moments}, one can focus only on $\widetilde A_2.$ By the change of variables 
\begin{align*}
    x_1 = i_1-j_2, &\ \  x_2=i_2-j_1, \\
    y_1=i_1-j_1, &\ \ y_2=i_1, 
\end{align*}
we have that 
\begin{align*}
    \widetilde A_{n,2} & = \sum_{\ell_1,\ell_2=1}^m a_{\ell_1}a_{\ell_2}\frac{1}{n\Delta_n}\sum_{x_1,x_2=-n+1}^{n-1} C_{g_{\ell_1},g_{\ell_2}}(x_1)\overline{C_{g_{\ell_1},g_{\ell_2}}(-x_2)}\\
     &\qquad\times\sum_{y_1=(1-n)\vee (1-n+x_1+x_2)}^{(n-1)\wedge(n-1+x_1+x_2)}K_{\theta}(y_1)K_{\theta}(x_1+x_2-y_1)\sum_{y_2=1\vee (1+y_1)}^{n\wedge (n+y_1)}1 \\ 
    & = \sum_{\ell_1,\ell_2=1}^m a_{\ell_1}a_{\ell_2}
    \sum_{x_1,x_2=-n+1}^{n-1} C_{g_{\ell_1},g_{\ell_2}}(x_1)e^{\mathbbm{i}x_1\theta}{\overline{ C_{g_{\ell_1},g_{\ell_2}}(-x_2)e^{-\mathbbm{i}x_2\theta}}}\\ 
    &\qquad\times\frac{1}{\Delta_n} \sum_{y_1=(1-n)\vee (1-n+x_1+x_2)}^{(n-1)\wedge(n-1+x_1+x_2)}K(y_1)K(x_1+x_2-y_1)\frac{{n-|y_1|}}{n}
    \\ & \sim \sum_{\ell_1,\ell_2=1}^m a_{\ell_1}a_{\ell_2} 4\pi^2 \int_{y=-1}^1K^2(y)dy \Big|\sum_{x=-\infty}^{\infty} C_{g_{\ell_1},g_{\ell_2}}(x)e^{\mathbbm{i}x\theta}\Big|^2
    \end{align*}
based on the proof of Proposition \ref{prop: central moments}. By \eqref{e:Cl1l2}, this is precisely
\begin{align*}
4\pi^2 \int_{y=-1}^1K^2(y)dy \sum_{\ell_1,\ell_2=1}^ma_{\ell_1}a_{\ell_2}  |\left\langle f(\theta)g_{\ell_1},g_{\ell_2}\right\rangle|^2.
\end{align*}

The derivations for $\overline{A}_{n,2}$ are similar. 
For instance, consider $\theta=0$:
\begin{align*}
    \overline{A}_{n,2} & = \frac{1}{n\Delta_n }\sum_{i_1,j_1=1}^n\sum_{i_2,j_2=1}^n K_{0}(i_1-j_1)K_{0}(i_2-j_2) \\
    & \qquad \times\sum_{\ell_1,\ell_2=1}^ma_{\ell_1}a_{\ell_2}
    \check C_{g_{\ell_1},g_{\ell_2}}(i_1-i_2)\overline{\check C_{g_{\ell_1},g_{\ell_2}}(j_1-j_2)},
\end{align*}
where
\begin{align*} 
\check C_{g_{\ell_1},g_{\ell_2}}(t) := \mathbb{E} [X_{g_{\ell_1}}(t) X_{g_{\ell_2}}(0)] = \left\langle\mathbb{E}\left[X(t)\otimes \overline{X(0)}\right]\overline{g_{\ell_2}},g_{\ell_1}\right\rangle
= \langle \check C(t)\overline{g_{\ell_2}},g_{\ell_1}\rangle.
\end{align*}
Making the change of variables 
\begin{align*}
    x_1 = i_1-i_2, &\ \  x_2=j_1-j_2, \\
    y_1=i_1-j_1, &\ \ y_2=i_1,
\end{align*}
we have that
\begin{align*}
    \overline{A}_{n,2} & = \sum_{\ell_1,\ell_2=1}^m a_{\ell_1}a_{\ell_2} \frac{1}{\Delta_n} \sum_{x_1,x_2=1-n}^{n-1}\check C_{g_{\ell_1},g_{\ell_2}}(x_1)\overline{\check C_{g_{\ell_1},g_{\ell_2}}(x_2)}\\ &\qquad\times\sum_{y_1=(1-n)\vee(1-n+x_1-x_2)}^{(n-1)\wedge(n-1)+x_1-x_2}K\left(\frac{y_1}{\Delta_n}\right)K\left(\frac{-x_1+x_2+y_1}{\Delta_n}\right)\frac{n-|y_1|}{n}\\ 
    & \sim 4\pi^2\int_{y=-1}^1K^2(y)dy \sum_{\ell_1,\ell_2=1}^m a_{\ell_1}a_{\ell_2} \left|\left\langle \check f(0)\overline{g_{\ell_2}},g_{\ell_1}\right\rangle\right|^2,
\end{align*}
using again the arguments of the proof of Proposition \ref{prop: central moments}.
Thus, we have verified (ii) of Proposition \ref{e:weak_conv} for $k=2$.


In a similar manner, the derivation of $\mathbb{E}\left[\sum_{\ell=1}^m a_\ell \left\langle\CT_ng_{\ell},g_{\ell}\right\rangle\right]^k$ for $k\ge 3$ for a general space $\bbH$ can be extended from that for the scalar case, and the details are omitted.

\section{Cumulants and Isserlis' formulas}\label{appendix:cumulants and Isserlis' theorem}
\subsection{Cumulants for functional data}
This section is of independent interest and provides an extension of Isserlis' theorem to the regime of Hilbert space valued Gaussiaan random variables. We start by providing the definition of the cumulants for scalar random variables taking values in $\mathbb{R}$. 
\begin{definition}\label{def:cumulants}
Let $Y_1,\ldots, Y_k$ be random variables taking values in $\mathbb{R}$ such that $\E(\prod_{j\in B} Y_j)$ is well defined and finite for all subsets $B$ of $\{1,\ldots,k\}$. Then,  $${\rm cum}\left(Y_1,\hdots,Y_k\right):=\sum_{\nu=(\nu_1,\hdots,\nu_q)}(-1)^{q-1}(q-1)!\prod_{l=1}^q\mathbb{E}\left[\prod_{j\in \nu_l}Y_j \right],$$
where the sum is over all unordered partitions of $\{1,\hdots,k\}.$
\end{definition}

The following lemma follows from the discussion on page 34 of \cite{rosenblatt1985stationary}.
\begin{lemma}\label{supp-le:rosenblatt}
Let $Y_i, i=1,\hdots,k$ be real random variables such that $\E(\prod_{j\in B} Y_j)$ is well defined and finite for all subsets $B$ of $\{1,\ldots,k\}$. Then 
\begin{align*}
    \mathbb{E}[Y_1\cdot\hdots\cdot Y_k] = \sum_{\nu=(\nu_1,\hdots,\nu_p)}\prod_{l=1}^p{\rm  cum}(Y_i;\ i\in\nu_l),
\end{align*}
where the sum is over all the unordered partitions of $\{1,\hdots,k\}.$
\end{lemma}

\begin{prop}\label{prop:Cov for cross product}
Let $\{X(t)\}$ be a stochastic process taking values in a Hilbert space $\bbH$, where $\E(\|X(t)\|^4) < \infty$ for all $t$.  Note that we
 do not assume here $X$ to be real.  Fix an arbitrary real CONS $\{e_i,\ i\in I\}$ of $\bbH$ and denote by 
 $X_i(t):=\langle X(t), e_i\rangle$. Then for any $t,s,w,v\in\mathbb{R}^d$, we have that 
\begin{align*}
    {\rm cum}\left(X(t), X(s), X(w), X(v)\right)  = \sum_i\sum_j {\rm cum}(X_j(t),\overline{X_i(s)},\ol{X_j(w)},X_i(v)).
    \end{align*}
\end{prop}
\begin{proof}
Recall the definition of cumulant in \eqref{def:cum}:
\begin{align*} 
\begin{split}
& {\rm cum}\left(X(t),X(s),X(w),X(v)\right) \\
&= \mathbb{E}\left\langle X(t)\otimes X(s),X(w)\otimes X(v) \right\rangle_{\rm HS} 
- \langle \E (X(t)\otimes X(s)), \E (X(w)\otimes X(v)) \rangle_{\rm HS} \\
& \hskip2cm- \mathbb{E} \left\langle X(t), X(w)\right\rangle_{\mathbb{H}}\cdot \mathbb{E}\left\langle  X(v), X(s)\right\rangle_{\mathbb{H}} \\
& \hskip3cm -  \left\langle \E (X(t)\otimes \overline{X(v)}), \E (X(w)\otimes \overline{X(s)})\right\rangle_{\rm HS}.
\end{split}
\end{align*}
For any $x(1),\ldots,x(4)\in\bbH$,
\begin{align*} 
\begin{split}
\langle x(1)\otimes x(2), x(3)\otimes x(4) \rangle_{\rm HS} 
& = \sum_i \langle (x(1)\otimes x(2)) e_i, (x(3)\otimes x(4)) e_i \rangle_\bbH \\
& = \left\langle x(1), x(3)\right\rangle_\bbH \overline{\left\langle x(2), x(4)\right\rangle_\bbH} \\
& = \sum_i\sum_j x_{i}(1)\overline{x_{i}(3) x_{j}(2)}x_{j}(4).
\end{split}
\end{align*}
It follows that
\begin{align*}
\left\langle X(t)\otimes X(s),X(w)\otimes X(v) \right\rangle_{\rm HS} =  \sum_i\sum_j \overline{X_i(s)}X_i(v)X_j(t)\overline{X_j(w)}. 
\end{align*}
It suffices to show that
\begin{align*}
\E \left\langle X(t)\otimes X(s),X(w)\otimes X(v) \right\rangle_{\rm HS} 
= \sum_i\sum_j \E (\overline{X_i(s)}X_i(v)X_j(t)\overline{X_j(w)} ),
\end{align*}
where the interchange of the order of summation and expectation can be justified by the fourth-moment assumption on the $X(t)$ and
Fubini's Theorem.

Similarly, we have that 
\[\langle \E (X(t)\otimes X(s)), \E (X(w)\otimes X(v)) \rangle_{\rm HS} = \sum_i\sum_j \E (\overline{X_i(s)}X_j(t))\mathbb{E}(X_i(v)\overline{X_j(w)} )\]
and
\[\langle \E (X(t)\otimes \overline{X(v)}), \E (X(w)\otimes \overline{X(s)}) \rangle_{\rm HS} = \sum_i\sum_j \E (X_j(t)X_i(v))\mathbb{E}(\overline{X_i(s)}\overline{X_j(w)} ),\]
where we used the fact that the CONS $\{e_j\}$ is real in order to write 
$\overline{X(s)} = \sum_i \overline{X_i(s)} e_j.$
Finally, 
\[\mathbb{E} \left\langle X(t), X(w)\right\rangle_{\mathbb{H}}\cdot \mathbb{E}\left\langle  X(v), X(s)\right\rangle_{\mathbb{H}} = \sum_i\sum_j\mathbb{E}X_j(t)\overline{X_j(w)} \mathbb{E}X_i(v)\overline{X_i(s)}.\]

Gathering all four terms one can easily see that the cumulant sum
\[\sum_{i}\sum_{j} {\rm cum}(X_j(t),\overline{X_i(s)}, \ol{X_j(w)},X_i(v))\]
is reconstructed.
\end{proof}

We end this subsection with a remark on the connection with a related but different notion of 
cumulant employed in \cite{panaretos2013fourier}.

\begin{remark}\label{s:rem:Panaretos-cumulants}

\cite{panaretos2013fourier} defines a notion of cumulant on the bottom of 
page 571 of the paper. In this remark, we will attempt to explain the connection between the condition $C(0,4)$
in \cite{panaretos2013fourier} with (c) of Assumption \ref{a:var_d}. 

For simplicity, we shall work with real Hilbert spaces.  Recall that in \cite{panaretos2013fourier}, the authors 
consider $\H = L^2[0,1]$ and define the so-called {\em cumulant kernel}:
 $$
 {\rm cum}_{\rm ker}(X(t_1),\cdots, X(t_k)):= \sum_{\nu=(\nu_1,\dots,\nu_p)} (-1)^{p-1} (p-1)! \prod_{\ell=1}^p 
 \E \Big[ \prod_{j\in \nu_\ell} X(\tau_j;t_j)\Big],
 $$
 where $X(t):= (X(\tau;t),\ \tau\in [0,1])\in L^2([0,1])$.  For a kernel of order $2k$, one can define the so-called cumulant 
 operator ${\cal R} : L^2 ([0,1]^k)
 \to L^2([0,1]^k)$, as 
 $$
{\cal R}(h) := 
 \int_{[0,1]^2} {\rm cum}_{\rm ker}( X(t_1),\cdots, X(t_{2k})) (\tau_1,\cdots,\tau_{2k}) h(\tau_{k+1},\cdots,\tau_{2k}) d\tau_{k+1}\cdots d\tau_{2k},
 $$
 where the latter is understood as a function of $(\tau_1,\cdots,\tau_k)$ that can be shown to belong to $L^2([0,1]^k)$.
 
 Fixing a CONS $\{e_j\}$ of $L^2([0,1])$, for $k=2$, we obtain that 
 $$
 {\rm cum}_{\rm ker}( X(t_1),\cdots, X(t_{4})) = \sum_{i,j,k,\ell} {\rm cum}(X_i(t_1),X_{j}(t_2),X_k(t_3),X_\ell(t_4)) 
 e_i\otimes e_j\otimes  e_k \otimes e_\ell,
 $$
 where ${\rm cum}$ stands for the usual cumulant of random variables, and where $X_i(t) = \langle X(t),e_i\rangle$ are the coordinates
 of $X(t)$ in the basis $\{e_j\}$. 
 Thus, in the basis $\{e_i\otimes e_j\}$ of $L^2([0,1]^2]\equiv L^2([0,1])\otimes L^2([0,1])$,  
 one can view the cumulant operator ${\cal R}: L^2([0,1])\otimes L^2([0,1]) \to L^2([0,1])\otimes L^2([0,1])$ as
 $$
 {\cal R} = \sum_{i,j,k,\ell} r_{(i,j), (k,\ell)} (e_i\otimes e_j) \otimes (e_k\otimes e_\ell),
 $$
 where $ r_{(i,j), (k,\ell)}:= {\rm cum}(X_i(t_1),X_{j}(t_2),X_k(t_3),X_\ell(t_4))$.
 
 From this perspective, by \eqref{e:cum_basis}, we obtain that our notion of a cumulant coincides with the trace of the Hilbert-Schmidt 
 cumulant operator ${\cal R}$:
 $$
 {\rm trace}({\cal R}) = {\rm cum}(X(t_1),X(t_2),X(t_3),X(t_4))= \sum_{i,j} r_{(i,j),(i,j)}.
 $$
 On the other hand, the norm of the cumulant kernel employed in the C(0,4) condition of \cite{panaretos2013fourier} becomes:
 $$
 \|{\rm cum}_{\rm ker}( X(t_1),\cdots, X(t_{4}))\|_{L^2}^2 = \sum_{i,j,k,\ell} {\rm cum}(X_i(t_1),X_{j}(t_2),X_k(t_3),X_\ell(t_4))^2.
 $$
 Whereas, recall that
 $$
 {\rm cum}( X(t_1),\cdots, X(t_{4})) = \sum_{i,j,k,\ell} {\rm cum}(X_i(t_1),X_{j}(t_2),X_k(t_3),X_\ell(t_4)).
 $$
 Thus, the condition $C(0,4)$ of \cite{panaretos2013fourier} that 
 $$
 \sum_{t_1,t_2,t_3}  \|{\rm cum}_{\rm ker}( X(t_1),X(t_2),X(t_3), X(0))\|_{L^2} <\infty
 $$
 is neither strictly weaker nor stronger than our condition (c) in Assumption \ref{a:var_d}.
 \end{remark}

\subsection{Isserlis' formulas}

The following is an extension of the classical Isserlis' formula to univariate complex Gaussian variables. 

\begin{lemma}\label{le:complex isserlis}
Let $Z_j = X_j + \ii Y_j,\ j=1,2,\cdots$ be zero-mean, complex jointly Gaussian random variables.  That is, $X_j, Y_j,\ j=1,2,\cdots$ are zero-mean jointly Gaussian $\mathbb R$-valued
random variables. Then, for all $m\in\mathbb N$, we have $\E [ \prod_{i=1}^{2m-1} Z_j] =0$, and 
\[\mathbb{E}\left(\prod_{j=1}^{2m}Z_j\right)=\sum_{\pi}\prod_{i=1}^m\mathbb{E}(Z_{a_{\pi,i}}Z_{b_{\pi,i}}),\]
where a pairing $\pi$ refers to a decomposition of $\{1,\ldots,2m\}$ into $m$ pairs, which are denoted as $(a_{\pi,i},b_{\pi,i}), i =1,\ldots,m$.
\end{lemma}
\begin{proof}
Recall that $Z_j=X_j+\mathbbm{i}Y_j$, where $X_i,Y_i$ are real. 
Let $\sigma_{a,b}^{(0,0)} = \E(X_aX_b), \sigma_{a,b}^{(0,1)} = \E(X_aY_b), \sigma_{a,b}^{(1,0)} = \E(Y_aX_b), \sigma_{a,b}^{(1,1)} = \E(Y_aY_b)$.
Write
$$
\E\left(\prod_{j=1}^{2m} (X_j+\ii Y_j)\right) = \sum_{S\subset \{1,\ldots,2m\}} \ii^{|S|} \E\left(\prod_{j\not\in S} X_j \prod_{k\in S} Y_k\right).
$$
By the Isserlis formula for real Gaussian random variables introduced by \cite{isserlis1918formula}, we have
\begin{align*} 
\E\left(\prod_{j\not\in S} X_j \prod_{k\in S} Y_k\right) = \sum_\pi \prod_{i=1}^m  \sigma_{a_{\pi,i},b_{\pi,i}}^{(\mathbbm{1}(a_{\pi,i}\in S),\mathbbm{1}(b_{\pi,i}\in S))},
\end{align*}
and hence
\begin{align*} 
\E\left(\prod_{j=1}^{2m} (X_j+\ii Y_j)\right) = \sum_\pi  \sum_{S\subset \{1,\ldots,2m\}} \ii^{|S|} \prod_{i=1}^m  \sigma_{a_{\pi,i},b_{\pi,i}}^{(\mathbbm{1}(a_{\pi,i}\in S),\mathbbm{1}(b_{\pi,i}\in S))}.
\end{align*}
For any given $\pi$ and $S$, we let $\alpha_i = \mathbbm{1}(a_{\pi,i}\in S), \beta_i=\mathbbm{1}{(b_{\pi,i}\in S)}$. Therefore, 
\begin{align} \label{e:C-Isserlis}
\begin{split}
\E\left(\prod_{j=1}^{2m} (X_j+\ii Y_j)\right) & = \sum_\pi  \sum_{{\substack{\alpha_i,\beta_i=0,1 \\ i=1,\ldots, m}}} \ii^{\sum_i \alpha_i+\sum_i \beta_i} \prod_{i=1}^m  \sigma_{a_{\pi,i},b_{\pi,i}}^{(\alpha_i,\beta_i)} \\
& = \sum_\pi \prod_{i=1}^m  \sum_{\alpha_i,\beta_i=0,1}\ii^{\alpha_i+\beta_i}  \sigma_{a_{\pi,i},b_{\pi,i}}^{(\alpha_i,\beta_i)} \\
& = \sum_\pi \prod_{i=1}^m  (1,\ii) C(a_{\pi,i}, b_{\pi,i}) (1,\ii)^\top,
\end{split}
\end{align}
where
\begin{align*} 
C(a,b) = \left(
\begin{array}{cc}
\E(X_aX_b) & \E(X_aY_b) \\
\E(X_bY_a) & \E(Y_aY_b)
\end{array}
\right).
\end{align*} 
Notice that $\E[Z_a Z_b] = (1, \ii)C(a,b)(1,\ii)^\top$ and thus the right-hand side of \eqref{e:C-Isserlis} equals
$$
\sum_\pi \prod_{i=1}^m \E(Z_{a_{\pi,i}}Z_{b_{\pi,i}}),
$$
which shows that the Isserlis formula for complex-valued r.v.'s is exactly the same as that for real-valued random variables.
\end{proof}

\begin{lemma}\label{le:extra Isserlis}
Let $\{X(t),t\in\mathbb{R}\}$ be a stationary Gaussian process in $\mathbb{C}$ with $C(t-s) = \mathbb{E}X(t)\overline{X(s)}$ and $\check C(t-s)=\mathbb{E}X(t)X(s)$. 
Consider $X(t_i),X(s_i)$, $i=1,\hdots,N+M$ for some $N,M\in \mathbb{Z}$, with
$N\ge 0$ and $M\ge 0$.  Denote by $\mathcal{P}_{N,M}$ the class of all pairings of the set
\[\{t_i,s_i|i=1,\hdots,N+M\}\] and by $\mathcal{P}_{N,M,-k}$ the class of all pairings of  \[\{t_i,s_i|i=1,\hdots,N+M\}\setminus \{t_k,s_k\}.\]
This means that $\tilde u\in \mathcal{P}_{N,M}$ iff \[\tilde u =\left\{\left\{\tau,\sigma\right\},\left\{\tau,\tilde\tau\right\},\left\{\sigma,\tilde\sigma\right\}| \tau\neq\tilde\tau\in\{t_i:i=1,\hdots,N+M\},\ \sigma\neq\tilde\sigma\in\{s_i:i=1,\hdots,N+M\}\right\}\]
and each symbol $t_i,s_i,\ i=1,\hdots,N+M$ can be used only once.  
Then,
\begin{align}\label{e:le:extra Isserlis}
\begin{split}
& \mathbb{E}\left[\prod_{n=1}^NX(t_n)\overline{X(s_n)}\cdot\prod_{m=1}^M \Big(X(t_{N+m})\overline{X(s_{N+m})}-C(t_{N+m}-s_{N+m})\Big)\right] \\
&= \sum_{\substack{\tilde u\in \mathcal{P}_{N,M}:\\ \cup_{m=1}^M \left\{\{t_{N+m},s_{N+m}\}\right\}\cap\tilde u=\emptyset}}\prod_{\{i,j\}\in\tilde u}\Big[C(\tau-\sigma)\mathbbm{1}_{\{i,j\}=\{\tau,\sigma\}} +\check C(\tau-\tilde \tau) \mathbbm{1}_{\{i,j\}=\{\tau,\tilde \tau\}}  \\
& \hspace{5cm} +\overline{\check C(\sigma-\tilde\sigma)}\mathbbm{1}_{\{i,j\}=\{\sigma,\tilde\sigma\}}\Big].
\end{split}
\end{align} 
\end{lemma}

This Isserlis-type result is used in the proof of Proposition \ref{prop: central moments} (see e.g.\ \eqref{e:r_ge_3_moment_bound}),
where the $k$th order moments of the spectral density estimators involve terms as in \eqref{e:le:extra Isserlis} where $N=0$.  The reason we formulate \eqref{e:le:extra Isserlis}
for general $N\ge 0$ is to facilitate the proof of this relation by the method of induction.

\begin{proof}[Proof of Lemma \ref{le:extra Isserlis}]
We will prove the desired equality by using induction on $N+M$. 
When $N+M=1,$ the equality holds trivially. 
For the basis of our induction, we use $N+M=2$. We look at the three different cases. 
\begin{itemize}
    \item[(a)] $N=2$. The equality trivially holds by the Isserlis' formula.
    \item[(b)] $N=M=1$. We have 
    \begin{align*}
         \mathbb{E}& \left[X(t_1)\overline{X(s_1)}\cdot(X(t_2)\overline{X(s_2)}-C(t_2-s_2))\right]\\ 
         & = \mathbb{E}\left[X(t_1)\overline{X(s_1)}X(t_2)\overline{X(s_2)}\right] - C(t_1-s_1)C(t_2-s_2)\\ 
         & = \mathbb{E}\left[X(t_1)\overline{X(s_1)}\right]\mathbb{E}\left[X(t_2)\overline{X(s_2)}\right]+\mathbb{E}\left[X(t_1)X(t_2)\right]\mathbb{E}\left[\overline{X(s_1)}\overline{X(s_2)}\right]\\ 
         & + \mathbb{E}\left[X(t_1)\overline{X(s_2)}\right]\mathbb{E}\left[\overline{X(s_1)}X(t_2)\right]- C(t_1-s_1)C(t_2-s_2)\\ 
         & = \check C(t_1-t_2)\overline{\check C(s_1-s_2)} +C(t_1-s_2)C(t_2-s_1)
    \end{align*}
    where Isserlis' formula was used in the second equality.
    \item[(c)] $M=2$. We have that 
    \begin{align} \label{e:M=2,N=0}
    \begin{split}
      &  \mathbb{E} \left[(X(t_1)\overline{X(s_1)}-C(t_1-s_1))\cdot (X(t_2)\overline{X(s_2)}-C(t_2-s_2))\right]\\& = \mathbb{E}\left[X(t_1)\overline{X(s_1)}X(t_2)\overline{X(s_2)}\right]-C(t_1-s_1)C(t_2-s_2)\\ & = \check C(t_1-t_2)\overline{\check C(s_1-s_2)}+C(t_1-s_2)C(t_2-s_1)
      \end{split}
    \end{align}
    similarly to case $(b)$.
\end{itemize}

    For the induction hypothesis, we assume that the desired equality holds when $N+M=r.$ 
    Let now $N+M=r+1.$ We discern two cases. 
    \begin{itemize}
    \item[(a)] $N=r+1$. The result follows directly by Isserlis' formula. 
    \item[(b)] $N<r+1$. The following holds
    \begin{align*}
         & \mathbb{E}\left[\prod_{n=1}^NX(t_n)\overline{X(s_n)}\cdot\prod_{m=1}^M  \Big(X(t_{N+m})\overline{X(s_{N+m})}-C(t_{N+m}-s_{N+m})\Big)\right]\\ &
          = \mathbb{E}\left[\prod_{n=1}^{N+1}X(t_n)\overline{X(s_n)}\prod_{m=2}^{M}  \Big(X(t_{N+m})\overline{X(s_{N+m})}-C(t_{N+m}-s_{N+m})\Big)\right]\\ 
         & \quad - C(t_{N+1}-s_{N+1})\mathbb{E}\left[\prod_{n=1}^NX(t_n)\overline{X(s_n)}\cdot \prod_{m=2}^{M}  \Big(X(t_{N+m})\overline{X(s_{N+m})}-C(t_{N+m}-s_{N+m})\Big)\right]\\ & = \hdots \\& = 
         \mathbb{E}\left[\prod_{n=1}^{N+M}X(t_n)\overline{X(s_n)}\right]\\ 
         & \quad - \sum_{m=1}^M C(t_{N+m}-s_{N+m})\mathbb{E}
         \Big[\prod_{n=1}^{N+m-1}X(t_n)\overline{X(s_n)}\\
         & \hskip6cm \cdot\prod_{k=m+1}^{M}\Big(X(t_{N+k})\overline{X(s_{N+k})}-C(t_{N+k}-s_{N+m})\Big)\Big]
    \end{align*}
Applying Isserlis' formula the first summand is equal to 
    \begin{align*}
    & \sum_{\tilde u\in \mathcal{P}_{N,M}}\prod_{\{i,j\}\in\tilde u}\Big[C(\tau-\sigma)\mathbbm{1}_{\{i,j\}=\{\tau,\sigma\}} +\check C(\tau-\tilde\tau)\mathbbm{1}_{\{i,j\}=\{\tau,\tilde\tau\}} \\
    & \hskip3cm +\overline{\check C(\sigma-\tilde\sigma)}\mathbbm{1}_{\{i,j\}=\{\sigma,\tilde\sigma\}}\Big].
    \end{align*}
By applying the induction hypothesis in the second summand, since all terms involve $r$ factors in total, we have that the second term is equal to     
\begin{align*}
& \sum_{m=1}^M C(t_{N+m}-s_{N+m})\sum_{\substack{\tilde u\in \mathcal{P}_{N,M,-m}:\\ \cup_{i=N+m+1}^{N+M} \left\{\{t_i,s_i\}\right\}\cap\tilde u=\emptyset}}\prod_{\{i,j\}\in\tilde u}\Big[C(\tau-\sigma)\mathbbm{1}_{\{i,j\}=\{\tau,\sigma\}} \\ 
& \hskip4cm +\check C(\tau-\tilde\tau)\mathbbm{1}_{\{i,j\}=\{\tau,\tilde\tau\}}
     +\overline{\check C(\sigma-\tilde\sigma)}\mathbbm{1}_{\{i,j\}=\{\sigma,\tilde\sigma\}}\Big].
\end{align*}
Denote the first term of the previous sum as $A$, the second term as $B$, write $B:=\sum_{m=1}^MB_m$. Also, let 
\[C_{i,j}^{\tilde u}:= C(\tau-\sigma)\mathbbm{1}_{\{i,j\}=\{\tau,\sigma\}}+\check C(\tau-\tilde\tau)\mathbbm{1}_{\{i,j\}=\{\tau,\tilde\tau\}}
     +\overline{\check C(\sigma-\tilde\sigma)}\mathbbm{1}_{\{i,j\}=\{\sigma,\tilde\sigma\}}.\]
     
Then we have that 
\begin{align*}
A-B_M & = \sum_{\tilde u\in\mathcal{P}_{N,M}}\prod_{\{i,j\}\in\tilde u}C_{i,j}^{\tilde u} - C(t_{N+M}-s_{N+M})\sum_{\tilde u'\in\mathcal{P}_{N,M,-M}}\prod_{\{i,j\}\in\tilde u'}C_{i,j}^{\tilde u'}\\ & = \sum_{\substack{\tilde u \in \mathcal{P}_{N,M}:\\ \{t_{N+M},s_{N+M}\}\not\in\tilde u}}\prod_{\{i,j\}\in\tilde u}C_{i,j}^{\tilde u}.
\end{align*}   
Similarly \[A-B_M-B_{M-1} = \sum_{\substack{\tilde u \in \mathcal{P}_{N,M}:\\ \cup_{i=M-1}^M\left\{\{t_{N+i},s_{N+i}\}\right\}\cap \tilde u=\emptyset}}\prod_{\{i,j\}\in\tilde u}C_{i,j}^{\tilde u}.\]
Continuing this way for all terms $B_j,\ j=1,\hdots,M$, we have that the proof is complete.

\end{itemize}
\end{proof}

\section{Proofs for Section \ref{s:rkhs}} \label{s:RKHS_proof}
As in Section \ref{s:rkhs}, $\bbH_n$ denotes the space spanned by $\{R(u, \cdot), u\in D_n\}$ and $\Pi_n$ is the projection operator onto $\bbH_n$.
Although the following result is standard, we include it here for the sake of completeness.

\begin{prop} \label{prop:rkhs}
Assume that the matrix ${\boldsymbol R}_n =\{R(u_{n,i},u_{n,j})\}_{i,j=1}^{m_n}$ is invertible. Let $g\in\bbH$ and ${\boldsymbol g} = (g(u_{n,1}),\ldots,g(u_{n,m_n}))^\top$. Then, the following hold.
\begin{enumerate} [label=(\roman*)]
\item
The projection $\tilde g = \Pi_n g = \sum_i c_i R(u_{n,i},\cdot)$ where $\boldsymbol{c} := (c_1,\ldots,c_{m_n})^\top = 
{\boldsymbol R}_n ^{-1} {\boldsymbol g}$,
and $\tilde g(u_{n,i}) = g(u_{n,i})$ for all $u_{n,i}\in D_n$. Moreover, $\left\|g-\tilde g\right\|_\bbH^2 = \|g\|_\bbH^2 - {\boldsymbol g} ^\top {\boldsymbol R}_n^{-1} {\boldsymbol g}$.
\item
$|\tilde g(u) - g(u)| \le \|g\|_\bbH \inf_{u'\in D_n} \sqrt{R(u,u)-2R(u,u')+R(u',u')}.$
\vskip.3cm
\end{enumerate} 
\end{prop} 

\begin{proof} \
\begin{enumerate}[label=(\roman*)]
\item
By the property of projection,
$$
\tilde g = \argmin_{h\in\bbH_n} \|g-h\|_\bbH.
$$
For $h = \sum_i c_i R(u_{n,i},\cdot)$, the reproducing property entails
$$
\|g-h\|_\bbH^2 = \|g\|_\bbH^2 - 2 \boldsymbol{c}^\top \boldsymbol{g} + \boldsymbol{c}^\top \boldsymbol{R}_n \boldsymbol{c},
$$
from which we conclude the minimizer $\boldsymbol{c}$ is ${\boldsymbol R}_n ^{-1} {\boldsymbol g}$. 
It then follows that
$$
\tilde{\boldsymbol{g}} := \left(\tilde g(u_{n,1}),\ldots,\tilde g(u_{n,m_n})\right)^\top = {\boldsymbol R}_n \boldsymbol{c}
= {\boldsymbol g}.
$$
\item
Applying again the fact that $g-\tilde g \perp R(u', \cdot)$ for all $u'\in D_n$, we have for any arbitrary $u\in E$,
$$
\tilde g(u) - g(u) = \langle \tilde g- g, R(u,\cdot)\rangle_{\bbH} = \langle \tilde g - g, R(u,\cdot)-R(u',\cdot)\rangle_{\bbH}, \ u' \in D_n.
$$
By (i) and the Cauchy-Schwarz inequality
$$
|\tilde g(u)-g(u)| \le \|g\|_\bbH \inf_{u'\in D_n} \|R(u,\cdot)-R(u',\cdot)\|_\bbH,
$$
where
$$
\|R(u,\cdot)-R(u',\cdot)\|_\bbH^2 = R(u,u)-2R(u,u')+R(u',u').
$$
\end{enumerate}
\end{proof}

{\bf Proof of Theorem \ref{thm:RKHS_bias}:} First,
\begin{align*}
\|\tilde f(\theta)-f(\theta)\|_{\rm HS} & = \|\Pi_n f(\theta)\Pi_n-f(\theta)\|_{\rm HS} \\
& \le \|\Pi_n f(\theta)\Pi_n-\Pi_n f(\theta)\|_{\rm HS} + \|\Pi_n f(\theta)- f(\theta)\|_{\rm HS} \\
&\le \|f(\theta)(\Pi_n-{\rm I})\|_{\rm HS} + \|(\Pi_n-{\rm I}) f(\theta)\|_{\rm HS} \\
&= 2 \|(\Pi_n-{\rm I}) f(\theta)\|_{\rm HS} \\
& = 2\left(\sum_{j=1}^\infty \nu_j^2 \|(\Pi_n-{\rm I}) \phi_j\|_\bbH^2\right)^{1/2}.
\end{align*}
Next, we consider $\|(\Pi_n-{\rm I})g\|_\bbH^2 = \|\tilde g-g\|_\bbH^2$ for a function $g\in\bbH$ with a Lipschitz continuous derivative. 
The derivation of this depends little on the value of $g(0)$. To simplify notation, let us make the simplification that the Sobolev space contains functions $g$ with $g(0)=0$. Thus, we take the kernel as $R(s,t) = s\wedge t$, i.e., the covariance kernel of the standard Brownian motion. 
Then the matrix $\boldsymbol{R}_n$ in \eqref{e:Rn-matrix} is indeed invertible. By Proposition \ref{prop:rkhs}, 
\begin{align} \label{e:g_tilde-g}
\|\tilde g-g\|_\bbH^2 = \|g\|_\bbH^2 - \boldsymbol{g}^\top \boldsymbol{R}_n^{-1} \boldsymbol{g}
\end{align}
where $\boldsymbol{g} = (g(u_{n,i}))_{i=1}^{m_n}$ contains the values of $g$ at the $u_{n,i}$. 
It follows that $\boldsymbol{R}_n$ has the Cholesky decomposition
\begin{equation}\label{e:Rn-via-Ln}
\boldsymbol{R}_n = m_n^{-1} \boldsymbol{L}_n\boldsymbol{L}_n^\top,
\end{equation}
where $\boldsymbol{L}_n$ is a lower triangular matrix of $1$'s and has inverse
$$
\boldsymbol{L}_n^{-1} = \begin{bmatrix}
1 & 0 & 0 & \cdots & 0 & 0 \\
-1 & 1 & 0 & \cdots & 0 & 0 \\
 0 & -1 & 1 & \cdots & 0 & 0 \\
 \vdots & \vdots & \vdots & \ddots & \vdots & \vdots\\
 0 & 0 & 0 & \cdots & 1 & 0\\
  0 & 0 & 0 & \cdots & -1 & 1
 \end{bmatrix}.
$$
Indeed, by the independence and the stationarity of the increments of the standard Brownian motion $B$, we have
$
\boldsymbol{Z} = \sqrt{m_n}( \boldsymbol{L}_n^{-1})^\top \boldsymbol{B},
$
where $\boldsymbol{B} = \Big( B(i/m_n) - B((i-1)/m_n) \Big)_{i=1}^{m_n}$  and $\boldsymbol{Z}
\sim {\cal N}(0, \boldsymbol{I}_{m_n})$ is a standard Normal 
random vector.  Since $\boldsymbol{R}_n=\E[\boldsymbol{B}\boldsymbol{B}^\top]$, we obtain
$\boldsymbol{I}_{m_n} = m_n (\boldsymbol{L}_n^{-1})^\top \boldsymbol{R}_n \boldsymbol{L}_n^{-1}$, which yields \eqref{e:Rn-via-Ln}.
Thus,
\begin{align} \label{e:g_approx}
\boldsymbol{g}^\top \boldsymbol{R}_n^{-1} \boldsymbol{g} = m_n \boldsymbol{g}^\top (\boldsymbol{L}_n^{-1})^\top \boldsymbol{L}_n^{-1} \boldsymbol{g} 
= m_n \sum_{i=1}^{m_n} (g(i/m_n)-g((i-1)/m_n))^2,
\end{align}
which is a Riemann approximation of $\|g\|_\bbH^2 = \int_0^1 (g'(t))^2 dt$ (recall $g(0)=0$).
 Since $|g'(s)-g'(t)|\le C |s-t|$, it follows from \eqref{e:g_tilde-g} and \eqref{e:g_approx} that
$$
\|\tilde g- g\|_\bbH^2 \le C m_n^{-1}.
$$
Indeed, by the mean value theorem, we have $g(i/m_n)-g((i-1)/m_n) = g'(\xi_{n,i}) m_n^{-1}$, for some $\xi_{n,i} \in [(i-1)/m_n, i/m_n]$, and hence
\begin{align*}
\|g - \wtilde g\|_{\bbH}^2 &= \int_0^1 (g'(t))^2 dt - \frac{1}{m_n} \sum_{i=1}^{m_n} (g'(\xi_{n,i}))^2 \\
&\le  \sum_{i=1}^{m_{n}} \int_{(i-1)/m_n}^{i/m_n} |g'(t) - g'(\xi_{n,i})|\cdot |g'(t) + g'(\xi_{n,i})| dt\\
&\le \frac{C}{m_n} \left(  \int_0^1 |g'(t)|dt + \frac{1}{m_n} \sum_{i=1}^{m_n} |g'(\xi_{n,i})| \right) = {\cal O}\left(m_n^{-1}\right),
\end{align*}
where in the last relation we used the fact that the Riemann sum converges to the integral $\int_0^1 |g'(t)|dt <\infty$,\ as $m_n\to\infty$. 
Applying this bound and by the assumption on the $\phi_j$, we obtain
\begin{align*}
\sum_{j=1}^\infty \nu_j^2 \|(\Pi_n-{\rm I}) \phi_j\|_\bbH^2
\le m_n^{-1} \sum_{j=1}^\infty C_j \nu_j^2. 
\end{align*}
This completes the proof.
\qed

\section{Some properties of the trace norm}\label{appendix:auxiliary results}
We collect some elementary facts of the trace norm in the following lemma.

\begin{lemma}\label{le:trace norm alt def}
Let ${\cal A}$ be a trace class operator on the Hilbert space $\bbH$. Then
\begin{enumerate}
\item[(i)] 
$\|{\cal A}\|_{\rm tr} = \sup_{{\cal W}: {\rm unitary}}\left|\langle {\cal A},{\cal W}\rangle_{\rm HS}\right|$;
\item[(ii)] $\sum_{i} |\langle {\cal A}f_i,g_i\rangle|\le \|{\cal A}\|_{\rm tr}$ for any CONSs $\{f_i\}$ and $\{g_i\}$;
\item[(iii)] $\sum_i\left|\left\langle {\cal A} e_i,e_i\right\rangle\right|\le \|{\cal A}\|_{\rm tr}$ for any CONS $\{e_i\}$.
\end{enumerate}
\end{lemma}
\begin{proof} (i)
Suppose $A$ has the SVD 
\begin{align} \label{e:svd_A}
{\cal A} = \sum_{j}\lambda_jv_j\otimes w_j,
\end{align}
where $\lambda_j \ge 0$ and $\{v_j\},\{w_j\}$ are CONS of $\bbH$. Then, we can write 
\[{\cal A}  = \left(\sum_j \lambda_jv_j\otimes v_j\right)\left(\sum_k v_k\otimes w_k\right)=: {\cal PU},\]
which is a polar decomposition of ${\cal A}$.
It follows that 
\[\|{\cal A}\|_{\rm tr} = {\rm trace}({\cal P})={\rm trace}({\cal AU}^{*})=\langle {\cal A}, {\cal U}\rangle_{\rm HS}.\]
Suppose ${\cal W}$ is unitary and has the SVD ${\cal W}=\sum_k a_k\otimes b_k$. Then 
\begin{align*}
\begin{split}
    \left|\langle{\cal A},{\cal W}\rangle_{\rm HS}\right| & = \left|\sum_k \langle {\cal A}b_k,{\cal W}b_k \rangle\right| \\ 
    & = \left|\sum_{j}\sum_k\lambda_j\langle (v_j\otimes w_j)b_k,a_k\rangle \right| \\ 
    & = \left|\sum_j\lambda_j\sum_k\langle v_j,a_k\rangle\langle b_k,w_j\rangle\right| \le \sum_j \lambda_j
    \end{split}
\end{align*}
by the Cauchy-Schwarz inequality.

(ii) By \eqref{e:svd_A},
\begin{align*}
\sum_{i} |\langle {\cal A}f_i,g_i\rangle| & = \sum_{i} |\sum_j \lambda_j\langle v_j, g_i\rangle \langle w_j, f_i\rangle| \\
& \le \sum_{j} \lambda_j \sum_i |\langle v_j, g_i\rangle \langle w_j, f_i\rangle|,
\end{align*}
and the result again follows from the Cauchy-Schwarz inequality.

(iii) This is a special case of (ii) with $f_i=g_i$. 

\end{proof}

\section{Examples}

In this section, we discuss several concrete examples that illustrate the breadth and scope of the conditions imposed in various results in the paper.  

\subsection{An example of the class ${\cal P}_D(\beta,L)$} \label{s:PLD}

We consider in this section with an example of a class of covariance structures, where the rate of consistency nearly matches the optimal rate of ${\cal P}_D(\beta,L).$ This class consists of regularly varying covariance structures, as follows. 
\begin{example}\label{ex:regularly varying power-law decay}
Consider $d=1$ and the scalar-valued case $\H = \C$. Let  
$$
C(k)=|k|^{-\beta-1}S(|h|),\beta>0,k\in\mathbb{Z}$$ where $S$ is a slowly varying function at infinity. It is not hard to see that the corresponding spectral densities $f\in {\cal P}_D(\beta+\epsilon,L)$ for any $\epsilon>0$, depending on the value of $L$. Also, assume that the kernel function is of the form $$K(h)=[1-|h|^{\lambda+1}]_+,h\in\mathbb{R}$$ for some $\lambda>0.$ We work in the discrete time setting, so we are using the estimator $\hat f_n(\theta)$.     

Thus, we have that  
\begin{align*}
    (2\pi) & [f(\theta)-\mathbb{E}\hat f_n(\theta)] = \sum_{|k|\ge\Delta_n}e^{\mathbbm{i}k\theta}C(k)+ \sum_{|k|<\Delta_n}e^{\mathbbm{i}k\theta}C(k)\left[1-K\left(\frac{k}{\Delta_n}\right)\right]
\end{align*}
Consider $\theta=0.$ Then, the previous expression is equal to  
\begin{align*}
    & 2\sum_{k\ge\Delta_n}k^{-\beta-1}S(k)+ 2\sum_{k<\Delta_n}k^{-\beta-1}S(k)\cdot\frac{k^{\lambda+1}}{\Delta_n^{\lambda+1}}
\end{align*}
Using the fact that for $p>-1,$ $$\int_{\alpha}^{x}t^pS(t)dt\sim(p+1)^{-1}x^{p+1}S(x),\ \text{as}\ x\to\infty $$ and for $p<-1,$ $$\int_{x}^{\infty}t^pS(t)dt\sim|p+1|^{-1}x^{p+1}S(x),\ \text{as}\ x\to\infty $$ 
we obtain that for $0<\beta<\min\{1,\lambda+1\}:$ 
\begin{align*}
    (2\pi)[f(0)-\mathbb{E}\hat f_n(0)]& \sim\frac{\Delta_n^{-\beta}S(\Delta_n)}{\beta} +\frac{\Delta_n^{-\beta}S(\Delta_n)}{\lambda+1-\beta} = \mathcal{O}\left(\Delta_n^{-\beta}\cdot S(\Delta_n)\right)
\end{align*}
Also, for $\beta>1,$ the same expression can be evaluated to be of the same order
$$\mathcal{O}\left(\Delta_n^{-\beta}\cdot S(\Delta_n) \right).$$
Compare this to the rate of Proposition \ref{prop:PL bias time series} for $0<\beta\neq1$.

We shift our interest now to the variance. At first, using $T$ and $\Delta$ in place of $T_n$ and $\Delta_n$ respectively, we have 
\begin{align*}                  
    \hat f_n(0)-& \mathbb{E}\hat f_n(0)  = \frac{1}{2\pi }\cdot\sum_{t=0}^T\sum_{s=0}^T K\left(\frac{t-s}{\Delta}\right)\frac{X(t)\cdot X(s)-C(t-s)}{|T-(t-s)|} \\ & = \frac{1}{2\pi}\cdot\sum_{v=0}^T\sum_{h=-\Delta}^{\Delta}K\left(\frac{h}{\Delta}\right)\frac{X(v+h)\cdot X(v)-C(h)}{|T-h|}\cdot\mathbbm{1}(0\le h+v\le T),
\end{align*}
by the change of variables $h=t-s,v=s$. Thus,

\begin{align*}
    (2\pi)^2\mathbb{E}\left|\hat f_n(0)-\mathbb{E}\hat f_n(0)\right|^2 & = \sum_{v=0}^T\sum_{w=0}^T\sum_{|h|\le\Delta}\sum_{|\tilde h|\le\Delta} K\left(\frac{h}{\Delta}\right)K\left(\frac{\tilde h}{\Delta}\right)\\ & \frac{\mathbb{E}\left\{\left[X(v+h)\cdot X(v)-C(h)\right]\left[X\left(w+\tilde h\right)\cdot X(w)-C\left(\tilde h\right)\right]\right\}}{|T-h|\left|T-\tilde h\right|}\\ & \cdot\mathbbm{1}\left(0\le h+v\le T\right)\cdot\mathbbm{1}\left(0\le \tilde h+w\le T\right)
\end{align*}

After using Isserlis' lemma for the expectation in the middle we finally obtain that 
\begin{align*}
    (2\pi)^2\mathbb{E}\left|\hat f_n(0)-\mathbb{E}\hat f_n(0)\right|^2 & = \frac{1}{T^2}\cdot\sum_{v=0}^T\sum_{w=0}^T\sum_{|h|\le\Delta}\sum_{|\tilde h|\le\Delta} K\left(\frac{h}{\Delta}\right)K\left(\frac{\tilde h}{\Delta}\right)\\ & \frac{C\left(v-w+h-\tilde h\right)C(v-w)+C(v-w+h)C\left(w-v+\tilde h\right)}{|T-h|\left|T-\tilde h\right|}\\ & \cdot\mathbbm{1}(0\le h+v\le T)\cdot\mathbbm{1}\left(0\le \tilde h+w\le T\right).
\end{align*}
We have already shown that both of these terms are absolutely of the order $\mathcal{O}\left(\frac{\Delta}{T}\right)$. Hence, showing asymptotic equivalence of just one of these integrals with a term of order $\Delta/T$ is enough to show that the variance as a whole is of the same order. We focus on the first summand and have with the change of variables $x=v-w,y=w,z=h-\tilde h,u=\tilde h,$ that 
\begin{align*}
    & \sum_{v=0}^T\sum_{w=0}^T\sum_{|h|\le\Delta}\sum_{|\tilde h|\le\Delta}  K\left(\frac{h}{\Delta}\right)K\left(\frac{\tilde h}{\Delta}\right) \frac{C\left(v-w+h-\tilde h\right)C(v-w)}{|T-h|\left|T-\tilde h\right|}\\ & \qquad\qquad\qquad\qquad\qquad \cdot\mathbbm{1}\left(0\le h+v\le T\right)\cdot\mathbbm{1}\left(0\le \tilde h+w\le T\right)\\ & = \sum_{x=-T}^T\sum_{z=-2\Delta}^{2\Delta}C(x+z)C(z)\cdot\sum_{u=-\Delta\vee(-\Delta-z)}^{\Delta\wedge(\Delta-z)}K\left(\frac{z+u}{\Delta}\right)K\left(\frac{u}{\Delta}\right)\\ & \qquad\qquad\cdot \frac{|T-x|}{|T-(z+u)||T-u|}\sum_{y=0\vee(-x)}^{T\wedge(T-x)}\mathbbm{1}(x+y+z+u\in[0,T])\mathbbm{1}(u+y\in[0,T]).
\end{align*}
Observing that \begin{align*}
    \frac{1}{T}\sum_{y=0\vee(-x)}^{T\wedge(T-x)}\mathbbm{1}(x+y+z+u\in[0,T])\mathbbm{1}(u+y\in[0,T]) &\le 1-\frac{|x|}{T}\\
    \frac{1}{2\Delta}\sum_{u=-\Delta\vee(-\Delta-z)}^{\Delta\wedge(\Delta-z)}K\left(\frac{z+u}{\Delta}\right)K\left(\frac{u}{\Delta}\right) & \le 1-\frac{|z|}{2\Delta}\\
    \frac{|T-x|}{|T-(z+u)||T-u|} & \le \frac{1}{|T|}
\end{align*}
and the fact that $-T\le x\le T$ and $-2\Delta\le z\le2\Delta,$ we see that the aforementioned terms are bounded by 1. Also, by Assumption \ref{a:cov-prime}, we have that $$\sum_{x=-T}^T\sum_{z=-2\Delta}^{2\Delta}C(x+z)C(z)dzdx<\infty.$$
Using the Dominated Convergence Theorem, we obtain that the quadruple summation, divided by $2\Delta/T$ converges to a constant. Thus, it is asymptotically equivalent to $\Delta/T$ as desired. 

This, for $\lambda+1>\beta>0$, leads to the consistency rate 
$$\mathcal{O}\left(\sqrt{\frac{\Delta_n}{|T_n|}}+\Delta_n^{-\beta}\cdot S(\Delta_n)\right).$$
Considering $0<\beta\neq1,$ we see that the optimal consistency rate in this case essentially matches the one in Theorem \ref{prop:PL bias time series}.

Observe that the regular variation only played a role in establishing asymptotic equivalence of the bias vanish rate. Indeed, for the rate of the variance, we only needed the integrability of the Covariance operator and the regular variation was not used. Also, recall that here the spectral density $f\in {\cal P}_D(\beta+\epsilon,L).$ So the rate we should be comparing to is $$|T_n|^{-\frac{\beta+\epsilon}{2(\beta+\epsilon)+1}}.$$ 
\end{example}

\subsection{Examples on Assumptions \ref{a:var} and \ref{a:var_d} }\label{appendix:examples}
We present some examples of non-trivial processes that satisfy the Assumptions \ref{a:var} and \ref{a:var_d}, so as to demonstrate that
the assumptions are not vacuous. We first consider an example that satisfies Assumption \ref{a:var}.

\begin{example}\label{ex:suppl-condition 1 verify}
Consider the process $$X(t)=Z(t)^2-1,$$
where $Z:=\{Z(t),t\in\mathbb{R}^d\}$ is a zero-mean, real-valued stationary Gaussian process with standard normal marginals. Denote the stationary covariance of $Z$ by $C_Z(\cdot)$ which we assume to satisfy $\int_{u\in\R^d}\sup_{\lambda\in B(0,2\delta)} |C_Z(\lambda+u)|du < \infty$, for some small enough $\delta>0$.  
This condition is quite mild and can be satisfied by covariances that are integrable and sufficiently smooth in the tail.
We verify that Assumption \ref{a:var} holds for $X$. We start by considering $\wt X(t)=Z(t)^2$. It follows that
\begin{itemize}
\abovedisplayskip=-\baselineskip
\belowdisplayskip=0pt
\abovedisplayshortskip=-\baselineskip
\belowdisplayshortskip=0pt
    \item[(a)] $\mathbb{E}[\wt X(t)]=1$
    \item[(b)] $\mathbb{E}[\wt X(t_1)\wt X(t_2)]=1+2C_Z(t_1-t_2)^2$
    \item[(c)] $\begin{aligned}[t]
        \mathbb{E}[\wt X(t_1)\wt X(t_2)\wt X(t_3)] & = 15a_2^2a_3^2+3a_2^2b_3^2+3a_2^2c_3^2+3b_2^2a_3^2+3b_2^2b_3^2+b_2^2c_3^2+6a_3b_3a_2b_2,
    \end{aligned}$
    \item[(d)] \begin{align*}
        \ \ \ \ \mathbb{E}[\wt X(t_1)\wt X(t_2) & \wt X(t_3)\wt X(t_4)]  = 105a_2^2a_3^2a_4^2+15a_2^2a_3^2b_4^2+15a_2^2a_3^2c_4^2+15a_2^2a_3^2d_4^2\\ 
        & + 15a_2^2b_3^2a_4^2+9a_2^2b_3^2b_4^2+3a_2^2b_3^2c_4^2+3a_2^2b_3^2d_4^2\\ 
        &+ 15a_2^2c_3^2a_4^2+3a_2^2c_3^2b_4^2+9a_2^2c_3^2c_4^2+3a_2^2c_3^2d_4^2\\
        & + 30a_2^2a_3b_3a_4b_4+30a_2^2a_3c_3a_4c_4+6a_2^2b_3c_3b_4c_4\\
        & + 15b_2^2a_3^2a_4^2+9b_2^2a_3^2b_4^2+3b_2^2a_3^2c_4^2+3b_2^2a_3^2d_4^2\\
        & + 9b_2^2b_3^2a_4^2+15b_2^2b_3^2b_4^2+3b_2^2b_3^2c_4^2+3b_2^2b_3^2d_4^2\\
        & + 3b_2^2c_3^2a_4^2+3b_2^2c_3^2b_4^2+3b_2^2c_3^2c_4^2+b_2^2c_3^2d_4^2\\ 
        & + 36b_2^2a_3b_3a_4b_4+12b_2^2a_3c_3a_4c_4+12b_2^2b_3c_3b_4c_4\\
        &+ 60a_2b_2a_3^2a_4b_4+36a_2b_2a_4b_4b_3^2+12a_2b_2c_3^2a_4b_4\\ &+ 60a_2b_2a_3b_3a_4^2 + 36a_2b_2a_3b_3b_4^2+12a_2b_2a_3b_3c_4^2\\
        & + 12a_2b_2a_3b_3d_4^2+24a_2b_2a_3b_3d_4^2+24a_2b_2a_3c_3b_4c_4+24a_2b_2b_3c_3a_4c_4,
    \end{align*}
\end{itemize}
where 
    \begin{align*}
        a_2 & = C_Z(t_1-t_2)\\
        a_3 & = C_Z(t_1-t_3)\\
        b_2 & = \sqrt{1-C_Z(t_1-t_2)^2}\\
        b_3 & = \frac{C_Z(t_2-t_3)-C_Z(t_1-t_2)\cdot C_Z(t_1-t_3)}{\sqrt{1-C_Z(t_1-t_2)^2}}\\
        c_3 & = \sqrt{1-C_Z(t_1-t_3)^2-\frac{\left[C_Z(t_2-t_3)-C_Z(t_1-t_2)\cdot C_Z(t_1-t_3)\right]^2}{1-C_Z(t_1-t_2)^2}}\\ 
        a_4 & = C_Z(t_1-t_4)\\
        b_4 & = \frac{C_Z(t_2-t_4)-C_Z(t_1-t_2)C_Z(t_1-t_4)}{\sqrt{1-C_Z(t_1-t_2)^2}}\\
        c_4 & = \frac{C_Z(t_3-t_4)-C_Z(t_1-t_3)C_Z(t_1-t_4)-b_3b_4}{c_3}\\
        d_4 & = \sqrt{1-a_4^2-b_4^2-c_4^2}.
    \end{align*}
After centering the process as $X(t)=\wt X(t)-1$ and simplifying the aforementioned expressions, we end up with the following moments: 
\begin{itemize}
\abovedisplayskip=-\baselineskip
\belowdisplayskip=0pt
\abovedisplayshortskip=-\baselineskip
\belowdisplayshortskip=0pt
    \item[(a)] $\mathbb{E}[X(t)]=0$
    \item[(b)] $\mathbb{E}[X(t_1)X(t_2)]=2C_Z(t_1-t_2)^2$
    \item[(c)] $\mathbb{E}[X(t_1)X(t_2)X(t_3)] = 8C_Z(t_1-t_2)C_Z(t_1-t_3)C_Z(t_2-t_3)$,
    \item[(d)] \begin{flalign*}
        \quad\ \mathbb{E}[X(t_1)X(t_2)& X(t_3)X(t_4)]  = 4C_Z(t_1-t_2)^2C_Z(t_3-t_4)^2 &\\ & +4C_Z(t_1-t_3)^2C_Z(t_2-t_4)^2+4C_Z(t_1-t_4)^2C_Z(t_2-t_3)^2&\\ 
        & +16C_Z(t_1-t_3)C_Z(t_1-t_4)C_Z(t_2-t_3)C_Z(t_2-t_4)&\\ & +16C_Z(t_1-t_2)C_Z(t_1-t_4)C_Z(t_2-t_3)C_Z(t_3-t_4)&\\ & +16C_Z(t_1-t_2)C_Z(t_1-t_3)C_Z(t_2-t_4)C_Z(t_3-t_4)&.
    \end{flalign*}
\end{itemize}
Using (d) above, we obtain that 
\[\mathbb{E}|X(t)|^4=60C_Z(0)^4=60<\infty,\]
showing that (a) of Condition \ref{a:var} holds. 
 
By the definition of the cumulants in Definition \ref{def:cumulants} we obtain that
\begin{align*}
    {\rm cum}(X(t_1),X(t_2),X(t_3),X(t_4)) & =  16C_Z(t_1-t_3)C_Z(t_1-t_4)C_Z(t_2-t_3)C_Z(t_2-t_4)\\ & +16C_Z(t_1-t_2)C_Z(t_1-t_4)C_Z(t_2-t_3)C_Z(t_3-t_4)\\ & +16C_Z(t_1-t_2)C_Z(t_1-t_3)C_Z(t_2-t_4)C_Z(t_3-t_4).
\end{align*}
Then, Assumption \ref{a:var}(b) is also satisfied, since \begin{align*}
    \sup_{w\in\mathbb{R}^d} & \int_{u\in\mathbb{R}^d}\int_{v\in\mathbb{R}^d}  \sup_{\lambda_1,\lambda_2,\lambda_3\in B(0,\delta)}|{\rm cum}(X(\lambda_1+u),X(\lambda_2+v),X(\lambda_3+w) ,X(0))|dvdu\\ & \leq 16\sup_{w\in\mathbb{R}^d}\int_{u\in\mathbb{R}^d}\int_{v\in\mathbb{R}^d} \sup_{\lambda_1,\lambda_2,\lambda_3\in B(0,\delta)} \\ & \quad\Big\{ |C_Z(\lambda_1-\lambda_3+u-w)C_Z(\lambda_1+u)C_Z(\lambda_2-\lambda_3+v-w)C_Z(\lambda_2+v)|\\
    & \quad + |C_Z(\lambda_1-\lambda_2+u-v)C_Z(\lambda_1+u)C_Z(\lambda_2-\lambda_3+v-w)C_Z(\lambda_3+w)|\\ & \quad + |C_Z(\lambda_1-\lambda_2+u-v)C_Z(\lambda_1-\lambda_3+u-w)C_Z(\lambda_2+v)C_Z(\lambda_3+w)|\Big\}dvdu
\\ & \leq 16|C_Z(0)|^2\int_{u\in\mathbb{R}^d}\int_{v\in\mathbb{R}^d} \sup_{\lambda_1,\lambda_2\in B(0,\delta)}\Big\{ |C_Z(\lambda_1+u)C_Z(\lambda_2+v)|\\
    & \qquad\qquad\qquad\qquad\qquad\quad + |C_Z(\lambda_1-\lambda_2+u-v)C_Z(\lambda_1+u)|\\ 
    & \qquad\qquad\qquad\qquad\qquad\quad + |C_Z(\lambda_1-\lambda_2+u-v)C_Z(\lambda_2+v)|\Big\}dvdu
    \\
    & \leq 48|C_Z(0)|^2\left(\int_{u\in\mathbb{R}^d}\sup_{\lambda\in B(0,2\delta)}|C_Z(\lambda+u)|du\right)^2<\infty.
\end{align*}

\qed
\end{example}

Next, we present an example inspired by the linear processes in Proposition 4.1 of \cite{panaretos2013fourier}. Assume that $\mathbb{H}$ is a separable (typically infinite-dimensional) Hilbert space. Let $\epsilon_t,\ t\in\mathbb{Z}$ be iid random elements of $\mathbb{H}$ such that $\mathbb{E}\|\epsilon_0\|^4<\infty$ and consider a sequence of bounded linear operators $A_s: \mathbb{H}\to \mathbb{H},\ s\in \mathbb{Z}$. Define
\begin{equation}\label{ex:X(t)}
X(t)=\sum_{s\in\mathbb{Z}}A_s\epsilon_{t-s},\ t\in\mathbb{Z}.
\end{equation}

In the following lemma, we show that the real process $X(t)$ is well defined under a mild square-summability condition on the operator norms of the coefficients.  To this end,
let ${\cal L}^2(\H)$ denote the Hilbert space of $\H$-valued random elements equipped with the inner product $\langle A,B\rangle_{\mathcal{L}^2}=\mathbb{E}\langle A,B\rangle$ 
for all $\H$-valued random elements $A$ and $B$ such that $\E[ \|A\|^2 + \|B\|^2]<\infty$.  The resulting norm in ${\cal L}^2(\H)$ will be denoted by $\|\cdot\|_{{\cal L}^2(\H)}$.

\begin{lemma}\label{lem:example}
Assume that the operator norms of $\{A_s,\ s\in\mathbb{Z}\}$ are square summable, namely that 
\begin{equation}
\sum_{s\in\mathbb{Z}}\|A_s\|_{\rm op}^2<\infty.
\end{equation}
Then, the series in \eqref{ex:X(t)} converges in $\|\cdot\|_{{\cal L}^2(\H)}$ and the process $\{X(t),\ t\in\mathbb{Z}\}$ defined.
\end{lemma}  
\begin{proof}
Let $\Sigma_{\epsilon}=\mathbb{E}[\epsilon_0\otimes\epsilon_0]$ be the covariance operator of every $\epsilon_t,\ t\in\mathbb{Z}$.
We start by defining 
\begin{equation}\label{ex:XN and X-N}
X^{(N)}(t)= \sum_{|s|\le N} A_s\epsilon_{t-s}\qquad{\rm and}\qquad X^{-(N)}(t)= \sum_{|s|> N} A_s\epsilon_{t-s}.
\end{equation}
We have that 
\begin{align*}
\mathbb{E}\langle X^{(N)}(t),X^{(N)}(t)\rangle &
= \sum_{|s_1|\le N}\sum_{|s_2|\le N}\mathbb{E}\langle A_{s_1}\epsilon_{t-s_1},A_{s_2}\epsilon_{t-s_2}\rangle\\ & = \sum_{|s|\le N}\mathbb{E}\langle A_s\epsilon_{t-s},A_s\epsilon_{t-s}\rangle = \sum_{|s|\le N}\mathbb{E}\langle \epsilon_{t-s},A_s^{\star}A_s\epsilon_{t-s}\rangle,
\end{align*} 
where the second equality follows by the independence of the $\epsilon$'s.
Let now $\{e_j\}$ be a CONS of $\mathbb{H}$ that diagonalizes $\Sigma_{\epsilon}$. Then, we can express the $\epsilon$'s as
\begin{equation*}
\epsilon_{t-s}=\sum_{j=1}^\infty Z_{t-s,j}e_j,
\end{equation*}
where $Z_{s,j}:= \langle \epsilon_s,e_j\rangle,$ are independent in $s$ because the $\epsilon_s$'s are iid. Also, because of the choice 
of $\{e_j\}$ as the eigenvectors of the covariance operator $\Sigma_\epsilon$, we have that for each fixed $s$, the $Z_{s,j}$'s are uncorrelated in $j$:
\[\mathbb{E} \left[Z_{s,i} \overline{Z_{s,j}}\right]=\lambda_i\cdot \delta_{i-j}.\]
Using those, we obtain 
\begin{align}\label{ex:square norm XN(t)}
\begin{split}
\mathbb{E}\left\langle X^{(N)}(t),X^{(N)}(t)\right\rangle & = \sum_{|s|\le N} \mathbb{E} \left\langle\sum_k e_k Z_{t-s,k},A_s^{\star}A_s\sum_{\ell}e_{\ell} Z_{t-s,{\ell}}\right\rangle\\
 &=\sum_{|s|\le N}\sum_{k}\sum_{\ell} \mathbb{E}\left[Z_{t-s,k} \overline{Z_{t-s,{\ell}}}\right] \cdot \langle e_k,A_s^{\star}A_se_{\ell}\rangle\\ 
 & = \sum_{|s|\le N}\sum_{k}\lambda_k \cdot \langle e_k,A_s^{\star}A_se_{k}\rangle\le \sum_k\lambda_k\sum_{|s|\le N}\langle e_k,A_s^{\star}A_se_{k}\rangle\\ 
 &\le {\rm tr}(\Sigma_{\epsilon})\cdot\sum_{|s|\le N}\|A_s^{\star}A_s\|_{\rm op}\le {\rm tr}(\Sigma_{\epsilon})\cdot \sum_{|s|\le N}\|A_s\|_{\rm op}^2<\infty.
 \end{split}
\end{align}   
With a similar argument to \eqref{ex:square norm XN(t)}, one has that for $M<N$ \begin{align*}
\mathbb{E}\|X^{(N)}(t)-X^{(M)}(t)\|^2\le {\rm tr}(\Sigma_{\epsilon})\sum_{M<|s|\le N}\|A_s\|_{\rm op}^2\to 0,
\end{align*}
as $N,M\to \infty$. This shows that the sequence $\{X^{(N)}(t) \}_{N\in\mathbb{N}}$ is a Cauchy sequence in the Hilbert space
 $\left(\mathcal{L}^2(\mathbb{H}),\langle\cdot,\cdot\rangle_{\mathcal{L}^2}\right)$, where $\langle A,B\rangle_{\mathcal{L}^2}=\mathbb{E}\langle A,B\rangle$ for $A,B$ random elements of $\mathbb{H}$.
Thus, the limit of this sequence exists and 
\[X^{(N)}(t)\to X(t)\in\mathcal{L}^2(\mathbb{H}),\]
which completes the proof. 
\end{proof}

\begin{prop} \label{ex:ar} Let $X(t)$ defined as in \eqref{ex:X(t)}. Assume that $\{A_s,\ s\in \mathbb{Z}\}$ are Hilbert-Schmidt operators with $\sum_{s\in\mathbb{Z}}\|A_s\|_{\rm HS}<\infty$. Moreover, letting 
$Z_{s,j} =\langle \epsilon_s,e_j\rangle$, where $\{e_j\}$ is a CONS diagonalizing $\Sigma_\epsilon:=\mathbb{E}[\epsilon_0\otimes\epsilon_0]$, assume that 
\begin{equation*}
\sum_{\ell_1,\ell_2,\ell_3,\ell_4}{\rm cum}( Z_{0,\ell_1},  Z_{0,\ell_2},  Z_{0,\ell_3}, Z_{0,\ell_4})^2\le  B<\infty.
\end{equation*}  
Then, the process $\{X(t), \ t\in\mathbb{Z}\}$ satisfies Assumption \ref{a:var_d}. 
\end{prop}

\begin{proof}
Recall that part $(a)$ of Assumption \ref{a:var_d} entails the finite fourth moment of $\|X(t)\|$. 

Let $X^{(N)}(t)$ and $X^{-(N)}(t)$ be as defined in \eqref{ex:XN and X-N}. 
Then, for every $k\in\mathbb{N}$ such that $\E \| \epsilon_t\|^k <\infty$, we have that 
\begin{align}\label{ex:Lp convergence of XNt}
\begin{split}
    \mathbb{E}\left\|X^{-(N)}(t)\right\|^k & \leq \sum_{|s_1|,\hdots,|s_k|>N}\|A_{s_1}\|_{\rm op}\hdots\|A_{s_k}\|_{\rm op}\mathbb{E}\left(\|\epsilon_{t-s_1}\|\hdots \|\epsilon_{t-s_k}\|\right)\\ & \leq \sum_{|s_1|,\hdots,|s_k|>N}\|A_{s_1}\|_{\rm op}\hdots\|A_{s_k}\|_{\rm op}\mathbb{E}\left(\|\epsilon_{t-s_1}\|^k\right)^{1/k}\hdots \mathbb{E}\left(\|\epsilon_{t-s_k}\|^k\right)^{1/k}\\ & =\mathbb{E}\|\epsilon_0\|^k\cdot \sum_{|s_1|,\hdots,|s_k|>N}\|A_{s_1}\|_{\rm op}\hdots\|A_{s_k}\|_{\rm op}\\ & =  \mathbb{E}\|\epsilon_0\|^k\cdot \left(\sum_{|s|>N}\|A_{s}\|_{\rm op}\right)^k\to0, \ \textrm{as}\ N\to\infty, 
    \end{split}
\end{align}
where the inequality in the second line follows from the generalized H\"older inequality \citep[cf. Theorem 11 of][]{hardy1952inequalities} and we used that $\|A_s\|_{\rm op}\le \|A_s\|_{\rm HS}$. 
Hence, we have ${\cal L}^k(\H)$-convergence of $X_t^{(N)}$ to $X_t,$ in the sense that  
$$
 \lim_{N\to\infty}\left(\mathbb{E}\left\Vert X(t)-X^{(N)}(t)\right\Vert_{\H}^k\right)^{1/k}=0.
$$  
These previous calculations also show directly that $\mathbb{E}\|X_t\|^k<\infty.$ Specifically, for $k=4$, part $(a)$ is proved. 

Now, for part $(b)$, as in the proof of Lemma \ref{lem:example}, letting $\{e_j\}$ be a CONS diagonalizing $\Sigma_\epsilon = \E[ \epsilon_0\otimes \epsilon_0]$, 
we write 
\begin{align*}
A_s = \sum_{i,j}a_{ij}(s) e_i\otimes e_j\ \ \mbox{ and }\ \ 
\epsilon_{t-s} = \sum_k Z_{t-s,k} e_k,
\end{align*} 
with $Z_{s,k}:= \langle \epsilon_s,e_k\rangle$.  Note that $\{e_i\otimes e_j\}$ is a CONS in the Hilbert space $\mathbb{X}$ of Hilbert-Schmidt operators on $\H$ equipped with $\langle \cdot, \cdot\rangle_{\rm HS}$
and the above expression for $A_s$ converges in $\|\cdot\|_{\rm HS}$. Let also 
\begin{align*}
A_{s,i\cdot}&:=\sum_j a_{ij}(s)e_i\otimes e_j\\
X_i(t) &:=\langle X(t), e_i\rangle = \sum_{s\in\mathbb{Z}}\sum_j a_{ij}(s)Z_{t-s,j}=\sum_{s\in\mathbb{Z}}A_{s,i\cdot}\epsilon_{t-s},
\end{align*}
so that $X(t) = \sum_{i} X_i(t) e_i$.  Recall the representation in Proposition \ref{prop:Cov for cross product} (see also 
\eqref{e:cum_basis} in the main paper).  For notational simplicity suppose that the process $X(t)$ is real relative to the CONS $\{e_i\}$, i.e., all the $X_i(t)$'s are 
real random variables.

We start by exploiting the multilinearity of the cumulants and the fact that $\epsilon_t$'s are iid. We have by Proposition \ref{prop:Cov for cross product} that 
${\rm cum}\left(X(u),X(v),X(w),X(0)\right)$ equals:
\begin{align}
    & \left|\sum_{i}\sum_{j} {\rm cum}\left(X_{i}(u),X_{j}(v),X_{i}(w),X_{j}(0)\right)\right|\nonumber\\
    & = \left|\sum_{i}\sum_{j} {\rm cum}\left(\sum_{s_1\in\mathbb{Z}}A_{s_1,i\cdot}\epsilon_{u-s_1},\sum_{s_2\in\mathbb{Z}}A_{s_2,j\cdot}\epsilon_{v-s_2},\sum_{s_3\in\mathbb{Z}}A_{s_3,i\cdot}\epsilon_{w-s_3},\sum_{s_4\in\mathbb{Z}}A_{s_4,j\cdot}\epsilon_{-s_4}\right)\right|\nonumber\\
    & = \left|\sum_{i}\sum_{j} {\rm cum}\left(\sum_{s_1\in\mathbb{Z}}A_{u-s_1,i\cdot}\epsilon_{s_1},\sum_{s_2\in\mathbb{Z}}A_{v-s_2,j\cdot}\epsilon_{s_2},\sum_{s_3\in\mathbb{Z}}A_{w-s_3,i\cdot}\epsilon_{s_3},\sum_{s_4\in\mathbb{Z}}A_{-s_4,j\cdot}\epsilon_{s_4}\right)\right|\nonumber
    \\
    & = \left|\sum_{i}\sum_{j} \sum_{s_1\in\mathbb{Z}}\sum_{s_2\in\mathbb{Z}}\sum_{s_3\in\mathbb{Z}}\sum_{s_4\in\mathbb{Z}}{\rm cum}\left(A_{u-s_1,i\cdot}\epsilon_{s_1},A_{v-s_2,j\cdot}\epsilon_{s_2},A_{w-s_3,i\cdot}\epsilon_{s_3},A_{-s_4,j\cdot}\epsilon_{s_4}\right)\right|\label{eq: F.3}\\ 
    & = \left|\sum_{i}\sum_{j} \sum_{s\in\mathbb{Z}}{\rm cum}\left(A_{u-s,i\cdot}\epsilon_{s},A_{v-s,j\cdot}\epsilon_{s},A_{w-s,i\cdot}\epsilon_{s},A_{-s,j\cdot}\epsilon_{s}\right)\right|\label{eq: epsilons are iid},
\end{align}
where \eqref{eq: epsilons are iid} follows from the fact that $\epsilon_t$'s are iid and \eqref{eq: F.3} will be justified in the end of this proof.

Continuing, \eqref{eq: epsilons are iid} is equal to 
\begin{align}     \label{e:summation expression}
\begin{split}
    &  \left|\sum_{i}\sum_{j} \sum_{s\in\mathbb{Z}}\sum_{\ell_1,\ell_2,\ell_3,\ell_4}a_{i\ell_1}(u-s)a_{j\ell_2}(v-s)a_{i\ell_3}(w-s)a_{j\ell_4}(-s){\rm cum}\left(Z_{s,\ell_1}, Z_{s,\ell_2}, Z_{s,\ell_3}, Z_{s,\ell_4}\right)\right|
    \\ & =\left|\sum_{i}\sum_{j} \sum_{s\in\mathbb{Z}}\sum_{\ell_1,\ell_2,\ell_3,\ell_4}a_{i\ell_1}(u-s)a_{j\ell_2}(v-s)a_{i\ell_3}(w-s)a_{j\ell_4}(-s){\rm cum}\left(Z_{0,\ell_1}, Z_{0,\ell_2}, Z_{0,\ell_3}, Z_{0,\ell_4}\right)\right|.
\end{split}
\end{align}

Changing the order of summation and applying the Cauchy-Schwarz inequality over $\sum_{\ell_1,\cdots,\ell_4}$,  we have that \eqref{e:summation expression} is bounded above by 
\begin{align*}
& \Big|\sum_{s}\sqrt{\sum_{\ell_1,\cdots,\ell_4}{\rm cum}(Z_{0,\ell_1}, Z_{0,\ell_2}, Z_{0,\ell_3}, Z_{0,\ell_4})^2}\\ 
& \qquad\qquad \cdot\sqrt{\sum_{\ell_1,\cdots,\ell_4} \Big( \sum_{i,j}  \Big[ a_{i\ell_1}(u-s)a_{j\ell_2}(v-s)a_{i\ell_3}(w-s)a_{j\ell_4}(-s) \Big] \Big)^2}\\
& \le  B\sum_{s\in\mathbb{Z}} \left( \sum_{\ell_1,\cdots,\ell_4}
  \Big( \sum_{i} a_{i\ell_1}(u-s)^2 \Big) \Big( \sum_{i} a_{i\ell_3}(w-s)^2 \Big) \Big( \sum_{j} a_{j\ell_2}(v-s)^2 \Big) \Big( \sum_{j}a_{j\ell_4}(-s)^2 \Big)   \right)^{1/2},\\
&=  B\sum_{s\in\mathbb{Z}}\|A_{u-s}\|_{\rm HS}\|A_{v-s}\|_{\rm HS}\|A_{w-s}\|_{\rm HS}\|A_{-s}\|_{\rm HS},
\end{align*}
where the above inequality follows by applying the Cauchy-Schwarz inequality twice -- once over $\sum_i$ and once over $\sum_j$.
The last relation follows from the fact that $\|A_t\|_{\rm HS}^2 = \sum_{\ell,i} a_{i\ell}(t)^2$.

Thus, we finally obtain:
\begin{align*}
& \sup_{w\in\mathbb{Z}}\sum_{u\in\mathbb{Z}}\sum_{v\in\mathbb{Z}}\left|\sum_{i}\sum_{j} {\rm cum}\left(X_{i}(u),X_{j}(v),X_{i}(w),X_{j}(0)\right)\right|\\ & \qquad\qquad \le \sup_{w\in\mathbb{Z}}\sum_{u\in\mathbb{Z}}\sum_{v\in\mathbb{Z}} B\sum_{s\in\mathbb{Z}}\|A_{u-s}\|_{\rm HS}\|A_{v-s}\|_{\rm HS}\|A_{w-s}\|_{\rm HS}\|A_{-s}\|_{\rm HS} \\ 
& \qquad\qquad \le B\sup_{w\in\mathbb{Z}}\|A_{w}\|_{\rm HS}\cdot \left(\sum_{s\in\mathbb{Z}}\|A_{u-s}\|_{\rm HS}\right)^3<\infty.
\end{align*}

Now, it only remains to justify the equality \eqref{eq: F.3}. We will use $X^{(N)}(t)$ and $X^{-(N)}(t)$ again. The calculations in \eqref{ex:Lp convergence of XNt} imply again by the generalized H\"older inequality and the Dominated Convergence Theorem that $$\mathbb{E}\left[\lim_{N\to\infty}X_{i}^{(N)}(u)X_{j}^{(N)}(v)X_{i}^{(N)}(w)X_{j}^{(N)}(0)\right]= \lim_{N\to\infty}\mathbb{E}\left[X_{i}^{(N)}(u)X_{j}^{(N)}(v)X_{i}^{(N)}(w)X_{j}^{(N)}(0)\right].$$

To this end, we introduce some notation.  For each pair $m = (m_1,m_2) \in \{(u,i),(v,j),(w,i),(0,j)\}$, we write $X_m^N$ for $X_{m_2}^N(m_1)$. 
For example, for $m=(u,i)$ we have that $X_m^{(N)}=X_i^{(N)}(u).$
Thus, using the definition of cumulants, we obtain
\begin{align*}
     &{\rm cum}\left(X_{i}(u),X_{j}(v),X_{i}(w),X_{j}(0)\right) \\ & = \sum_{\nu=(\nu_1,\hdots,\nu_q)}(-1)^{q-1}(q-1)!\prod_{l=1}^q\mathbb{E}\left[\prod_{m\in\nu_l}\lim_{N\to\infty}X_m^{(N)}\right]\\ & = \sum_{\nu=(\nu_1,\hdots,\nu_q)}(-1)^{q-1}(q-1)!\prod_{l=1}^q\mathbb{E}\left[\lim_{N\to\infty}\prod_{m\in\nu_l}X_m^{(N)}\right]\\
     & =\lim_{N\to\infty} \sum_{\nu=(\nu_1,\hdots,\nu_q)}(-1)^{q-1}(q-1)!\prod_{l=1}^q\mathbb{E}\left[\prod_{m\in\nu_l}X_m^{(N)}\right]\\ & = \lim_{N\to\infty} {\rm cum}\Bigg(\sum_{|s_1|\leq|N|}A_{s_1,i\cdot}\epsilon_{u-s_1},\sum_{|s_2|\leq|N|}A_{s_2,j\cdot}\epsilon_{v-s_2},\\ & \qquad\qquad\qquad\qquad\qquad\qquad\qquad\qquad\sum_{|s_3|\leq|N|}A_{s_3,i\cdot}\epsilon_{w-s_3},\sum_{|s_4|\leq|N|}A_{s_4,i\cdot}\epsilon_{-s_4}\Bigg)\\ & =\lim_{N\to\infty} \sum_{|s_1|,|s_2|,|s_3|,|s_4|\leq|N|}{\rm cum}\left(A_{s_1,i\cdot}\epsilon_{u-s_1},A_{s_2,j\cdot}\epsilon_{v-s_2},A_{s_3,i\cdot}\epsilon_{w-s_3},A_{s_4,i\cdot}\epsilon_{-s_4}\right) \\& = \sum_{|s_1|,|s_2|,|s_3|,|s_4|\in\mathbb{Z}}{\rm cum}\left(A_{s_1,i\cdot}\epsilon_{u-s_1},A_{s_2,j\cdot}\epsilon_{v-s_2},A_{s_3,i\cdot}\epsilon_{w-s_3},A_{s_4,i\cdot}\epsilon_{-s_4}\right),
\end{align*}
where the sum is over all unordered partitions of $\{(u,i),(v,j),(w,i),(0,j)\}$. The proof is complete.

\end{proof}

\end{document}